\documentclass[11pt,reqno]{amsart}

\usepackage[left=1in,right=1in,top=1in,bottom=0.8in]{geometry}
\usepackage[T1]{fontenc}
\usepackage[utf8]{inputenc}
\usepackage[tt=false]{libertine}
\usepackage{microtype}
\usepackage{amsmath,amssymb,mathrsfs,mathtools}
\usepackage{graphicx}
\usepackage{aliascnt}
\usepackage{enumitem}
\usepackage{xcolor}
\usepackage[colorlinks=true]{hyperref}
\usepackage{cleveref}

\definecolor{arxivblue}{RGB}{0,55,110}
\hypersetup{
  citecolor=violet,
  linkcolor=violet,
  urlcolor=arxivblue,
  pdftitle={Large Deviations for the Two-Dimensional Dean--Kawasaki Equation with Coulomb Interactions},
  pdfauthor={Xiaohao Ji},
  pdfsubject={A large deviation principle for a mass-subcritical two-dimensional exact-square-root Dean--Kawasaki equation with Coulomb interactions},
  pdfkeywords={Dean--Kawasaki equation, Coulomb interaction, Keller--Segel equation, Biot--Savart equation, stochastic kinetic solution, large deviations}
}

\numberwithin{equation}{section}
\allowdisplaybreaks[2]

\newtheorem{theorem}{Theorem}[section]
\newaliascnt{proposition}{theorem}
\newtheorem{proposition}[proposition]{Proposition}
\aliascntresetthe{proposition}
\newaliascnt{lemma}{theorem}
\newtheorem{lemma}[lemma]{Lemma}
\aliascntresetthe{lemma}
\newaliascnt{corollary}{theorem}
\newtheorem{corollary}[corollary]{Corollary}
\aliascntresetthe{corollary}
\theoremstyle{definition}
\newaliascnt{definition}{theorem}
\newtheorem{definition}[definition]{Definition}
\aliascntresetthe{definition}
\theoremstyle{remark}
\newaliascnt{remark}{theorem}
\newtheorem{remark}[remark]{Remark}
\aliascntresetthe{remark}
\newenvironment{nouppercase}{\renewcommand{\uppercasenonmath}[1]{}}{}

\crefname{theorem}{Theorem}{Theorems}
\Crefname{theorem}{Theorem}{Theorems}
\crefname{proposition}{Proposition}{Propositions}
\Crefname{proposition}{Proposition}{Propositions}
\crefname{lemma}{Lemma}{Lemmas}
\Crefname{lemma}{Lemma}{Lemmas}
\crefname{corollary}{Corollary}{Corollaries}
\Crefname{corollary}{Corollary}{Corollaries}
\crefname{definition}{Definition}{Definitions}
\Crefname{definition}{Definition}{Definitions}
\crefname{remark}{Remark}{Remarks}
\Crefname{remark}{Remark}{Remarks}

\newcommand{\T}{\mathbb T}
\newcommand{\R}{\mathbb R}
\newcommand{\E}{\mathbb E}
\newcommand{\Pp}{\mathbb P}
\newcommand{\cG}{\mathcal G}
\newcommand{\Green}{G}
\newcommand{\cX}{\mathcal X}
\newcommand{\J}{J}
\newcommand{\dd}{\,\mathrm d}
\newcommand{\weakto}{\rightharpoonup}
\newcommand{\QT}{Q_T}
\DeclareMathOperator{\diver}{div}

\title[Large Deviations for the Two-Dimensional Dean--Kawasaki Equation]
{\LARGE Large Deviations for the Two-Dimensional\\
Dean--Kawasaki Equation with Coulomb Interactions}

\author{Xiaohao Ji}
\date{\textnormal{\today\qquad
Email: \href{mailto:jixh1020@gmail.com}{\texttt{jixh1020@gmail.com}}}}

\begin{document}


\begin{nouppercase}
\maketitle
\end{nouppercase}
\vspace{-2em}
\begin{center}
\small Institut f\"ur Mathematik, Freie Universit\"at Berlin,
14195 Berlin, Germany
\end{center}
\vspace{-0.5em}

\begin{abstract}
We establish global well-posedness and a small-noise large deviation
principle for the two-dimensional Dean--Kawasaki equation with Coulomb
interaction
\[
 \partial_t\rho^\varepsilon
 =\Delta\rho^\varepsilon
 -\nabla\!\cdot\!\bigl(\rho^\varepsilon(V*\rho^\varepsilon)\bigr)
 -\sqrt\varepsilon\,
 \nabla\!\cdot\!\bigl(
 \sqrt{\rho^\varepsilon}\circ\xi^{K(\varepsilon)}
 \bigr),
\]
where
\[
 V
 =\lambda_V\bigl(
 \cos(\alpha_V)\nabla G+\sin(\alpha_V)J\nabla G
 \bigr).
\]
Here $\lambda_V\geq0$ is the interaction strength, $\alpha_V$ is the interaction
angle, $G$ is the mean-zero Green function of $-\Delta$ on $\T^2$, $J$ is
rotation by $\pi/2$, and $\xi^K$ is a finite-mode approximation of
space--time white noise.  For finite-entropy initial data of mass $M$
satisfying 
\[
\lambda_V\max\bigl\{\cos(\alpha_V),0\bigr\}M<8\pi, 
\]
the equation with the exact square-root coefficient is globally pathwise
well posed at every finite Fourier cutoff in
the class of stochastic kinetic solutions.  Following Fehrman and Gess
\cite{FG23,FG24}, we prove a large deviation principle on
$L^1((0,T)\times\T^2)$ under the scaling $K(\varepsilon)\to\infty$ and
$\varepsilon K(\varepsilon)^4\to0$, with good rate function given by the
quadratic control problem for the skeleton equation.
\end{abstract}

\setcounter{tocdepth}{1}
\tableofcontents
\enlargethispage{5pt}

\section{Introduction and main results}

The Dean--Kawasaki equation is a fluctuating hydrodynamic model for
empirical particle densities \cite{D96,K98}.  We consider the
two-dimensional equation with Coulomb interactions.  We prove global
well-posedness throughout the mass-subcritical regime and a large deviation
principle as the noise intensity and spatial correlation scale vanish
jointly.  The conservative square-root noise preserves mass and reflects
the density-dependent mobility of the underlying particle system.

We consider a general isotropic Coulomb interaction in two dimensions,
decomposed into a radial Keller--Segel component and a rotational
Biot--Savart component.
More precisely, on the unit flat torus $\T^2$, endowed with Lebesgue
measure of total mass one, the equation is
\begin{equation}\label{eq:dk-exact}
 \dd\rho^{\varepsilon}
 =\Big[\Delta\rho^{\varepsilon}
 -\diver\bigl(\rho^{\varepsilon}(V*\rho^{\varepsilon})\bigr)
 \Big]\dd t
 -\sqrt\varepsilon\,
 \diver\big(\sqrt{\rho^{\varepsilon}}\circ\dd W_{K(\varepsilon)}\big).
\end{equation}
Here $W_K$ is the real Fourier truncation of a cylindrical Wiener process
to frequencies at most $K$.  Well-posedness is proved for each fixed
$K$; the dependence $K=K(\varepsilon)$ is needed only for the large
deviation limit.  Writing $\Green$ for the mean-zero Green function of
$-\Delta$ on $\T^2$ and $\J$ for rotation through $\pi/2$, we use the
periodic kernel
\begin{equation}\label{eq:coulomb-decomposition}
 V
 =\lambda_V\cos(\alpha_V)\nabla\Green
 +\lambda_V\sin(\alpha_V)\J\nabla\Green,
\end{equation}
where $\lambda_V\geq0$ is the interaction strength and
$\alpha_V\in\R/(2\pi\mathbb Z)$ is the interaction angle.  This decomposition
describes the rotation-covariant Coulomb kernels in the plane; arbitrary
anisotropic kernels of the same singular order are not included.
Thus
$\lambda_V\cos(\alpha_V)>0$ is attractive,
$\lambda_V\cos(\alpha_V)\leq0$ is repulsive or has no radial component,
$\cos(\alpha_V)=0$ is pure Biot--Savart, and
$\sin(\alpha_V)=0$ is pure Keller--Segel.  The Green-function and noise
normalizations are specified in Subsection~\ref{subsec:organization-notation}.
The Stratonovich notation in \eqref{eq:dk-exact} is interpreted through
Definition~\ref{def:stochastic-kinetic-solution}.

For a nonnegative initial density $\rho_0$, write
$H(r)=r\log r-r+1$ for $r\geq0$, with $H(0)=1$, and
$M=\int_{\T^2}\rho_0$.  The Coulomb singularity has the mass-critical
scaling of the corresponding planar equation.  We assume the subcritical
mass condition
\begin{equation}\label{eq:subcritical}
 \lambda_V\max\{\cos(\alpha_V),0\}M<8\pi.
\end{equation}
This imposes no mass restriction when $\cos(\alpha_V)\leq0$.
The critical case is not covered by our results.
Define the free energy and its physical dissipation by
\begin{equation}\label{eq:free-energy-dissipation}
 \begin{aligned}
 \mathcal F_V(\rho)
 &=\int_{\T^2}H(\rho)
 -\frac{\lambda_V\cos(\alpha_V)}2
 \int_{\T^2}\rho(\Green*\rho),\\
 \mathcal D_V(\rho)
 &=\left\|
 2\nabla\sqrt\rho
 -\lambda_V\cos(\alpha_V)\sqrt\rho\,
  (\nabla\Green*\rho)
 \right\|_{L^2(\T^2)}^2.
 \end{aligned}
\end{equation}
Here \(\mathcal D_V(\rho)=+\infty\) unless
\(\sqrt\rho\in H^1(\T^2)\) and the displayed weighted flux belongs to
\(L^2(\T^2;\R^2)\).
The rotational component is absent from
\eqref{eq:free-energy-dissipation} because it is orthogonal to the radial
potential gradient and divergence free.  For a nonnegative density $\rho$,
we define its Fisher information by
\begin{equation}\label{eq:fisher-information}
 \mathcal I(\rho)=4\int_{\T^2}|\nabla\sqrt\rho|^2,
 \qquad \sqrt\rho\in H^1(\T^2),
\end{equation}
with $\mathcal I(\rho)=+\infty$ otherwise.

\subsection{Main results}

\begin{theorem}[Mass-subcritical stochastic well-posedness]
\label{thm:main-spde}
Let $\rho_0\in L^1(\T^2)$ be deterministic and nonnegative, with
$\int_{\T^2}H(\rho_0)<\infty$, and assume \eqref{eq:subcritical}.
\begin{enumerate}[label=\textup{(\roman*)},leftmargin=2.2em]
\item For every $\varepsilon>0$ and $K\in\mathbb N$, equation
\eqref{eq:dk-exact} is globally pathwise well posed in the class of stochastic
kinetic solutions.  The solution is nonnegative,
conserves mass, attains the initial datum, and
satisfies
\begin{equation*}
 \sup_{t\leq T}\int_{\T^2}H(\rho(t))
 +\int_0^T\mathcal D_V(\rho(t))\,\dd t
 +\int_0^T\mathcal I(\rho(t))\,\dd t<\infty
 \qquad\text{almost surely}.
\end{equation*}
\item For every predictable control $g$ satisfying, for some deterministic
$N<\infty$,
\begin{equation*}
 \int_0^T\!\int_{\T^2}|g|^2\leq N
 \qquad\text{almost surely},
\end{equation*}
the equation obtained by adding
$-\diver(\sqrt\rho\,P_Kg)$ is globally well posed in the same class, with
the same pathwise bounds.  For fixed $\rho_0$, a Borel
function of the finite-dimensional driving path represents the uncontrolled
solution, and its Cameron--Martin shifts represent the solutions of the
controlled equations.
\item If $K=K(\varepsilon)\to\infty$ and
$\varepsilon K(\varepsilon)^4\to0$, the entropy maximum, physical
dissipation, and Fisher information in \textup{(i)} are tight in probability,
uniformly over controls with a fixed almost-sure $L^2$ bound.
\end{enumerate}
\end{theorem}

The solution class in Definition~\ref{def:stochastic-kinetic-solution}
uses a global kinetic identity and a high-velocity condition localized by
physical energy; its relation to \cite[Definition~3.4]{FG24} is explained
in Remark~\ref{rem:comparison-fg-definition}.  No function-valued limit at
fixed $\varepsilon$ as $K\to\infty$ is asserted.

We prove the large deviation principle by the weak-convergence method
\cite{BDM08,DE}, following Fehrman and Gess
\cite[Sections~6.2--6.3]{FG23}.  The Brownian variational representation
\cite{MP} introduces predictable controls $g$ with quadratic cost
$\frac12\|g\|_{L^2((0,T)\times\T^2)}^2$.  A Cameron--Martin shift adds
the drift $-\diver(\sqrt\rho\,P_Kg)$.  As the noise vanishes and
the cutoff is removed, the limiting controlled equation is the
deterministic skeleton
\begin{equation}\label{eq:skeleton-intro}
 \partial_t\rho
 =\Delta\rho-\diver\bigl(\rho(V*\rho)\bigr)
 -\diver(\sqrt\rho\,g),
 \qquad \rho|_{t=0}=\rho_0.
\end{equation}
The next proposition provides the well-posedness and compactness needed
to identify these limits.  We denote the solution associated with $g$ by
$\rho[g]$.

\begin{proposition}[Skeleton well-posedness and stability]
\label{prop:main-skeleton}
Let $\rho_0\in L^1(\T^2)$ be nonnegative with finite entropy, and assume
\eqref{eq:subcritical}.  For every
$g\in L^2((0,T)\times\T^2;\R^2)$,
equation \eqref{eq:skeleton-intro} has a unique global
entropy solution $\rho[g]$ in the sense of
Definition~\ref{def:entropy-skeleton}.  Equivalently, it is the unique
renormalized kinetic solution with the exact parabolic defect in this
class; see Remark~\ref{rem:weak-kinetic-equivalence}.  For the fixed
initial datum $\rho_0$,
\begin{equation*}
 \begin{aligned}
 g_n\weakto g\quad\text{in }L^2((0,T)\times\T^2;\R^2)
 \quad\Longrightarrow\quad
 \rho[g_n]\longrightarrow\rho[g]
 \quad\text{in }L^1((0,T)\times\T^2).
 \end{aligned}
\end{equation*}
Consequently, for every $N<\infty$,
\begin{equation}\label{eq:main-skeleton-compactness}
 \left\{\rho[g]:
 \|g\|_{L^2((0,T)\times\T^2)}^2\leq N\right\}
 \quad\text{is compact in }L^1((0,T)\times\T^2).
\end{equation}
\end{proposition}

The compactness in \eqref{eq:main-skeleton-compactness} verifies one
of the two conditions in the weak-convergence method.  The other is
convergence of the controlled stochastic solutions to the skeleton
when the controls converge in distribution in the weak $L^2$ topology.
Theorem~\ref{thm:controlled-convergence} proves this under
the scaling
\begin{equation}\label{eq:ldp-scaling-intro}
 K(\varepsilon)\longrightarrow\infty,
 \qquad
 \varepsilon K(\varepsilon)^4\longrightarrow0.
\end{equation}
The first limit recovers the unprojected control; the second removes the
finite-mode It\^o corrections.  Together, skeleton compactness and
controlled convergence yield the following large deviation principle by
\cite[Theorem~6]{BDM08}, as in \cite[Theorems~6.6 and~6.8]{FG23}.

\begin{theorem}[Large deviation principle]\label{thm:ldp-intro}
Let $\rho_0\geq0$ be deterministic with finite entropy and mass $M$.
Assume \eqref{eq:subcritical} and \eqref{eq:ldp-scaling-intro}.  The laws of
$\rho^{\varepsilon}$ satisfy a large deviation principle on
$L^1((0,T)\times\T^2)$ with speed $\varepsilon^{-1}$ and good rate function
\begin{equation}\label{eq:rate-intro}
 I_{\rho_0}(\rho)
 =\frac12\inf\left\{
 \|g\|_{L^2((0,T)\times\T^2)}^2:\rho=\rho[g]
 \right\},
\end{equation}
where the infimum of the empty set is $+\infty$.  The large deviation
principle is understood in the usual open- and closed-set sense.  The rate
function also admits the dual dynamical representation of
Proposition~\ref{prop:dual-rate}.
\end{theorem}

\subsection{Related Literature}

Dean \cite{D96} derived an evolution equation for the empirical density
of interacting diffusions, while Kawasaki \cite{K98} developed a
dynamic density-functional description of fluctuating densities.  In
the context of chemotaxis, Chavanis \cite{Cha10} introduced a stochastic
Keller--Segel model accounting for fluctuations due to a finite number
of particles.  For a review of these formulations and their connections
with stochastic density functional theory and fluctuating hydrodynamics,
we refer to Illien \cite{Ill25}.

The mathematical analysis depends on the spatial structure of the noise.
For spatially white conservative noise, Konarovskyi, Lehmann and von
Renesse \cite{KLvR19,KLvR20} characterized the admissible particle-valued
solutions of the Dean--Kawasaki martingale problem.  Regularized models
were derived and analyzed by Cornalba, Shardlow and Zimmer
\cite{CSZ19,CSZ20}.  A related question is how accurately stochastic
density equations approximate particle-density fluctuations.
Quantitative results were obtained by Cornalba and Fischer
\cite{CF23arma} and, for weakly interacting particles, by Cornalba,
Fischer, Ingmanns and Raithel \cite{CFIR26}.  Weak-error estimates for
nonlinear SPDE approximations and smooth mean-field interactions
\cite{DKP24,DJP25} provide further results in this direction.
Damnjanovi\'c, Djurdjevac and Perkowski \cite{DDP26} studied how the
failure of positivity preservation affects weak approximation by a
spectral regularization of the Dean--Kawasaki equation, and supported
their error analysis with numerical experiments.

Our analysis builds on the kinetic theory developed by Fehrman and Gess
for nonlinear diffusion with conservative noise \cite{FG19}.  Their
work \cite{FG24} treats the
exact square-root coefficient with spatially correlated noise, including
kinetic well-posedness and common-noise comparison.  In \cite{FG23}, they
established well-posedness and weak-to-strong stability of the controlled
skeleton, and used these properties to prove small-noise large deviations.
For nonlocal interactions, \cite{WWZ22,WZ24} obtain well-posedness and
large deviations under Ladyzhenskaya--Prodi--Serrin conditions on the
kernel and additional integrability assumptions on its divergence.
Nonzero two-dimensional Coulomb kernels lie in
$L^{2,\infty}\setminus L^2$; when the radial component is present, their
divergence also contains a Dirac mass.  They therefore fall outside those
assumptions.

The approximation and compactness argument of Sun, Wu and the author
\cite[Theorem~1.2(2) and Section~3]{JSW26} provides a probabilistically
weak construction for Coulomb interactions.  In the version cited here,
\cite[Theorem~1.3]{JSW26} proves a large deviation upper bound and a lower
bound on regular trajectories for a family with a regularized square-root
coefficient.  The present paper establishes the a priori estimates and
comparison needed for pathwise well-posedness and a full large deviation
principle with the exact square-root coefficient throughout the
mass-subcritical regime.

For the planar parabolic--elliptic Keller--Segel equation, the competition
between entropy and logarithmic attraction leads to the critical mass
$8\pi$; see \cite{BDP06,DP04,JL92,Wei18}.  Its variational
structure is closely related to the sharp logarithmic
Hardy--Littlewood--Sobolev inequality of Carlen and Loss \cite{CL92}; on
the torus, we use the compact-manifold exponential inequality of Fontana
\cite{Fon93}.  The critical-mass dynamics were studied in \cite{BCM08}.
Uniqueness in a free-energy class was established by Ega\~na Fern\'andez
and Mischler \cite{EFM16}, while Bedrossian and Masmoudi \cite{BM14}
developed local well-posedness and Lipschitz dependence for
Keller--Segel and two-dimensional Navier--Stokes with measure-valued
initial data, subject in the Keller--Segel case to subcritical atomic
masses.  For a planar Keller--Segel equation with rotational flux,
Espejo and Wu \cite{EW20} identified the threshold
$8\pi/(\lambda_V\cos\alpha_V)$ when $\lambda_V\cos\alpha_V>0$.
Here we establish well-posedness in this classical subcritical range
for the stochastic equation and its skeleton with $L^2$ controls.

Stochastic Keller--Segel equations have also been studied with other
noise structures.  Huang and Qiu \cite{HQ21} proved well-posedness and a
microscopic derivation for a common-noise model.  Misiats, Stanzhytskyi
and Topaloglu \cite{MST22} studied global existence and blow-up for
stochastic perturbations of Keller--Segel, while Mayorcas and
Toma\v{s}evi\'c \cite{MT23} considered conservative transport noise.
Flandoli, Galeati and Luo \cite{FGL24} obtained quantitative scaling
limits for transport-noise SPDEs, with an application to Keller--Segel.
The transport-noise models among these works have coefficients linear in
the density, unlike the conservative square-root noise considered here.
For related models, Martini and Mayorcas \cite{AA25,AA24} established local
well-posedness and small-noise results for additive-noise approximations
of Keller--Segel--Dean--Kawasaki dynamics.

Entropy and Fisher-information estimates also play an important role in
the analysis of singular interacting particle systems.
Fournier, Hauray and Mischler \cite{FHM14} proved propagation of chaos
for the viscous vortex model using entropy and Fisher-information
estimates.  For Keller--Segel particles, Fournier and Jourdain
\cite{FJ17} studied the particle approximation and collisions, and Tardy
\cite{Tar24} proved subsequential convergence of empirical measures in
the subcritical and critical cases.  Quantitative mean-field estimates
for singular interactions were developed in \cite{BJW19,JW18,Ser20}.
Wang, Zhao and Zhu \cite{WZZ23} proved Gaussian fluctuation limits for
singular-kernel systems, including the viscous vortex model, and Chen and
Ge \cite{CG25} established a sample-path large deviation principle for
stochastic vortex dynamics.  These particle-number limits differ from
the joint small-noise and Fourier-cutoff limit studied here.  The latter
has the quadratic action with mobility $\rho$ that appears in
macroscopic fluctuation theory \cite{BDGJL}.

\subsection{Organization and notation}
\label{subsec:organization-notation}

The proof relies on an entropy estimate and an $L^1$ comparison.
For nonattractive interactions, the deterministic drift dissipates the
ordinary entropy;
for attractive interactions, the Keller--Segel free energy is coercive
under \eqref{eq:subcritical} and controls the entropy and, through its
dissipation, the Fisher information.  In the stochastic equation, the
entropy Hessian cancels the singular It\^o drift as in Fehrman and Gess.
The Coulomb energy introduces an additional logarithmic term that the
entropy does not control near vacuum.  Localizing and shifting the
interaction energy produces a second Hessian cancellation, and stopping
at successive entropy levels yields global estimates.

The second estimate is a logarithmic $L^1$ comparison.  The near-endpoint
Coulomb bound gives an Osgood inequality for both the stochastic equation
and the skeleton, using only mass conservation and integrated Fisher
information.  These estimates yield the compactness, uniqueness, and
control stability needed to verify the weak-convergence criterion for
large deviations, following \cite{FG23}.

The paper is organized as follows.
Section~\ref{sec:analytic} collects the Coulomb estimates, consequences
of mass conservation and Fisher-information bounds, and the Osgood lemma
used for both equations.  Section~\ref{sec:spde}
develops the free-energy estimates and stochastic cancellations, then
constructs stochastic kinetic solutions.  Section~\ref{sec:skeleton}
gives the corresponding skeleton construction.
Section~\ref{sec:comparison} proves the Coulomb comparison estimate for
both equations and
deduces pathwise uniqueness, the stochastic solution map, and skeleton
stability.  Section~\ref{sec:full-ldp} establishes controlled convergence
and the large deviation principle.  Appendix~\ref{app:approximation}
justifies the regularizations, nonlocal It\^o formula, and compactness
arguments; Appendix~\ref{app:osgood} proves the Osgood lemma.

We use the following conventions throughout.  For the fixed time horizon
$T>0$, set
\[
 Q_T=(0,T)\times\T^2,
 \qquad r_+=\max\{r,0\}.
\]
For $r>0$, let
$\log_+r=\max\{\log r,0\}$ and
$\log_-r=\max\{-\log r,0\}$, with
$\log_+0=0$ and $\log_-0=+\infty$.  We write $A\lesssim B$ if
$A\leq CB$ for a constant independent of the approximation, cutoff, and
solution under consideration; subscripts indicate the permitted dependence
of $C$.  The calligraphic $\mathcal I$ always denotes the Fisher
information in \eqref{eq:fisher-information}, whereas $I_{\rho_0}$ denotes
the large-deviation rate function in \eqref{eq:rate-intro}.

The torus Green function and the rotation matrix in
\eqref{eq:coulomb-decomposition} are normalized by
\begin{equation}\label{eq:green-normalization}
 -\Delta\Green=\delta_0-1,
 \qquad \int_{\T^2}\Green=0,
 \qquad
 \J=\begin{pmatrix}0&-1\\1&0\end{pmatrix}.
\end{equation}

We now specify the Fourier truncation used in
\eqref{eq:dk-exact}.  Choose
$\mathbb Z^2_+\subset\mathbb Z^2\setminus\{0\}$ containing exactly one
element of each pair $\{m,-m\}$.  For $K\in\mathbb N$, set
\begin{equation}\label{eq:fourier-cutoff}
 \begin{gathered}
 \mathcal K_K=\{m\in\mathbb Z^2_+:|m|\leq K\},
 \qquad
 \mathcal E_K=\{0\}\cup(\mathcal K_K\times\{\mathrm c,\mathrm s\}),\\
 e_0=1,\qquad
 e_{m,\mathrm c}(x)=\sqrt2\cos(2\pi m\cdot x),\qquad
 e_{m,\mathrm s}(x)=\sqrt2\sin(2\pi m\cdot x).
 \end{gathered}
\end{equation}
Let $P_K$ be the componentwise $L^2$-orthogonal projection onto the span of
these modes.  On a stochastic basis
$(\Omega,\mathcal F,(\mathcal F_t)_{t\geq0},\Pp)$ satisfying the usual
conditions, let $\{\beta^{k,j}:k\in\mathcal E_K,\ j=1,2\}$ be independent
standard Brownian motions.  With $\mathbf e_1,\mathbf e_2$ denoting the
canonical basis of $\R^2$, define
\begin{equation}\label{eq:finite-mode-noise}
 W_K(t)=\sum_{k\in\mathcal E_K}\sum_{j=1}^2
 \beta^{k,j}(t)e_k\mathbf e_j,
 \qquad
 F_{1,K}=\sum_{k\in\mathcal E_K}e_k^2,
 \qquad
 N_K=\sum_{k\in\mathcal E_K}|\nabla e_k|^2.
\end{equation}
The sine--cosine pairing gives
\[
 F_{1,K}=1+2|\mathcal K_K|,\qquad
 N_K=8\pi^2\sum_{m\in\mathcal K_K}|m|^2,\qquad
 \sum_{k\in\mathcal E_K}e_k\nabla e_k=0.
\]
In particular, $F_{1,K}\simeq K^2$ and $N_K\simeq K^4$; thus
\eqref{eq:ldp-scaling-intro} is equivalent to
$\varepsilon N_{K(\varepsilon)}\to0$ and implies
$\varepsilon F_{1,K(\varepsilon)}\to0$.

\section{Coulomb estimates and the Osgood principle}\label{sec:analytic}

We derive integrability bounds for the Coulomb flux from mass and Fisher
information, then prove the logarithmic interaction estimate.  Osgood's
lemma will be used in the stochastic and deterministic comparisons.

\subsection{Coulomb identities and bounds from mass and Fisher information}

The normalization \eqref{eq:green-normalization} and the decomposition
\eqref{eq:coulomb-decomposition} give, for every $f\in L^1(\T^2)$,
the distributional identities
\begin{equation}\label{eq:coulomb-identities}
 \begin{aligned}
 -\Delta(\Green*f)
 &=f-\int_{\T^2}f,\\
 \diver(V*f)
 &=-\lambda_V\cos(\alpha_V)
 \left(f-\int_{\T^2}f\right),
 \qquad
 \diver(\J\nabla\Green*f)=0.
 \end{aligned}
\end{equation}
Whenever $f$ has finite Green energy, one also has
\begin{equation}\label{eq:green-quadratic-identity}
 \int_{\T^2}f(\Green*f)
 =\int_{\T^2}|\nabla\Green*f|^2.
\end{equation}
In particular, this applies to the nonnegative finite-entropy densities used
below, by approximation and Lemma~\ref{lem:green-potential-coercivity}.

\begin{lemma}[Consequences of mass and Fisher information]
\label{lem:mass-fisher-consequences}
Let $\rho\geq0$ satisfy
\begin{equation*}
 \rho\in L^\infty(0,T;L^1(\T^2)),\qquad
 \int_{\T^2}\rho(t)=M\quad\text{for almost every }t, \qquad
 \int_0^T\mathcal I(\rho(t))\,\dd t<\infty.
\end{equation*}
Then
\begin{align}
 \|\rho(t)\|_{L^2}^2
 &\lesssim M\mathcal I(\rho(t))+M^2,
 &\rho&\in L^2(Q_T),\label{eq:mass-fisher-L2}\\
 \sqrt\rho&\in L^4(0,T;L^4(\T^2)),
 &\rho&\in L^4(0,T;L^{4/3}(\T^2)),\label{eq:mass-fisher-L4}\\
 V*\rho&\in L^4(0,T;L^4(\T^2)).&&
 \label{eq:mass-fisher-velocity}
\end{align}
For any two such densities,
\begin{equation}\label{eq:coulomb-interaction-integrability}
 \sqrt{\rho^i}|\nabla\sqrt{\rho^i}|\,|V*\rho^j|,
 \quad
 \nabla\rho^i\cdot(V*\rho^j),
 \quad
 \rho^i\diver(V*\rho^j)
 \in L^1(Q_T).
\end{equation}
Moreover,
\begin{equation}\label{eq:coulomb-flux-integrability}
 \sqrt\rho\,(V*\rho)\in L^2(Q_T),
 \qquad
 \nabla\rho,\ \rho(V*\rho)
 \in L^2(0,T;L^1(\T^2)).
\end{equation}
\end{lemma}

\begin{proof}
The torus Gagliardo--Nirenberg inequality applied to $\sqrt\rho$ gives
\eqref{eq:mass-fisher-L2} and the first assertion in
\eqref{eq:mass-fisher-L4}.  At exponent $8/3$ it gives
\begin{equation*}
 \|\rho(t)\|_{L^{4/3}}^4
 =\|\sqrt{\rho(t)}\|_{L^{8/3}}^8
 \lesssim_M 1+\mathcal I(\rho(t)).
\end{equation*}
The order-one Hardy--Littlewood--Sobolev estimate maps $L^{4/3}$ to
$L^4$ under convolution with the singular part of $\nabla\Green$; its
regular part is bounded.  This proves
\eqref{eq:mass-fisher-velocity}.  The remaining assertions follow from
H\"older's inequality, \eqref{eq:coulomb-identities}, and
$\nabla\rho=2\sqrt\rho\nabla\sqrt\rho$; the same estimates give
\eqref{eq:coulomb-flux-integrability}.
\end{proof}

\subsection{Endpoint Coulomb estimates and logarithmic loss}

\begin{lemma}[Quantitative endpoint estimates]
The following estimates hold uniformly in $\theta$ and $q$ in the stated
ranges.
\begin{enumerate}[label=\textup{(\roman*)}]
\item For $0<\theta\leq1/4$ and
$f\in L^{1/(1-\theta)}(\T^2)$,
\begin{equation}\label{eq:endpoint-hls}
 \|V*f\|_{L^{2/(1-2\theta)}}
 \lesssim \lambda_V\theta^{-1/2}
 \|f\|_{L^{1/(1-\theta)}}.
\end{equation}
\item For $q\geq4$ and $u\in H^1(\T^2)$,
\begin{equation}\label{eq:endpoint-sobolev}
 \|u\|_{L^q}
 \lesssim \sqrt q\,
 \|\nabla u\|_{L^2}^{1-2/q}\|u\|_{L^2}^{2/q}
 +\|u\|_{L^2}.
\end{equation}
\end{enumerate}
\end{lemma}

\begin{proof}
The singular part of $\nabla\Green$ is the two-dimensional Riesz kernel,
and the remainder is smooth.  For $1<p\leq4/3$ and
$q=2p/(2-p)$, the near--far decomposition in Hedberg's inequality gives
\begin{equation*}
 |\mathord\cdot|^{-1}*|f|\lesssim
 \|f\|_{L^p}^{p/2}(\mathcal M f)^{1-p/2}+\|f\|_{L^p}.
\end{equation*}
Here $\mathcal M$ denotes the periodic Hardy--Littlewood maximal
operator.  Since $q(1-p/2)=p$ and
\begin{equation*}
 \|\mathcal M f\|_{L^p}
 \lesssim\frac p{p-1}\|f\|_{L^p},
\end{equation*}
we obtain
\begin{equation*}
 \||\mathord\cdot|^{-1}*f\|_{L^q}
 \lesssim
 \left(\frac p{p-1}\right)^{1-p/2}\|f\|_{L^p}.
\end{equation*}
Taking $p=(1-\theta)^{-1}$ and using that
$\cos(\alpha_V)I+\sin(\alpha_V)\J$ is orthogonal proves
\eqref{eq:endpoint-hls}.

For \eqref{eq:endpoint-sobolev}, let $\{P_j\}$ be a periodic
Littlewood--Paley decomposition.  Bernstein's inequality and
Cauchy--Schwarz yield
\begin{align*}
 \|P_{\leq j_0}u\|_{L^q}
 &\lesssim2^{j_0(1-2/q)}\|u\|_{L^2},\\
 \|P_{>j_0}u\|_{L^q}
 &\lesssim\sum_{j>j_0}2^{-2j/q}2^j\|P_ju\|_{L^2}
 \lesssim\sqrt q\,2^{-2j_0/q}\|\nabla u\|_{L^2}.
\end{align*}
If $\|\nabla u\|_{L^2}>\|u\|_{L^2}$, choose $2^{j_0}$ comparable to
their ratio; otherwise take $j_0=0$.
\end{proof}

\begin{lemma}[Logarithmic Coulomb interaction]\label{lem:log-interaction}
Let $\rho^1,\rho^2\geq0$ have the same mass $M$ and satisfy
$\sqrt{\rho^i}\in H^1(\T^2)$.  Then
\begin{equation}\label{eq:log-interaction}
\begin{aligned}
 &\int_{\T^2}|\nabla\rho^2|\,
 \bigl|V*(\rho^1-\rho^2)\bigr|
 +\lambda_V|\cos(\alpha_V)|
 \int_{\T^2}\rho^2|\rho^1-\rho^2|\\
 &\quad\lesssim_{M,V}
 \left(1+\sum_{i=1}^2
 \mathcal I(\rho^i)\right)
 \|\rho^1-\rho^2\|_{L^1}
 \log\!\left(
 \frac{e(1+2M)}{\|\rho^1-\rho^2\|_{L^1}}
 \right).
\end{aligned}
\end{equation}
The right-hand side is zero when the two densities agree.
\end{lemma}

\begin{proof}
The assertion is immediate if $M=0$ or the two densities agree.  For
$0<\theta\leq1/4$, interpolation gives
\begin{equation}\label{eq:difference-interpolation}
 \|\rho^1-\rho^2\|_{L^{1/(1-\theta)}}
 \leq
 \|\rho^1-\rho^2\|_{L^1}^{1-2\theta}
 \|\rho^1-\rho^2\|_{L^2}^{2\theta}.
\end{equation}
Applying \eqref{eq:endpoint-sobolev} to $\sqrt{\rho^2}$ with
exponents $1/\theta$ and $2/\theta$ gives
\begin{align*}
 \|\sqrt{\rho^2}\|_{L^{1/\theta}}
 &\lesssim_M\theta^{-1/2}
 \left(1+
 \mathcal I(\rho^2)^{(1-2\theta)/2}\right),\\
 \|\rho^2\|_{L^{1/\theta}}
 &\lesssim_M\theta^{-1}
 \left(1+
 \mathcal I(\rho^2)^{1-\theta}\right).
\end{align*}
Furthermore, \eqref{eq:mass-fisher-L2} implies
\begin{equation*}
 \|\rho^1-\rho^2\|_{L^2}^2
 \lesssim_M 1+\sum_{i=1}^2
 \mathcal I(\rho^i).
\end{equation*}
Using $\nabla\rho^2=2\sqrt{\rho^2}\nabla\sqrt{\rho^2}$,
H\"older's inequality, \eqref{eq:endpoint-hls}, and
\eqref{eq:difference-interpolation}, followed by weighted Young
inequalities, yields
\begin{align*}
 &\int_{\T^2}|\nabla\rho^2|\,
 \bigl|V*(\rho^1-\rho^2)\bigr|
 +\lambda_V|\cos(\alpha_V)|
 \int_{\T^2}\rho^2|\rho^1-\rho^2|\\
 &\quad\lesssim_{M,V}\frac1\theta
 \left(1+\sum_{i=1}^2
 \mathcal I(\rho^i)\right)
 \|\rho^1-\rho^2\|_{L^1}^{1-2\theta}.
\end{align*}
Since $\|\rho^1-\rho^2\|_{L^1}\leq2M$, the choice
\begin{equation*}
 \theta=
 \frac1{4\log\!\left(e(1+2M)/
 \|\rho^1-\rho^2\|_{L^1}\right)}
\end{equation*}
satisfies $0<\theta\leq1/4$ and
\begin{equation*}
 \theta^{-1}=4\log\!\left(
 \frac{e(1+2M)}{\|\rho^1-\rho^2\|_{L^1}}
 \right),
 \qquad
 \|\rho^1-\rho^2\|_{L^1}^{1-2\theta}
 \leq e^{1/2}\|\rho^1-\rho^2\|_{L^1}.
\end{equation*}
Substitution eliminates $\theta$ and proves
\eqref{eq:log-interaction}.
\end{proof}

\subsection{Osgood's lemma}

\begin{lemma}[Osgood's lemma]\label{lem:shared-osgood}
Let $M>0$, let $a\in L^1(0,T)$ be nonnegative, and let
$\delta:[0,T]\to[0,2M]$ be measurable.  Suppose that, for some
$C\geq0$, for almost every $t\in(0,T)$,
\begin{equation}\label{eq:abstract-osgood-inequality}
 \delta(t)\leq \delta(0)+C\int_0^t a(s)\delta(s)
 \log\!\left(\frac{e(1+2M)}{\delta(s)}\right)\dd s,
\end{equation}
where the integrand is zero when $\delta(s)=0$.  If $\delta(0)=0$, then
$\delta(t)=0$ for almost every $t$.  If $\delta(0)>0$, then
\begin{equation}\label{eq:osgood-continuous-dependence}
 \log\!\left(\frac{e(1+2M)}{\delta(t)}\right)
 \geq
 \exp\!\left(-C\int_0^t a(s)\dd s\right)
 \log\!\left(\frac{e(1+2M)}{\delta(0)}\right).
\end{equation}
The left-hand side of \eqref{eq:osgood-continuous-dependence} is interpreted
as $+\infty$ when $\delta(t)=0$.  In particular, this deterministic conclusion
may be applied pathwise whenever \eqref{eq:abstract-osgood-inequality}
holds outside one null set.
\end{lemma}
The proof is given in Appendix~\ref{app:osgood}.

\section{Mass-subcritical stochastic estimates and construction}\label{sec:spde}

We prove the estimates and existence assertions of
Theorem~\ref{thm:main-spde}.  The nonattractive entropy estimate and weak
construction follow
\cite[Theorem~3.1 and Proposition~3.3]{JSW26}, after translating the sign
convention.  For attractive interactions, we use the mass-subcritical
free-energy estimate and entropy localization proved below.
Section~\ref{sec:comparison} proves comparison, pathwise uniqueness, and
the existence of a Borel solution map, using the common-noise comparison
of \cite[Theorem~4.7]{FG24}.

\subsection{Stochastic kinetic formulation}

For a predictable $g$ with finite $L^2(Q_T)$ energy, we write the
controlled equation as
\begin{equation}\label{eq:controlled-spde}
\begin{aligned}
 \dd\rho={}&\left[
 \Delta\rho-\diver\bigl(\rho(V*\rho)\bigr)
 -\diver(\sqrt\rho\,P_Kg)
 \right]\dd t
 -\sqrt\varepsilon\,\diver(\sqrt\rho\circ\dd W_K).
\end{aligned}
\end{equation}

The stationary sine--cosine family satisfies
$\sum_ke_k\nabla e_k=0$.  Thus, for smooth strictly positive
approximations, \eqref{eq:controlled-spde} has the It\^o form
\begin{equation}\label{eq:controlled-spde-ito}
\begin{aligned}
 \dd\rho={}&\left[
 \Delta\rho-\diver\bigl(\rho(V*\rho)\bigr)
 -\diver(\sqrt\rho\,P_Kg)
 +\frac{\varepsilon F_{1,K}}8\Delta\log\rho
 \right]\dd t\\
 &-\sqrt\varepsilon
 \sum_{k\in\mathcal E_K}\sum_{j=1}^2
 \partial_j(e_k\sqrt\rho)\,\dd\beta_t^{k,j}.
\end{aligned}
\end{equation}
This identity is not used as a weak formulation at vacuum.  It only
identifies the correction in velocity-localized renormalized identities
and in the smooth coefficient approximation of
Appendix~\ref{app:approximation}.

\begin{definition}[Stochastic kinetic solution]
\label{def:stochastic-kinetic-solution}
Let $(\Omega,\mathcal F,(\mathcal F_t)_{t\in[0,T]},\Pp)$ carry the
Brownian motions in \eqref{eq:finite-mode-noise}, and let $g$ be a
predictable $L^2(\T^2;\R^2)$-valued process such that
$\int_{Q_T}|g|^2<\infty$ almost surely.
A stochastic kinetic solution of \eqref{eq:controlled-spde} with initial
datum $\rho_0$ and control $g$ is a nonnegative,
almost surely continuous $L^1(\T^2)$-valued predictable process
\begin{equation*}
 \rho\in L^1(\Omega\times(0,T);L^1(\T^2))
\end{equation*}
such that, almost surely for every $t\in[0,T]$,
\begin{equation*}
 \int_{\T^2}\rho(x,t)\,\dd x
 =\int_{\T^2}\rho_0(x)\,\dd x=M.
\end{equation*}
It belongs to the physical-energy class
\begin{equation}\label{eq:stochastic-finite-energy}
 \sup_{t\in[0,T]}\int_{\T^2}H(\rho(t))
 +\int_0^T\mathcal D_V(\rho(t))\,\dd t<\infty
 \qquad\Pp\text{-a.s.}
\end{equation}
Under the subcritical condition \eqref{eq:subcritical},
Lemma~\ref{lem:green-potential-coercivity} shows that this is equivalent, up to
constants depending on $V$ and $M$, to the corresponding condition with
$\sup_{t\leq T}\mathcal F_V(\rho(t))$ in place of the entropy supremum.

Its kinetic function and initial kinetic function are
\begin{equation*}
 \chi_\rho(x,\xi,t)=\mathbf 1_{\{0<\xi<\rho(x,t)\}},
 \qquad
 \chi_{\rho_0}(x,\xi)=\mathbf 1_{\{0<\xi<\rho_0(x)\}}.
\end{equation*}
There is a measurable random nonnegative Radon measure $q$ on
$\T^2\times(0,\infty)\times[0,T]$, almost surely finite on compact velocity
intervals, such that, for every
$\psi\in C_c^\infty(\T^2\times(0,\infty))$, the process
\begin{equation*}
 t\longmapsto
 \int_{\T^2\times(0,\infty)\times[0,t]}
 \psi(x,\xi)\,q(\dd x\,\dd\xi\,\dd s)
\end{equation*}
is predictable, as in \cite[Definition~3.1]{FG24}.  In addition, almost
surely as measures on
$\T^2\times(0,\infty)\times[0,T]$,
\begin{equation}\label{eq:stochastic-parabolic-lower-bound}
 4\xi\,\delta_{\rho(x,t)}(\dd\xi)
 |\nabla\sqrt\rho(x,t)|^2\dd x\,\dd t\leq q,
\end{equation}
where the Fisher regularity needed to interpret the left-hand side follows
directly from \eqref{eq:stochastic-finite-energy} and
\eqref{eq:full-fisher-physical-dissipation}.  For $m\in\mathbb N$,
define the canonical energy stopping time
\begin{equation}\label{eq:kinetic-energy-stops}
 \sigma_m:=T\wedge\inf\left\{t\geq0:
  \sup_{s\leq t}\int_{\T^2}H(\rho(s))
  +\int_0^t\mathcal D_V(\rho(s))\,\dd s
  +\int_0^t\!\int_{\T^2}|g|^2\,\dd x\,\dd s>m
 \right\},
\end{equation}
with $\inf\varnothing=\infty$.  Lower semicontinuity of
$\rho\mapsto\int H(\rho)$ along the adapted $L^1$-continuous path, or
equivalently the use of rational times in the running maximum, shows that
$\sigma_m$ is a stopping time.  The kinetic measure is required to satisfy,
for every $m\in\mathbb N$,
\begin{equation}\label{eq:stopped-high-velocity-tail}
 \lim_{R\to\infty}\E q\bigl(
 \T^2\times[R,R+1]\times[0,\sigma_m]\bigr)=0.
\end{equation}

Finally, for every $\psi\in C_c^\infty(\T^2\times(0,\infty))$, almost
surely for every $t\in[0,T]$,
\begin{align}
&\int_0^\infty\!\int_{\T^2}
 \chi_\rho(x,\xi,t)\psi(x,\xi)\,\dd x\,\dd\xi
 =\int_0^\infty\!\int_{\T^2}
 \chi_{\rho_0}(x,\xi)\psi(x,\xi)\,\dd x\,\dd\xi
 \notag\\
&\quad-\int_0^t\!\int_{\T^2}
 \nabla\rho\cdot(\nabla_x\psi)(x,\rho)\,\dd x\,\dd s
 -\frac{\varepsilon F_{1,K}}8
 \int_0^t\!\int_{\T^2}
 \nabla\log\rho\cdot(\nabla_x\psi)(x,\rho)\,\dd x\,\dd s
 \notag\\
&\quad-\int_{\T^2\times(0,\infty)\times[0,t]}
 \partial_\xi\psi(x,\xi)\,q(\dd x\,\dd\xi\,\dd s)
 +\frac{\varepsilon N_K}{2}
 \int_0^t\!\int_{\T^2}
 \rho\,(\partial_\xi\psi)(x,\rho)\,\dd x\,\dd s
 \notag\\
&\quad+\int_0^t\!\int_{\T^2}
 \bigl(\rho(V*\rho)+\sqrt\rho\,P_Kg\bigr)
 \cdot(\nabla_x\psi)(x,\rho)\,\dd x\,\dd s
 \notag\\
&\quad+\int_0^t\!\int_{\T^2}
 \bigl(\rho(V*\rho)+\sqrt\rho\,P_Kg\bigr)
 \cdot\nabla\rho\,(\partial_\xi\psi)(x,\rho)\,\dd x\,\dd s
 \notag\\
&\quad+\sqrt\varepsilon
 \sum_{k\in\mathcal E_K}\sum_{j=1}^2
 \int_0^t\!\int_{\T^2}
 \left[
 \sqrt\rho\,e_k\,\partial_j\psi(x,\rho)
 +\sqrt\rho\,e_k\,\partial_j\rho\,
 (\partial_\xi\psi)(x,\rho)
 \right]\dd x\,\dd\beta_s^{k,j}.
 \label{eq:controlled-stochastic-kinetic-identity}
\end{align}
Here $\nabla_x\psi$ differentiates only the spatial variable, before the
velocity variable is evaluated at $\xi=\rho$.  All gradient terms are
localized to a compact subset of $(0,\infty)$ by the support of $\psi$ and
are locally integrable on the stops
\eqref{eq:kinetic-energy-stops}; the stochastic integral is understood as a
continuous local martingale.  When $g=0$ we use the same terminology, with
the control terms omitted.
\end{definition}

\begin{remark}[Relation with the Fehrman--Gess formulation]
\label{rem:comparison-fg-definition}
The compact-velocity identity in
Definition~\ref{def:stochastic-kinetic-solution} is the specialization of
\cite[Definition~3.4, equation~(3.8)]{FG24} to $\Phi(r)=r$,
$\sigma(r)=\sqrt r$, and $F_2=0$.  The modification concerns the
integrability and kinetic-measure conditions, which are used here on the
canonical energy stops \eqref{eq:kinetic-energy-stops}.

Indeed, Proposition~\ref{prop:fisher-from-dissipation} gives
\begin{equation*}
 \mathcal I(\rho)
 \lesssim_{V,M}1+\mathcal D_V(\rho)
 +\exp\!\left(C_{V,M}
 \left(1+\int_{\T^2}H(\rho)\right)\right).
\end{equation*}
The pathwise free-energy estimate therefore yields pathwise finite Fisher
information, whereas an expectation bound would require an exponential
entropy estimate.  After integration in time, the right-hand side is
deterministically bounded on every canonical stop.  We accordingly replace
the global
high-velocity condition in \cite[Definition~3.4]{FG24} by
\eqref{eq:stopped-high-velocity-tail}.  This condition is verified by
Proposition~\ref{prop:approximation-high-velocity-tail}, and the stopped
low-velocity estimate follows from the argument of
\cite[Proposition~4.6]{FG24}; see
Lemma~\ref{lem:canonical-energy-localization}.  Since the canonical stops
exhaust $[0,T]$ almost surely, these localized conditions suffice for the
kinetic comparison and compactness arguments.
\end{remark}

\begin{theorem}[Global stochastic well-posedness]
\label{thm:spde-gwp}
Let $\rho_0\geq0$ have finite entropy and satisfy
\eqref{eq:subcritical}.  Fix $\varepsilon>0$ and $K<\infty$.
Equation~\eqref{eq:dk-exact}, equivalently
\eqref{eq:controlled-spde} with $g=0$, is globally well posed in the class
of stochastic kinetic solutions from
Definition~\ref{def:stochastic-kinetic-solution}.  The
modes are the real orthonormal
sine--cosine family in
\eqref{eq:fourier-cutoff}; in particular, they satisfy the three constant
covariance identities following \eqref{eq:finite-mode-noise}, and $P_K$
is an $L^2$ contraction.  The solution is nonnegative, conserves the mass
$M$, and has initial trace $\rho_0$.  For every $T<\infty$,
\begin{equation}\label{eq:spde-pathwise-estimates}
 \sup_{t\leq T}\int_{\T^2}H(\rho(t))
 +\int_0^T\mathcal D_V(\rho(t))\,\dd t
 +\int_0^T\mathcal I(\rho(t))\,\dd t<\infty
 \qquad\Pp\text{-a.s.}
\end{equation}
The expectation of the first two terms in
\eqref{eq:spde-pathwise-estimates} is finite.  At an arbitrary fixed noise
strength, the full Fisher integral is asserted to be finite almost surely,
but no first or exponential moment of it is included in this statement.
Its kinetic measure satisfies the high- and low-velocity estimates on the
canonical stops in Definition~\ref{def:stochastic-kinetic-solution} and
Lemma~\ref{lem:canonical-energy-localization}.  If, in addition,
\begin{equation}\label{eq:expected-full-fisher}
 \E\int_0^T\mathcal I(\rho(t))\,\dd t<\infty,
\end{equation}
then both tails are unlocalized and the solution satisfies the global
conditions of \cite[Definition~3.4]{FG24}.
\end{theorem}

For the controlled extension, we first consider predictable $g$ for which,
for some deterministic $N_g<\infty$,
\begin{equation}\label{eq:bounded-control-action}
 \int_{Q_T}|g|^2\leq N_g\qquad\Pp\text{-a.s.}
\end{equation}

If $M=0$, nonnegativity and mass conservation give $\rho=0$, and every
claim is immediate; the logarithmic chemical-potential notation is not
used for this solution.  We henceforth assume $M>0$.

\subsection{Physical free energy and the vacuum obstruction}

We start with the entropy production created by the deterministic flow.  The Green-potential calculations in the next two lemmas are justified first
for smooth positive densities and then by mass-preserving heat regularization
and lower semicontinuity.
For a smooth positive density, define the physical chemical potential by
\begin{equation}\label{eq:physical-chemical-potential}
 \mu(\rho)=\log\rho-
 \lambda_V\cos(\alpha_V)\Green*\rho.
\end{equation}
In particular, $\mathcal D_V(\rho)=\int_{\T^2}\rho|\nabla\mu(\rho)|^2$
and $\mathcal I(\rho)=\int_{\T^2}|\nabla\rho|^2/\rho$ for such densities.
The deterministic controlled equation can then be written as
\begin{equation*}
 \partial_t\rho
 =\diver\bigl(\rho\nabla\mu(\rho)\bigr)
 -\lambda_V\sin(\alpha_V)
  \diver\bigl(\rho\J\nabla\Green*\rho\bigr)
 -\diver(\sqrt\rho\,g).
\end{equation*}
Since $D\mathcal F_V(\rho)=\mu(\rho)$ up to a spatial constant,
integration by parts yields
\begin{equation}\label{eq:deterministic-free-energy-backbone}
 \frac{\dd}{\dd t}\mathcal F_V(\rho)
 +\mathcal D_V(\rho)
 =\int_{\T^2}\sqrt\rho\,g\cdot\nabla\mu(\rho).
\end{equation}
Indeed, the rotational contribution vanishes because
\begin{equation*}
 \int_{\T^2}\rho\,\J\nabla\Green*\rho\cdot\nabla\mu(\rho)
 =\int_{\T^2}\J\nabla\Green*\rho\cdot\nabla\rho=0.
\end{equation*}
Here we used both
$\J\nabla\Green*\rho\cdot\nabla\Green*\rho=0$ and
$\diver(\J\nabla\Green*\rho)=0$.  Thus Young's inequality gives the
common deterministic estimate
\begin{equation*}
 \mathcal F_V(\rho(t))
 +\frac12\int_0^t\mathcal D_V(\rho(s))\,\dd s
 \leq\mathcal F_V(\rho_0)
 +\frac12\int_0^t\!\int_{\T^2}|g|^2.
\end{equation*}

The following lemma collects the potential bounds and sharp coercivity
behind the Keller--Segel mass threshold.

\begin{lemma}[Green-potential bounds and sharp coercivity]
\label{lem:green-potential-coercivity}
Let $\rho\geq0$ have mass $M$ and finite entropy.  Then
\begin{equation*}
 \|\Green*\rho\|_{L^\infty}
 +\|\nabla\Green*\rho\|_{L^2}^2
 \lesssim_M 1+\int_{\T^2}H(\rho).
\end{equation*}
The same estimate holds with $\rho$ inside either convolution replaced by
a measurable $f$ satisfying $0\leq f\leq\rho$.  Moreover,
\begin{equation*}
 \frac12\int_{\T^2}\rho(\Green*\rho)
 \leq\frac{M}{8\pi}\int_{\T^2}H(\rho)+C_M,
\end{equation*}
and hence
\begin{equation}\label{eq:free-energy-coercivity}
 \mathcal F_V(\rho)
 \geq
 \left(1-\frac{\lambda_V(\cos(\alpha_V))_+M}{8\pi}\right)
 \int_{\T^2}H(\rho)-C_{V,M}.
\end{equation}
Consequently, under \eqref{eq:subcritical}, one may fix any
\begin{equation*}
 0<c_{\mathrm{sub}}<
 1-\frac{\lambda_V(\cos(\alpha_V))_+M}{8\pi}
\end{equation*}
and obtain
\begin{equation*}
 \mathcal F_V(\rho)
 \geq c_{\mathrm{sub}}\int_{\T^2}H(\rho)-C_{V,M,c_{\mathrm{sub}}}.
\end{equation*}
If $\lambda_V\cos(\alpha_V)\leq0$, the Green quadratic form is
nonnegative and $\mathcal F_V(\rho)\geq\int H(\rho)$.
\end{lemma}

\begin{proof}
The local expansion
\begin{equation*}
 \Green(z)=-\frac1{2\pi}\log|z|+g(z),
 \qquad g\in C^\infty,
\end{equation*}
shows that $\Green$ is bounded from below and that
$\exp(\beta\Green_+)\in L^1(\T^2)$ for every $0<\beta<4\pi$.
Young's inequality $rz\leq r\log r-r+e^z$, applied with
$z=\beta\Green_+(x-y)$, gives, uniformly in $x$,
\begin{equation*}
 \int_{\T^2}\rho(y)\Green_+(x-y)\,\dd y
 \lesssim_{M,\beta}1+\int_{\T^2}H(\rho).
\end{equation*}
The lower bound for $\Green$ controls the negative part, and
$\|\nabla\Green*\rho\|_2^2=\int\rho(\Green*\rho)$ gives the gradient
estimate.  If $0\leq f\leq\rho$, use
$f\log_+f\lesssim\rho\log_+\rho+\rho$ in the same argument.

The coercivity assertions are immediate for $M=0$.  For $M>0$, the sharp
Moser--Trudinger inequality on the unit torus \cite{Fon93} states that
every mean-zero $\varphi\in H^1(\T^2)$ satisfies
\begin{equation*}
 \log\int_{\T^2}e^\varphi
 \leq\frac1{16\pi}\int_{\T^2}|\nabla\varphi|^2+C_{\T^2}.
\end{equation*}
Apply the entropy variational inequality for $\rho/M$ with
$\varphi=8\pi(\Green*\rho)/M$ and use
\eqref{eq:green-quadratic-identity}.  This gives
\begin{equation*}
 \frac{8\pi}{M^2}\int_{\T^2}\rho(\Green*\rho)
 \leq \frac1M\int_{\T^2}\rho\log\!\left(\frac\rho M\right)
 +\frac{4\pi}{M^2}\int_{\T^2}\rho(\Green*\rho)+C_{\T^2}.
\end{equation*}
Moving the last Green term to the left and absorbing the mass-dependent
constants therefore yields
\begin{equation*}
 \frac12\int_{\T^2}\rho(\Green*\rho)
 \leq \frac{M}{8\pi}\int_{\T^2}H(\rho)+C_M,
\end{equation*}
which proves the sharp interaction estimate.
Multiplication by the attractive coefficient, together with positivity of
the Green form in the nonattractive case, proves
\eqref{eq:free-energy-coercivity}.
\end{proof}

\begin{lemma}[Entropy-sublevel interpolation]
\label{lem:entropy-sublevel-interpolation}
Fix $M,L<\infty$.  For every $a>0$, every nonnegative $\rho$ satisfying
\begin{equation*}
 \int_{\T^2}\rho=M,
 \qquad \int_{\T^2}H(\rho)\leq L,
 \qquad \sqrt\rho\in H^1(\T^2),
\end{equation*}
obeys
\begin{equation}\label{eq:entropy-sublevel-interpolation}
 \|\rho\|_{L^2}^2
 \leq a\mathcal I(\rho)
 +C_{a,M}\exp\!\left(C_a(1+L)\right).
\end{equation}
\end{lemma}

\begin{proof}
For $R>e$, let $(\sqrt\rho-\sqrt R)_+$ denote the positive part.  On
$\{\rho\leq4R\}$ one has $\rho^2\leq4R\rho$, while on
$\{\rho>4R\}$ one has
$\rho^2\leq16(\sqrt\rho-\sqrt R)_+^4$.  Hence
\begin{equation*}
 \|\rho\|_{L^2}^2
 \leq4RM+16\|(\sqrt\rho-\sqrt R)_+\|_{L^4}^4.
\end{equation*}
The torus Gagliardo--Nirenberg inequality and the entropy tail bound give
\begin{align*}
 \|(\sqrt\rho-\sqrt R)_+\|_{L^4}^4
 &\lesssim
 \|(\sqrt\rho-\sqrt R)_+\|_{L^2}^2\mathcal I(\rho)
 +\|(\sqrt\rho-\sqrt R)_+\|_{L^2}^4,\\
 \|(\sqrt\rho-\sqrt R)_+\|_{L^2}^2
 &\leq\int_{\{\rho>R\}}\rho
 \lesssim\frac{M+L}{\log R}.
\end{align*}
Choose $R$ so that the coefficient of $\mathcal I(\rho)$ is at most
$a$.  The required $R$ grows at most exponentially in $(1+L)/a$, which
gives \eqref{eq:entropy-sublevel-interpolation}.
\end{proof}

\begin{proposition}[Full Fisher information from physical dissipation]
\label{prop:fisher-from-dissipation}
If $\rho\geq0$ has mass $M$, finite entropy, and
$\mathcal D_V(\rho)<\infty$, then
\begin{equation}\label{eq:full-fisher-physical-dissipation}
 \mathcal I(\rho)
 \leq2\mathcal D_V(\rho)
 +C_{V,M}\exp\!\left(C_{V,M}
 \left(1+\int_{\T^2}H(\rho)\right)\right).
\end{equation}
Moreover, we have
\begin{equation}\label{eq:fisher-from-dissipation}
 \mathcal I(\rho)
 +\int_{\T^2}\rho|\nabla\Green*\rho|^2
 \lesssim_{V,M}1+\mathcal D_V(\rho)
 +\exp\!\left(C_{V,M}
 \left(1+\int_{\T^2}H(\rho)\right)\right).
\end{equation}
\end{proposition}

\begin{proof}
For $A>e$, set $f=(\sqrt\rho-\sqrt A)_+$.  The entropy tail and the torus
Gagliardo--Nirenberg inequality give
\begin{equation*}
 \|f\|_2^2\lesssim\frac{M+\int H(\rho)}{\log A},
 \qquad
 \|\rho\|_2^2
 \lesssim AM+\frac{M+\int H(\rho)}{\log A}\mathcal I(\rho)+C_M.
\end{equation*}
The exact expansion of the physical dissipation is
\begin{equation*}
 \mathcal D_V(\rho)
 =\mathcal I(\rho)
 -2\lambda_V\cos(\alpha_V)(\|\rho\|_2^2-M^2)
 +\lambda_V^2\cos^2(\alpha_V)
 \int\rho|\nabla\Green*\rho|^2.
\end{equation*}
Dropping the last, nonnegative term and using the preceding interpolation
bound yields
\begin{equation*}
 \mathcal I(\rho)
 \leq \mathcal D_V(\rho)+C_{V,M}(1+A)
 +C_{V,M}\frac{M+\int H(\rho)}{\log A}\,\mathcal I(\rho).
\end{equation*}
Choose
$\log A=C_{V,M}(1+\int H(\rho))$ with a sufficiently large constant.
The last term is then absorbed into the left-hand side, which proves
\eqref{eq:full-fisher-physical-dissipation}.  Finally,
\begin{equation}\label{eq:weighted-potential-gradient}
 \int_{\T^2}\rho|\nabla\Green*\rho|^2
 \leq\|\rho\|_2\|\nabla\Green*\rho\|_4^2
 \lesssim_M1+\|\rho\|_2^2,
\end{equation}
because $\|\nabla\Green*\rho\|_4\lesssim
\|\rho-M\|_{4/3}$ and
$\|\rho\|_{4/3}^2\leq M\|\rho\|_2$.  This proves
\eqref{eq:fisher-from-dissipation}; approximation completes the proof for
nonsmooth densities.
\end{proof}

\paragraph{The stochastic obstruction.}

The physical free energy $\mathcal F_V$, rather than the entropy alone, is
the natural Lyapunov functional for the deterministic Coulomb drift.  We
therefore begin with its It\^o differential.  The entropy component retains
the exact cancellation of the singular It\^o drift identified by
Fehrman--Gess.  The interaction component, however, leaves the logarithmic
pairing in \eqref{eq:vacuum-log-identity}.  For every smooth positive density
of mass $M$, monotonicity of the logarithm gives
\[
 \int_{\T^2}(\rho-M)\log\rho
 =\frac12\iint_{\T^2\times\T^2}
 (\rho(x)-\rho(y))(\log\rho(x)-\log\rho(y))\,\dd x\,\dd y
 \geq0.
\]
Its sign in the It\^o remainder is therefore favorable when
$\lambda_V\cos(\alpha_V)\leq0$, but unfavorable in the attractive regime.
In the latter case the unlocalized free-energy estimate does not control
this term near vacuum.  The following balance isolates this obstruction.

\begin{proposition}[Smooth untruncated stochastic free-energy balance]
Let $\rho$ be a smooth strictly positive approximation of
\eqref{eq:controlled-spde}.  Then
\begin{equation}\label{eq:physical-stochastic-balance}
\begin{aligned}
 \dd\mathcal F_V(\rho)+\mathcal D_V(\rho)\,\dd t
 ={}&\int_{\T^2}\sqrt\rho\,P_Kg\cdot\nabla\mu(\rho)\,\dd t\\
 &+\sqrt\varepsilon\sum_{k,j}
 \left(\int_{\T^2}\sqrt\rho\,e_k
 \partial_j\mu(\rho)\right)\dd\beta_t^{k,j}
 +\mathcal R_{\varepsilon,K}(\rho)\,\dd t,
\end{aligned}
\end{equation}
where
\begin{equation}\label{eq:physical-stochastic-remainder}
\begin{aligned}
 \mathcal R_{\varepsilon,K}(\rho)
 ={}&\frac{\varepsilon N_K}{2}
 +\frac{\varepsilon\lambda_V\cos(\alpha_V)F_{1,K}}8
 \int_{\T^2}\nabla(\Green*\rho)\cdot\nabla\log\rho\\
 &-\frac{\varepsilon\lambda_V\cos(\alpha_V)}2
 \sum_{k\in\mathcal E_K}\sum_{j=1}^2
 \left\langle
 \partial_j(e_k\sqrt\rho),
 \Green*\partial_j(e_k\sqrt\rho)
 \right\rangle.
\end{aligned}
\end{equation}
The two nonlocal corrections satisfy
\begin{align}
 \int_{\T^2}\nabla(\Green*\rho)\cdot\nabla\log\rho
 &=\int_{\T^2}(\rho-M)\log\rho,
 \label{eq:vacuum-log-identity}\\
 \sum_{k,j}\left\langle
 \partial_j(e_k\sqrt\rho),
 \Green*\partial_j(e_k\sqrt\rho)
 \right\rangle
 &=F_{1,K}M-
 \sum_{k\in\mathcal E_K}
 \left(\int_{\T^2}e_k\sqrt\rho\right)^2.
 \label{eq:green-mode-identity}
\end{align}
\end{proposition}

This is a smooth-level identity used to identify the obstruction.  The
regularized nonlocal It\^o calculation employed in the construction is
Proposition~\ref{prop:regularized-ito-trace}.

\begin{proof}
The deterministic terms in \eqref{eq:physical-stochastic-balance} are
those of \eqref{eq:deterministic-free-energy-backbone}.  For the entropy part of
the stochastic trace, the contribution of the It\^o drift and the
gradient contribution of the entropy Hessian are
\begin{equation}\label{eq:fg-entropy-cancellation}
 -\frac{\varepsilon F_{1,K}}8
 \int_{\T^2}|\nabla\log\rho|^2
 +\frac{\varepsilon F_{1,K}}8
 \int_{\T^2}|\nabla\log\rho|^2=0.
\end{equation}
This is precisely the exact-square-root cancellation in
\cite[Proposition~5.9, equations~(5.16)--(5.18), and
Proposition~5.18, equation~(5.29)]{FG24}.  The remaining local entropy
trace is $\varepsilon N_K/2$.  Pairing the same correction with the
Coulomb part of $D\mathcal F_V$ and taking the Hessian of its Green
quadratic form gives \eqref{eq:physical-stochastic-remainder}. The Green equation gives \eqref{eq:vacuum-log-identity}.  Integrating
twice by parts gives, for each $k$,
\begin{equation*}
 \sum_{j=1}^2\left\langle
 \partial_j(e_k\sqrt\rho),\Green*\partial_j(e_k\sqrt\rho)
 \right\rangle
 =\int_{\T^2}e_k^2\rho
 -\left(\int_{\T^2}e_k\sqrt\rho\right)^2.
\end{equation*}
Summation proves \eqref{eq:green-mode-identity}; its right-hand side
lies in $[0,F_{1,K}M]$.
\end{proof}

When $\lambda_V\cos(\alpha_V)\leq0$, the weak construction and the
entropy--Fisher estimate are those of
\cite[Theorem~3.1 and Proposition~3.3]{JSW26}, after translating the
sign convention: the radial entropy production is nonpositive, whereas
the Biot--Savart contribution vanishes.  The same entropy calculation
with the control term, followed by Young's inequality, is uniform over
bounded predictable controls.  Equations~\eqref{eq:mass-fisher-L2} and
\eqref{eq:weighted-potential-gradient} then give the required bound on
the physical dissipation.  Hence only the attractive case requires the
localized energy.  Throughout the localization argument
below, we therefore assume
\begin{equation*}
 \lambda_V\cos(\alpha_V)>0.
\end{equation*}

\subsection{Localized interaction energy and the exact It\^o cancellation}

The logarithmic trace term motivates the following modification: the ordinary entropy
is left unchanged, while only the interaction part of the free energy is
switched off near vacuum.  A positive shift in the localized potential
gives the local interaction-Hessian contribution a favorable nonpositive
sign and produces the second exact cancellation below.  For $S\geq1$,
choose a smooth nondecreasing function such that
\begin{equation}\label{eq:localized-cutoff-objects}
 \begin{gathered}
 \theta_S(r)=0\quad(r\leq e^{-S}),\qquad
 \theta_S(r)=1\quad(r\geq e^{-S/2}),\qquad
 0\leq r\theta_S'(r)\lesssim S^{-1},\\
 h_S(r)=\int_0^r\theta_S(s)\,\dd s,\qquad
 \Lambda_*=1+M(-\inf_{\T^2}\Green)_+,\qquad
 \Phi_S(\rho)=\Green*h_S(\rho)+\Lambda_*.
 \end{gathered}
\end{equation}
By \eqref{eq:localized-cutoff-objects}, we have
$0\leq\rho-h_S(\rho)\leq e^{-S/2}$; moreover,
$\theta_S'(\rho)$ is supported in $\{\rho\leq e^{-S/2}\}$, and
\begin{equation}\label{eq:mass-uniform-positive-shift}
 \Phi_S(\rho)\geq1.
\end{equation}
Indeed, $\Green*h_S(\rho)\geq(\inf\Green)\int h_S(\rho)$ and
$0\leq\int h_S(\rho)\leq M$.
The shift $\Lambda_*$ depends only on the mass and on $\Green$; in
particular, it is independent of entropy stopping levels.

\begin{definition}[Localized Coulomb energy]
For a nonnegative density of mass $M$, set
\begin{equation}\label{eq:localized-energy-objects}
 \begin{aligned}
 \mathcal F_S(\rho)
 &:=\int_{\T^2}H(\rho)
 -\lambda_V\cos(\alpha_V)\left[
 \frac12\int_{\T^2}h_S(\rho)(\Green*h_S(\rho))
 +\Lambda_*\left(\int_{\T^2}h_S(\rho)-M\right)\right],\\
 \mu_S(\rho)
 &:=\log\rho-\lambda_V\cos(\alpha_V)\theta_S(\rho)\Phi_S(\rho),\\
 \mathcal D_S(\rho)
 &:=\int_{\T^2}\rho|\nabla\mu_S(\rho)|^2.
 \end{aligned}
\end{equation}
At vacuum, $\mathcal D_V$ and $\mathcal D_S$ mean the squared norms of
the corresponding weighted fluxes in
\eqref{eq:free-energy-dissipation} and
\eqref{eq:localized-energy-objects}, rather than a pointwise value of the
logarithm.
\end{definition}

\begin{lemma}[Comparison and variations of the localized energy]
\label{lem:localized-energy-comparison}
For every state and every $S\geq1$,
\begin{equation}\label{eq:localized-energy-comparison}
 |\mathcal F_S(\rho)-\mathcal F_V(\rho)|
 \lesssim_{V,M}e^{-S/2},
 \qquad
 c_{\mathrm{sub}}\int_{\T^2}H(\rho)-C_{V,M,c_{\mathrm{sub}}}
 \leq\mathcal F_S(\rho)
 \leq\int_{\T^2}H(\rho)+C_{V,M}e^{-S/2}.
\end{equation}
For smooth positive $\rho$ and every smooth direction $v$,
\begin{equation}\label{eq:shifted-energy-first-variation}
 D\mathcal F_S(\rho)=\mu_S(\rho),
\end{equation}
up to an irrelevant spatial constant, and
\begin{equation}\label{eq:shifted-energy-hessian}
 \begin{aligned}
 D^2\mathcal F_S(\rho)[v,v]
 ={}&\int_{\T^2}\frac{v^2}{\rho}
 -\lambda_V\cos(\alpha_V)
 \int_{\T^2}\theta_S'(\rho)\Phi_S(\rho)v^2\\
 &-\lambda_V\cos(\alpha_V)
 \left\langle\theta_S(\rho)v,
 \Green*(\theta_S(\rho)v)\right\rangle.
 \end{aligned}
\end{equation}
\end{lemma}

\begin{proof}
Symmetry of $\Green$, $\|\rho-h_S(\rho)\|_\infty\leq e^{-S/2}$, and
$\int\rho=M$ give
\begin{equation*}
 \left|\langle\rho,\Green*\rho\rangle
 -\langle h_S(\rho),\Green*h_S(\rho)\rangle\right|
 =\left|\langle \rho-h_S(\rho),
 \Green*(\rho+h_S(\rho))\rangle\right|
 \leq2M\|\Green\|_1e^{-S/2}.
\end{equation*}
Moreover $0\leq M-\int h_S(\rho)\leq e^{-S/2}$.  This proves the first
estimate in \eqref{eq:localized-energy-comparison}; the remaining two
follow from Lemma~\ref{lem:green-potential-coercivity} and positivity
of the Green quadratic form.  Direct differentiation gives
\eqref{eq:shifted-energy-first-variation} and
\eqref{eq:shifted-energy-hessian}.
\end{proof}

Applying It\^o's formula to $\mathcal F_S$ yields the following identity,
involving both the physical chemical potential $\mu$ from
\eqref{eq:physical-chemical-potential} and the auxiliary first variation
$\mu_S$.

\begin{proposition}[It\^o formula for the localized free energy]
\label{prop:localized-coulomb-ito}
Let a smooth positive mass-preserving process solve the It\^o equation
\eqref{eq:controlled-spde-ito}, and suppose that the Fourier modes are
the orthonormal family in \eqref{eq:fourier-cutoff}.  Then
\begin{equation}\label{eq:localized-coulomb-ito}
 \begin{aligned}
 \dd\mathcal F_S(\rho)
 ={}&\left[-\frac12\mathcal D_V(\rho)
 -\frac12\mathcal D_S(\rho)
 +\frac12\int_{\T^2}\rho
 \left|\nabla\bigl(\mu_S(\rho)-\mu(\rho)\bigr)\right|^2\right.\\
 &\left.\qquad
 +\lambda_V\sin(\alpha_V)
 \int_{\T^2}\rho\,\J\nabla\Green*\rho\cdot\nabla\mu_S(\rho)\right.\\
 &\left.\qquad
 +\int_{\T^2}\sqrt\rho\,P_Kg\cdot\nabla\mu_S(\rho)
 +\mathscr C_S^{\varepsilon,K}(\rho)\right]\dd t
 +\dd\mathcal M_t^S,
 \end{aligned}
\end{equation}
where
\begin{equation}\label{eq:localized-energy-martingale}
 \begin{aligned}
 \dd\mathcal M_t^S
 &:=\sqrt\varepsilon\sum_{k\in\mathcal E_K}\sum_{j=1}^2
 \left(\int_{\T^2}e_k\sqrt\rho\,\partial_j\mu_S(\rho)\right)
 \dd\beta_t^{k,j},
 \qquad
 \dd\langle\mathcal M^S\rangle_t
 \leq\varepsilon\mathcal D_S(\rho)\,\dd t.
 \end{aligned}
\end{equation}
The complete trace remaining after cancellation is
\begin{equation}\label{eq:complete-ito-hessian-grouping}
 \begin{aligned}
 \mathscr C_S^{\varepsilon,K}(\rho)
 ={}&\frac{\varepsilon N_K}{2}
 +\frac{\varepsilon\lambda_V\cos(\alpha_V)F_{1,K}}8
 \int_{\T^2}\theta_S(\rho)\nabla\Green*h_S(\rho)
 \cdot\nabla\log\rho\\
 &-\frac{\varepsilon\lambda_V\cos(\alpha_V)N_K}{2}
 \int_{\T^2}\rho\theta_S'(\rho)\Phi_S(\rho)\\
 &-\frac{\varepsilon\lambda_V\cos(\alpha_V)}2
 \sum_{k\in\mathcal E_K}\sum_{j=1}^2
 \left\langle\theta_S(\rho)\partial_j(e_k\sqrt\rho),
 \Green*\bigl(\theta_S(\rho)\partial_j(e_k\sqrt\rho)\bigr)
 \right\rangle.
 \end{aligned}
\end{equation}
In particular,
\begin{equation}\label{eq:localized-trace-upper-bound}
 \mathscr C_S^{\varepsilon,K}(\rho)
 \lesssim_{V,M}\varepsilon N_K
 +\varepsilon F_{1,K}
 \left(1+\int_{\T^2}H(\rho)+MS\right).
\end{equation}
\end{proposition}

Proposition~\ref{prop:regularized-ito-trace} gives the corresponding identity
for the Green-kernel and square-root approximations used below.  Thus the
displayed cancellation is retained throughout the limiting construction.

\begin{proof}
The radial part of the physical drift is
$\diver(\rho\nabla\mu(\rho))$.  Hence polarization gives exactly
\begin{equation*}
 -\int_{\T^2}\rho\nabla\mu_S\cdot\nabla\mu
 =-\frac12\mathcal D_V-\frac12\mathcal D_S
 +\frac12\int_{\T^2}\rho|\nabla(\mu_S-\mu)|^2.
\end{equation*}
The rotational drift and the control give the two remaining deterministic
terms in \eqref{eq:localized-coulomb-ito}.  Bessel's inequality for the
orthonormal modes proves the bracket estimate in
\eqref{eq:localized-energy-martingale}.

The entropy part cancels exactly as in
\eqref{eq:fg-entropy-cancellation}, leaving $\varepsilon N_K/2$.
The localization produces a second exact cancellation:
\begin{equation}\label{eq:localized-coulomb-hessian-cancellation}
 \frac{\varepsilon\lambda_V\cos(\alpha_V)F_{1,K}}8
 \int_{\T^2}\theta_S'(\rho)\Phi_S(\rho)
       \frac{|\nabla\rho|^2}{\rho}
 -\frac{\varepsilon\lambda_V\cos(\alpha_V)F_{1,K}}8
 \int_{\T^2}\theta_S'(\rho)\Phi_S(\rho)
       \frac{|\nabla\rho|^2}{\rho}=0.
\end{equation}
The first term in
\eqref{eq:localized-coulomb-hessian-cancellation} comes from
differentiating $\theta_S(\rho)\Phi_S(\rho)$ in the first variation; the
second comes from the local interaction Hessian.  Both cancellations are
performed before any singular term is estimated.  The uncancelled terms
are exactly
\eqref{eq:complete-ito-hessian-grouping}.  Its last two lines are
nonpositive by \eqref{eq:mass-uniform-positive-shift} and positivity of
the Green quadratic form.

Since
$0\leq\int_0^r\theta_S(s)s^{-1}\,\dd s\leq S+\log_+r$,
integration by parts gives
\begin{equation*}
 \int\theta_S(\rho)\nabla\Green*h_S(\rho)\cdot\nabla\log\rho
 =\int\left(h_S(\rho)-\int h_S(\rho)\right)
 \left(\int_0^\rho\frac{\theta_S(s)}s\,\dd s\right).
\end{equation*}
The last integral is nonnegative because both scalar factors are
nondecreasing functions of $\rho$, and it is bounded above by
$MS+\int H(\rho)+C_M$.  This proves
\eqref{eq:localized-trace-upper-bound}.
\end{proof}

The trace estimate is global in the state.  Compatibility between
$\mathcal F_S$ and the physical drift is instead an entropy-sublevel
statement, to which we now turn.

\begin{proposition}[Compatibility on an entropy sublevel]
\label{prop:stopped-coulomb-energy}
If $R,L\geq1$, $\int H(\rho)\leq R$, and $S\geq L(1+R)$, then
\begin{equation}\label{eq:localized-compatibility-bounds}
 \begin{aligned}
 \int_{\T^2}\rho
 \left|\nabla\bigl(\mu_S(\rho)-\mu(\rho)\bigr)\right|^2
 &\lesssim_{V,M}\frac1{L^2}\mathcal D_V(\rho)
 +e^{-S/2}(1+R)+e^{-S},\\
 \left|\lambda_V\sin(\alpha_V)
 \int_{\T^2}\rho\,\J\nabla\Green*\rho\cdot\nabla\mu_S(\rho)\right|
 &\lesssim_{V,M}e^{-S/2}
 \left(1+\int_{\T^2}H(\rho)\right).
 \end{aligned}
\end{equation}
Consequently, there are $L_0,c_0>0$, depending only on $V$ and $M$,
such that, if $L\geq L_0$ and $S=L(1+R)$ along a stochastic interval,
then
\begin{equation*}
 \begin{aligned}
 \dd\mathcal F_S(\rho)
 +c_0\bigl(\mathcal D_V(\rho)+\mathcal D_S(\rho)\bigr)\dd t
 \leq{}&\left[C_{V,M}e^{-L/2}+C\varepsilon N_K\right.\\
 &\left.\qquad+C_V\varepsilon F_{1,K}
 \left(1+\int H(\rho)+MS\right)
 +\int_{\T^2}|g|^2\right]\dd t
 +\dd\mathcal M_t^S,
 \end{aligned}
\end{equation*}
with $\dd\langle\mathcal M^S\rangle_t
\leq\varepsilon\mathcal D_S(\rho)\dd t$.  One may take $c_0=1/4$ after
increasing $L_0$.
\end{proposition}

\begin{proof}
The exact difference of the chemical potentials is
\begin{equation*}
 \nabla(\mu_S-\mu)
 =\lambda_V\cos(\alpha_V)\left[
 \nabla\Green*(\rho-h_S(\rho))
 +(1-\theta_S(\rho))\nabla\Green*h_S(\rho)
 -\theta_S'(\rho)\Phi_S(\rho)\nabla\rho\right].
\end{equation*}
The first two terms on the right have weighted squared norms bounded by
$e^{-S}$ and $e^{-S/2}(1+R)$, respectively.  Indeed,
$\|\nabla\Green*(\rho-h_S(\rho))\|_\infty
\leq\|\nabla\Green\|_1e^{-S/2}$, whereas the second term is supported on
$\{\rho\leq e^{-S/2}\}$ and is controlled by
Lemma~\ref{lem:green-potential-coercivity}.

For the cutoff-derivative term, observe on $\{\rho>0\}$ that
\begin{equation*}
 \frac{\nabla\rho}{\sqrt\rho}
 =\sqrt\rho\,\nabla\mu(\rho)
 +\lambda_V\cos(\alpha_V)\sqrt\rho\,\nabla\Green*\rho.
\end{equation*}
Squaring on $\{\rho\leq\delta\}$ and using
Lemma~\ref{lem:green-potential-coercivity} gives, for every $\delta>0$,
\begin{equation}\label{eq:low-density-fisher}
 \int_{\{\rho\leq\delta\}}\frac{|\nabla\rho|^2}{\rho}
 \leq2\mathcal D_V(\rho)
 +2\lambda_V^2\cos^2(\alpha_V)\delta
 \|\nabla\Green*\rho\|_2^2
 \lesssim_{V,M}\mathcal D_V(\rho)
 +\delta\left(1+\int_{\T^2}H(\rho)\right).
\end{equation}
Taking $\delta=e^{-S/2}$, the last term satisfies
\begin{align*}
 \int\rho|\theta_S'(\rho)\Phi_S(\rho)\nabla\rho|^2
 &\lesssim_M\frac{(1+R)^2}{S^2}
 \int_{\{\rho\leq e^{-S/2}\}}\frac{|\nabla\rho|^2}{\rho}\\
 &\lesssim_{V,M}\frac1{L^2}
 \left(\mathcal D_V(\rho)+e^{-S/2}(1+R)\right),
\end{align*}
by \eqref{eq:low-density-fisher}.  This proves
the first estimate in \eqref{eq:localized-compatibility-bounds}.

For the rotational term, an integration by parts using
$\diver(\J\nabla\Green*\rho)=0$, cancels the shift and cutoff-derivative
terms and gives
\begin{equation*}
 \begin{aligned}
 &\lambda_V\sin(\alpha_V)
 \int_{\T^2}\rho\,\J\nabla\Green*\rho\cdot\nabla\mu_S(\rho)\\
 &\qquad=-\lambda_V^2\sin(\alpha_V)\cos(\alpha_V)
 \int_{\T^2}h_S(\rho)\,
 \bigl[\J\nabla\Green*(\rho-h_S(\rho))\bigr]
 \cdot\bigl(\nabla\Green*h_S(\rho)\bigr).
 \end{aligned}
\end{equation*}
Here the contribution containing
$(\J\nabla\Green*h_S(\rho))\cdot(\nabla\Green*h_S(\rho))$ vanishes
pointwise.  The
kernel bounds for $\nabla\Green$ and $D^2\Green$ give
\begin{equation}\label{eq:biot-savart-log-lipschitz}
 \left|
 \bigl[\J\nabla\Green*(\rho-h_S(\rho))\bigr](x)
 -\bigl[\J\nabla\Green*(\rho-h_S(\rho))\bigr](y)
 \right|
 \lesssim e^{-S/2}r\left(1+\log_+\frac1r\right),
 \qquad r=d(x,y).
\end{equation}
Indeed, one integrates $|\nabla\Green|$ over the two balls of radius
$2r$ and uses the mean-value bound on their complement, whose radial
integral is $\int_r^1s^{-1}\,\dd s$.  Oddness of $\nabla\Green$
symmetrizes the preceding integral; then
\eqref{eq:biot-savart-log-lipschitz}, exponential integrability of the
logarithmic kernel, and Young's entropy inequality yield
the second estimate in \eqref{eq:localized-compatibility-bounds}.

Choose $L_0$ large enough that the $L^{-2}\mathcal D_V$ term in
\eqref{eq:localized-compatibility-bounds} is absorbed, and use
$(1+R)e^{-L(1+R)/2}\lesssim e^{-L/2}$.  Insert the resulting mismatch
and rotational bounds, together with
\eqref{eq:localized-trace-upper-bound} into
\eqref{eq:localized-coulomb-ito}.  The control satisfies
\begin{equation*}
 \int_{\T^2}\sqrt\rho\,P_Kg\cdot\nabla\mu_S(\rho)
 \leq\frac14\mathcal D_S(\rho)+\int_{\T^2}|g|^2.
\end{equation*}
The stated inequality follows after absorption.  The trace bound does not
require an entropy bound; such a bound is needed to compare the physical
and localized chemical potentials.
\end{proof}

Fix a coercivity constant $c_{\mathrm{sub}}$ as in
Lemma~\ref{lem:green-potential-coercivity}, and
fix $L\geq L_0$, where $L_0$ is given by
Proposition~\ref{prop:stopped-coulomb-energy}.  We first work with the
smooth regularized solutions, whose entropy has a continuous
representative, and transfer the resulting uniform estimate to the
kinetic limit only afterward.

The hitting times below are stopping times because
Lemma~\ref{lem:regularized-entropy-chain-rule} in
Appendix~\ref{app:approximation} gives a continuous adapted representative
of the regularized entropy.  This entropy continuity is established
independently of the localization argument; it does not follow from
$L^1$ continuity of the density.

\subsection{Global estimates and weak construction}

Set
\begin{equation}\label{eq:dyadic-entropy-localization}
 \begin{gathered}
 R_n=2^n\left(1+\int_{\T^2}H(\rho_0)\right),\qquad n\geq0,\\
 \tau_0=0,\qquad
 \tau_{n+1}=\inf\left\{t\geq\tau_n:
 \int_{\T^2}H(\rho(t))\geq R_{n+1}\right\},\\
 S_n=L(1+R_{n+1}).
 \end{gathered}
\end{equation}
The infimum of the empty set is $+\infty$.  On
$(\tau_n,\tau_{n+1}]$ we apply the fixed-functional identity to
$\mathcal F_{S_n}$.  Thus only the localization scale in the estimate is
changed; the equation and its noise are unchanged.

Consequently,
\begin{equation}\label{eq:entropy-hitting-value}
 \int_{\T^2}H(\rho(\tau_n))=R_n
 \quad\text{on }\{\tau_n<\infty\},\qquad n\geq1.
\end{equation}
Thus no state-dependent parameter is inserted into It\^o's formula, which
would create derivatives with respect to $S$ and mixed Hessian terms.

Two elementary consequences of the dyadic choice will be used in the
summation.  Lemma~\ref{lem:localized-energy-comparison} and the dyadic
growth of $R_n$ give
\begin{equation}\label{eq:energy-switch-budget}
 \sum_{n\geq1}\left|
 \mathcal F_{S_n}(\rho(\tau_n))
 -\mathcal F_{S_{n-1}}(\rho(\tau_n))
 \right|\lesssim_{V,M}e^{-L/2}.
\end{equation}
Indeed,
$|\mathcal F_{S_n}(\rho(\tau_n))-
\mathcal F_{S_{n-1}}(\rho(\tau_n))|
\lesssim_{V,M}(e^{-S_n/2}+e^{-S_{n-1}/2})$.
If $t\in(\tau_n,\tau_{n+1}]$, then
\begin{equation}\label{eq:active-scale-running-maximum}
 \begin{aligned}
 S_n&\leq L\left[1+2\max\left\{
 1+\int_{\T^2}H(\rho_0),
 \sup_{s\leq t}\int_{\T^2}H(\rho(s))\right\}\right]\\
 &\lesssim L\left(1+\int_{\T^2}H(\rho_0)
 +\sup_{s\leq t}\int_{\T^2}H(\rho(s))\right).
 \end{aligned}
\end{equation}
For $n\geq1$, \eqref{eq:entropy-hitting-value} gives
$\sup_{s\leq t}\int H(\rho(s))\geq R_n$ and $R_{n+1}=2R_n$; the first
stage is bounded directly by the initial entropy.

For $N\geq1$, stop after finitely many stages and define
\begin{equation}\label{eq:concatenated-dissipation-martingale}
 \begin{aligned}
 \mathscr H_t^N
 &:=\sup_{s\leq t\wedge\tau_N}\int_{\T^2}H(\rho(s)),\\
 \mathscr D_t^N
 &:=\int_0^{t\wedge\tau_N}\mathcal D_V(\rho(s))\,\dd s
 +\sum_{n=0}^{N-1}\int_{t\wedge\tau_n}^{t\wedge\tau_{n+1}}
 \mathcal D_{S_n}(\rho(s))\,\dd s,\\
 \mathcal M_t^N
 &:=\sum_{n=0}^{N-1}\int_0^{t\wedge\tau_N}
 \mathbf1_{(\tau_n,\tau_{n+1}]}(s)\,\dd\mathcal M_s^{S_n},
 \qquad
 \langle\mathcal M^N\rangle_t\leq\varepsilon\mathscr D_t^N.
 \end{aligned}
\end{equation}
The indicator has a predictable version, so the concatenated stochastic
integral is well defined.  No strong Markov property is used.

\begin{theorem}[Subcritical entropy estimate]
\label{thm:subcritical-entropy-estimate}
Suppose that $\lambda_V\cos(\alpha_V)>0$,
$\lambda_V\cos(\alpha_V)M<8\pi$, and that
\eqref{eq:bounded-control-action} holds.
Let $\rho$ be a smooth positive mass-preserving approximation for which
Proposition~\ref{prop:localized-coulomb-ito}, or its compatible
regularized counterpart in
Proposition~\ref{prop:regularized-ito-trace}, is valid.  There are constants
$c_0,C>0$, independent of the number of stages and of the compatible
coefficient and Green regularizations, such that
\begin{equation}\label{eq:summed-localized-energy}
 \begin{aligned}
 &c_{\mathrm{sub}}\int_{\T^2}H(\rho(t\wedge\tau_N))
 +c_0\mathscr D_t^N\\
 &\quad\leq C\left[1+\int_{\T^2}H(\rho_0)+N_g
 +T\left\{e^{-L/2}+\varepsilon N_K
 +\varepsilon F_{1,K}L
 \left(1+\int_{\T^2}H(\rho_0)\right)\right\}\right]\\
 &\qquad+C\varepsilon F_{1,K}L\int_0^t\mathscr H_s^N\,\dd s
 +\mathcal M_t^N.
 \end{aligned}
\end{equation}
Consequently, uniformly in $N$,
\begin{equation}\label{eq:entropy-dissipation-tail}
 \begin{aligned}
 &\Pp\Biggl(\mathscr H_T^N+\mathscr D_T^N>
 C\exp\!\left(\frac{C\varepsilon F_{1,K}LT}{c_{\mathrm{sub}}}\right)
 \Biggl[1+\int_{\T^2}H(\rho_0)+N_g+r\\[-2mm]
 &\hspace{42mm}
 +T\left\{e^{-L/2}+\varepsilon N_K
 +\varepsilon F_{1,K}L
 \left(1+\int_{\T^2}H(\rho_0)\right)\right\}\Biggr]\Biggr)\\
 &\hspace{32mm}\leq\exp\!\left(-\frac{c_0r}{\varepsilon}\right),
 \qquad r\geq0,
 \end{aligned}
\end{equation}
and
\begin{equation}\label{eq:entropy-dissipation-expectation}
 \begin{aligned}
 \E[\mathscr H_T^N+\mathscr D_T^N]
 \lesssim{}&\exp\!\left(\frac{C\varepsilon F_{1,K}LT}
 {c_{\mathrm{sub}}}\right)
 \left[1+\int_{\T^2}H(\rho_0)+N_g+\frac{\varepsilon}{c_0}\right.\\
 &\left.\qquad
 +T\left\{e^{-L/2}+\varepsilon N_K
 +\varepsilon F_{1,K}L
 \left(1+\int_{\T^2}H(\rho_0)\right)\right\}\right].
 \end{aligned}
\end{equation}
For $\varepsilon=0$, the deterministic estimate
\eqref{eq:summed-localized-energy} and its resulting uniform bound hold with
the martingale and every $\varepsilon$-dependent term omitted; the
probability-tail formulation \eqref{eq:entropy-dissipation-tail} is not used.
\end{theorem}

\begin{proof}
Apply Proposition~\ref{prop:stopped-coulomb-energy}, or its uniform
regularized counterpart, with the fixed parameter $S_n$ on each of the
first $N$ intervals.  Sum these finitely
many inequalities, include the switch costs from
\eqref{eq:energy-switch-budget}, and use
\eqref{eq:active-scale-running-maximum}.  The lower and upper energy
bounds in \eqref{eq:localized-energy-comparison} yield
\eqref{eq:summed-localized-energy}; all constants are independent of
$N$.  If an ordinary bracket localization is needed, stop additionally
when $\mathscr D^N$ reaches an integer and remove this stop after obtaining the
uniform estimate.

The bracket bound in \eqref{eq:concatenated-dissipation-martingale} and
the exponential local-martingale inequality give
\begin{equation}\label{eq:concatenated-martingale-tail}
 \Pp\left(\sup_{t\leq T}
 \left\{\mathcal M_t^N-\frac{c_0}{2}\mathscr D_t^N\right\}>r\right)
 \leq e^{-c_0r/\varepsilon}.
\end{equation}
Indeed, the process
\begin{equation*}
 \exp\!\left(\frac{c_0}{\varepsilon}\mathcal M_t^N
 -\frac{c_0^2}{2\varepsilon^2}\langle\mathcal M^N\rangle_t\right)
\end{equation*}
is a nonnegative local martingale; multiplication by the remaining
nonincreasing finite-variation factor gives the supermartingale used in
the maximal inequality.  This does not assign mean zero to an
unlocalized local martingale.

Subtract $c_0\mathscr D^N/2$ from the martingale in
\eqref{eq:summed-localized-energy}.  Taking the running supremum in the
entropy term and applying deterministic Gr\"onwall pathwise gives
\begin{equation}\label{eq:running-maximum-bound}
 \begin{aligned}
 \mathscr H_T^N+\mathscr D_T^N
 \lesssim{}&\exp\!\left(\frac{C\varepsilon F_{1,K}LT}
 {c_{\mathrm{sub}}}\right)
 \left[1+\int_{\T^2}H(\rho_0)+N_g\right.\\
 &+T\left\{e^{-L/2}+\varepsilon N_K
 +\varepsilon F_{1,K}L
 \left(1+\int_{\T^2}H(\rho_0)\right)\right\}\\
 &\left.+\sup_{t\leq T}
 \left\{\mathcal M_t^N-\frac{c_0}{2}\mathscr D_t^N\right\}\right].
 \end{aligned}
\end{equation}
Equations \eqref{eq:concatenated-martingale-tail} and
\eqref{eq:running-maximum-bound} prove
\eqref{eq:entropy-dissipation-tail}; integration of the exponential tail
proves \eqref{eq:entropy-dissipation-expectation}.
\end{proof}

\begin{corollary}[Removal of the entropy-level stopping]
Under the hypotheses of
Theorem~\ref{thm:subcritical-entropy-estimate},
\begin{equation}\label{eq:entropy-level-nonaccumulation}
 \begin{aligned}
 \Pp(\tau_N\leq T)
 &\leq\exp\Biggl[-\frac{c_0}{\varepsilon}
 \Biggl(C^{-1}e^{-C\varepsilon F_{1,K}LT}R_N\\
 &\qquad-C\Biggl\{1+\int_{\T^2}H(\rho_0)+N_g
 +T\Bigl(e^{-L/2}+\varepsilon N_K\\
 &\hspace{42mm}
 +\varepsilon F_{1,K}L
 \Bigl(1+\int_{\T^2}H(\rho_0)\Bigr)\Bigr)\Biggr\}\Biggr)_+\Biggr]
 \longrightarrow0.
 \end{aligned}
\end{equation}
Thus the entropy levels cannot accumulate in finite time.  After removal
of the entropy stopping,
\begin{equation}\label{eq:global-entropy-dissipation-estimate}
 \begin{aligned}
 &\E\left[\sup_{t\leq T}\int_{\T^2}H(\rho(t))
 +\int_0^T\mathcal D_V(\rho(t))\,\dd t\right.\\
 &\left.\hspace{18mm}
 +\sum_{n\geq0}\int_{\tau_n\wedge T}^{\tau_{n+1}\wedge T}
 \mathcal D_{S_n}(\rho(t))\,\dd t\right]\\
 &\quad\lesssim \exp(C\varepsilon F_{1,K}LT)
 \left[1+\int_{\T^2}H(\rho_0)+N_g+\varepsilon\right.\\
 &\left.\hspace{35mm}
 +T\left\{e^{-L/2}+\varepsilon N_K
 +\varepsilon F_{1,K}L
 \left(1+\int_{\T^2}H(\rho_0)\right)\right\}\right].
 \end{aligned}
\end{equation}
\end{corollary}

\begin{proof}
On $\{\tau_N\leq T\}$, entropy continuity gives
$\mathscr H_T^N\geq R_N$.
Combine this with \eqref{eq:running-maximum-bound} and
\eqref{eq:concatenated-martingale-tail}.  Since these events decrease in
$N$, \eqref{eq:entropy-level-nonaccumulation} excludes finite-time
accumulation.  Fatou's lemma in
\eqref{eq:entropy-dissipation-expectation} proves
\eqref{eq:global-entropy-dissipation-estimate}.
\end{proof}

The preceding argument proves an a priori estimate, not a local existence
or continuation theorem.  It is first applied to the compatible smooth
approximations, where entropy hitting times are legitimate, and then
transferred by strong $L^1$ compactness and lower semicontinuity.  The
solution construction and identification are addressed below and in
Appendix~\ref{app:approximation}.

At every fixed finite $(\varepsilon,K)$,
\eqref{eq:global-entropy-dissipation-estimate} and
Proposition~\ref{prop:fisher-from-dissipation} imply
\begin{equation*}
 \int_0^T\mathcal D_V(\rho(t))\,\dd t<\infty,
 \qquad
 \int_0^T\mathcal I(\rho(t))\,\dd t<\infty
 \quad\Pp\text{-a.s.}
\end{equation*}
The second assertion holds almost surely; the preceding estimate does
not imply a finite first moment at an arbitrary noise strength.  For
every fixed noise strength and $A<\infty$, however,
\begin{equation*}
 \begin{aligned}
 \E\int_0^T\!\int_{\{\rho\leq A\}}
 \frac{|\nabla\rho|^2}{\rho}
 \lesssim{}&\E\int_0^T\mathcal D_V(\rho(t))\,\dd t\\
 &+A\E\int_0^T\left(1+\int_{\T^2}H(\rho(t))\right)\dd t<\infty,
 \end{aligned}
\end{equation*}
by \eqref{eq:low-density-fisher}.  In particular, for every
$\ell\geq1$,
\begin{equation*}
 \E\int_0^T
 \left\|\nabla\bigl[(\rho\wedge\ell)\vee\ell^{-1}\bigr]\right\|_2^2
 \dd t
 \leq\ell\E\int_0^T\!\int_{\{\rho\leq\ell\}}
 \frac{|\nabla\rho|^2}{\rho}<\infty,
\end{equation*}
which is precisely the local Sobolev integrability in
\cite[equation~(3.5)]{FG24}.  No exponential moment of the full Fisher
integral at speed $\varepsilon^{-1}$ is asserted.

\paragraph{Weak construction.}

The compatible regularizations and their order of removal are specified in
Appendix~\ref{app:approximation}; the entropy-level stopping is removed before
the coefficient and Green-kernel regularizations are removed.  At each fixed
regularization, \cite[Proposition~3.2]{JSW26} gives the unforced strong
solution, whose measurable factorization follows
\cite[Theorem~6.1, equation~(6.6)]{FG23}.  Controlled approximations are
also available: under \eqref{eq:bounded-control-action}, the $L^2$
contraction property of $P_K$ gives
\begin{equation}\label{eq:finite-mode-novikov}
 \E\exp\left\{\frac1{2\varepsilon}
 \int_{Q_T}|P_Kg|^2\right\}
 \leq\exp\left\{\frac{N_g}{2\varepsilon}\right\}<\infty.
\end{equation}
By finite-dimensional Girsanov
\cite[Proposition~6.2]{FG23}, evaluating this map at
\begin{equation}\label{eq:finite-mode-cameron-martin-shift}
 \sqrt\varepsilon W_K+\int_0^\cdot P_Kg(s)\,\dd s
\end{equation}
produces the controlled approximation.  The estimates above were derived
directly for the controlled equation and therefore hold under the original
measure; Girsanov is used only to identify the shifted solution map.

For the singular limit, we first take $g=0$.  Weak construction and
kinetic identification follow
\cite[Theorem~3.1 and Proposition~3.5]{JSW26}.  In the attractive case,
their a priori estimate is replaced by
\eqref{eq:global-entropy-dissipation-estimate}; in the nonattractive case,
the entropy calculation cited above, with the control term absorbed by
Young's inequality, gives the required estimate.  For compactness, we use
Lemma~\ref{lem:cutoff-time-compactness}, which replaces the global expected
Fisher-information bound in the cited proof by compactness of density
truncations, followed by removal of the entropy and Fisher-information
localization in probability.  The entropy and physical
flux satisfy the lower-semicontinuity estimates in
Lemmas~\ref{lem:entropy-tail-transfer} and
\ref{lem:physical-flux-lower-semicontinuity}; the Coulomb drift passes to
the limit by Lemma~\ref{lem:interaction-compactness}.
The remaining terms are identified directly in the compact-velocity
stochastic kinetic identity, as detailed in
Appendix~\ref{app:approximation}.  The resulting unforced candidate
satisfies every clause of
Definition~\ref{def:stochastic-kinetic-solution} except the high-velocity
condition; Proposition~\ref{prop:approximation-high-velocity-tail} supplies
that condition, and Lemma~\ref{lem:canonical-energy-localization} gives the
low-velocity tail.  Thus \eqref{eq:spde-pathwise-estimates} holds for the
unforced solution.  After pathwise uniqueness and measurable
factorization, the same Girsanov shift constructs the singular controlled
solution; Section~\ref{sec:comparison} also verifies its estimates under
the original probability measure.

For a control of merely almost surely finite action, we will use
$g\mathbf1_{[0,\gamma_N]}$, where
\begin{equation}\label{eq:control-action-stop}
 \gamma_N=T\wedge\inf\left\{t\geq0:
 \int_0^t\!\int_{\T^2}|g|^2\geq N\right\}.
\end{equation}
After uniqueness, the stopped solutions paste consistently as
$N\to\infty$, as in Section~\ref{sec:comparison}.

The argument above establishes unforced weak existence and the required
a priori estimates.  Pathwise uniqueness, the controlled extension, and
measurable representation in Theorems~\ref{thm:spde-gwp} and
\ref{thm:main-spde}\textup{(i)--(ii)} are completed in
Section~\ref{sec:comparison};
part~\textup{(iii)} of Theorem~\ref{thm:main-spde} follows from
the uniform estimate \eqref{eq:uniform-small-noise-energy}.

\section{The skeleton equation: formulation and construction}\label{sec:skeleton}

The skeleton equation associated with the small-noise problem is
\begin{equation}\label{eq:skeleton}
 \partial_t\rho
 =\Delta\rho-\diver\bigl(\rho(V*\rho)\bigr)
 -\diver(\sqrt\rho\,g),
 \qquad \rho(0)=\rho_0,
\end{equation}
where $g\in L^2(\QT;\R^2)$.  After omitting the Coulomb flux,
\eqref{eq:skeleton} is the case $\Phi(r)=r$ of
\cite[equation~(1.2)]{FG23}.  In this section we introduce the entropy and
kinetic formulations and construct an entropy solution with the required energy
bounds.  Uniqueness and weak-to-strong stability are deduced from the shared
kinetic comparison in Section~\ref{sec:comparison}.

\subsection{Entropy and kinetic formulations}

\begin{definition}[Entropy skeleton solution]\label{def:entropy-skeleton}
Let $\rho_0\in L^1(\T^2)$ be nonnegative with
$\int_{\T^2}H(\rho_0)<\infty$ and mass $M$.
An entropy solution of \eqref{eq:skeleton} is a nonnegative function
$\rho$ such that
\begin{equation}\label{eq:skeleton-entropy-class}
 \rho\in L^\infty(0,T;L^1(\T^2)),
 \qquad \sqrt\rho\in L^2(0,T;H^1(\T^2)),
 \qquad
 \operatorname*{ess\,sup}_{t\in[0,T]}
 \int_{\T^2}H(\rho(t))<\infty,
\end{equation}
\begin{equation*}
 \int_{\T^2}\rho(t)=M
 \quad\text{for almost every }t,
\end{equation*}
and, for every $\varphi\in C_c^\infty([0,T)\times\T^2)$,
\begin{equation}\label{eq:skeleton-weak}
\begin{aligned}
0={}&\int_{\T^2}\rho_0\varphi(0)
 +\int_{\QT}\rho(\partial_t+\Delta)\varphi\\
 &+\int_{\QT}\rho(V*\rho)\cdot\nabla\varphi
 +\int_{\QT}\sqrt\rho\,g\cdot\nabla\varphi.
\end{aligned}
\end{equation}
The initial condition is included in \eqref{eq:skeleton-weak}.
\end{definition}

\begin{definition}[Renormalized kinetic skeleton solution]
\label{def:kinetic-skeleton}
Let $\rho_0\in L^1(\T^2)$ be nonnegative with finite entropy, and let
$g\in L^2(\QT;\R^2)$.  A renormalized kinetic solution of
\eqref{eq:skeleton} is a nonnegative function satisfying
\eqref{eq:skeleton-entropy-class},
\begin{equation*}
 \int_{\T^2}\rho(t)=\int_{\T^2}\rho_0=M
 \quad\text{for almost every }t,
\end{equation*}
and whose kinetic functions and exact parabolic defect are
\begin{equation*}
 \begin{gathered}
 \chi_\rho(x,\xi,t)=\mathbf 1_{\{0<\xi<\rho(x,t)\}},\\
 \chi_{\rho_0}(x,\xi)=\mathbf 1_{\{0<\xi<\rho_0(x)\}},\\
 p_\rho(\dd x\,\dd\xi\,\dd t)
 =4\xi\,\delta_{\rho(x,t)}(\dd\xi)
 |\nabla\sqrt\rho(x,t)|^2\dd x\,\dd t.
 \end{gathered}
\end{equation*}
There is a null set
$\mathcal N\subset(0,T]$ such that, for every
$t\in[0,T]\setminus\mathcal N$ and every
\begin{equation*}
 \psi\in C_c^\infty(\T^2\times(0,\infty))
 \cap C(\T^2\times[0,\infty)),
 \qquad \psi(x,0)=0,
\end{equation*}
these kinetic functions and defect satisfy
\begin{equation*}
\begin{aligned}
 \int_{\T^2\times\R}\chi_\rho(x,\xi,t)\psi(x,\xi)
       \dd x\dd\xi
={}&\int_0^t\!\int_{\T^2\times\R}
       \chi_\rho(x,\xi,s)\Delta_x\psi(x,\xi)
       \dd x\dd\xi\dd s
    -\int_{\T^2\times\R\times(0,t)}
       \partial_\xi\psi\,\dd p_\rho \\
 &+2\int_0^t\!\int_{\T^2}
     \rho\nabla\sqrt\rho\cdot
     \bigl(g+\sqrt\rho\,(V*\rho)\bigr)
     (\partial_\xi\psi)(x,\rho)\dd x\dd s \\
 &+\int_0^t\!\int_{\T^2}
     \sqrt\rho\,\bigl(g+\sqrt\rho\,(V*\rho)\bigr)
     \cdot(\nabla_x\psi)(x,\rho)\dd x\dd s \\
 &+\int_{\T^2\times\R}\chi_{\rho_0}(x,\xi)\psi(x,\xi)
       \dd x\dd\xi.
\end{aligned}
\end{equation*}
In the two integrals over $(0,t)\times\T^2$, all occurrences of $\rho$
and $g$ are evaluated at $(x,s)$, and $\nabla_x\psi$ differentiates before
the velocity variable is evaluated at $\xi=\rho$.  This is
\cite[Definition~2.4, equation~(2.12)]{FG23} with $\Phi(r)=r$ and control
$g+\sqrt\rho\,(V*\rho)$.
\end{definition}

\begin{remark}[Weak and kinetic formulations]
\label{rem:weak-kinetic-equivalence}
The equivalence follows from \cite[Theorem~4.3]{FG23}.  Indeed,
Lemma~\ref{lem:mass-fisher-consequences} gives
$\sqrt\rho\,(V*\rho)\in L^2(\QT;\R^2)$, so
\eqref{eq:skeleton} is their equation with control
$g+\sqrt\rho\,(V*\rho)$.  The additional renormalized products
$\nabla\rho\cdot(V*\rho)$ and
$\rho\,\diver(V*\rho)$ belong to $L^1(\QT)$ by the same lemma.
Consequently, Definitions~\ref{def:entropy-skeleton} and
\ref{def:kinetic-skeleton} are equivalent without any change to the
Fehrman--Gess argument.
\end{remark}

\begin{proposition}[Global well-posedness of the skeleton]
\label{prop:skeleton-gwp}
Let $\rho_0\in L^1(\T^2)$ be nonnegative with finite entropy and mass
$M$, and suppose that
\eqref{eq:subcritical} holds.
For every $g\in L^2(\QT;\R^2)$, equation \eqref{eq:skeleton} has a unique
entropy solution, denoted by $\rho[g]$.  It has a representative in
$C([0,T];\mathcal D'(\T^2))$, attains $\rho_0$ in that topology, and is
equivalently the unique renormalized kinetic solution with the exact
parabolic defect in the class \eqref{eq:skeleton-entropy-class}.

For almost every $t\in[0,T]$,
\begin{equation}\label{eq:skeleton-free-energy-bound}
 \mathcal F_V(\rho[g](t))
 +\frac12\int_0^t\mathcal D_V(\rho[g](s))\dd s
 \leq \mathcal F_V(\rho_0)
 +\frac12\int_0^t\!\int_{\T^2}|g|^2,
\end{equation}
where $\mathcal F_V$ and $\mathcal D_V$ are defined in
\eqref{eq:free-energy-dissipation}.  For every $N<\infty$, we also have
\begin{equation}\label{eq:skeleton-entropy-fisher-bound}
 \sup_{\|g\|_{L^2(\QT)}^2\leq N}
 \left\{
 \operatorname*{ess\,sup}_{t\in[0,T]}\int_{\T^2}H(\rho[g](t))
 +\int_0^T\mathcal D_V(\rho[g](s))\dd s
 +\int_0^T\mathcal I(\rho[g](s))\,\dd s
 \right\}<\infty.
\end{equation}
The bound depends only on $T$, $M$, $\lambda_V$, $\alpha_V$,
$\int_{\T^2}H(\rho_0)$, $N$, and the strict gap in
\eqref{eq:subcritical}.
The strict subcritical condition is needed for the global estimate, but
not for conditional uniqueness once mass conservation and finite
integrated Fisher information are known.
\end{proposition}

\subsection{Structure-preserving existence and energy estimate}

\begin{proposition}[Existence and the free-energy inequality]
\label{prop:skeleton-free-energy}
Under the assumptions of Proposition~\ref{prop:skeleton-gwp}, there exists an
entropy solution satisfying \eqref{eq:skeleton-free-energy-bound} and
\eqref{eq:skeleton-entropy-fisher-bound}.
\end{proposition}

\begin{proof}
If $M=0$, nonnegativity forces $\rho_0=0$, and $\rho=0$ is the only
possible solution.  We henceforth assume $M>0$.  Let $Q_t=e^{t\Delta}$,
choose
\begin{equation*}
 \rho_0^n=Q_{1/n}\rho_0,
 \qquad
 g_n\longrightarrow g\quad\text{in }L^2(\QT;\R^2),
\end{equation*}
with $g_n$ smooth, and define
\begin{equation*}
 V_n=Q_{1/n}V,
 \qquad c_n=(Q_{1/n}\Green)*\rho_n,
 \qquad
 V_n*\rho_n
 =\lambda_V\cos(\alpha_V)\nabla c_n
  +\lambda_V\sin(\alpha_V)\J\nabla c_n.
\end{equation*}
The heat approximation preserves mass, converges strongly in $L^1$, and
converges in entropy.  Moreover, finite entropy and
Lemma~\ref{lem:green-potential-coercivity} imply
$\nabla\Green*\rho_0\in L^2(\T^2)$.  By self-adjointness, the regularized
Green quadratic form is the squared $L^2$ norm of a heat regularization of
this gradient; strong continuity of $Q_t$ on $L^2$ therefore shows that the
corresponding regularized free energy converges to $\mathcal F_V(\rho_0)$.

For fixed $n$, consider
\begin{equation*}
 \partial_t\rho_n
 =\Delta\rho_n-\diver\bigl(\rho_n(V_n*\rho_n)\bigr)
 -\diver\left(\frac{\rho_n}{\sqrt{\rho_n+n^{-1}}}g_n\right),
 \qquad \rho_n(0)=\rho_0^n.
\end{equation*}
After smoothly extending the square-root coefficient and truncating the
density, existence, nonnegativity, mass conservation, and removal of the
truncation are the linear-diffusion specialization of
\cite[Lemma~5.4 and Propositions~5.5--5.7]{FG23}.  The regularized
Coulomb drift is a smooth lower-order term; at fixed $n$, its coefficients
are bounded by the conserved mass, so the standard maximum estimate gives
global continuation.

The first variation of the regularized free energy is, up to a spatial
constant, $\log\rho_n-\lambda_V\cos(\alpha_V)c_n$.  The rotational
component contributes zero since
\begin{equation*}
 \int_{\T^2}\rho_n\J\nabla c_n\cdot
 \nabla\bigl(\log\rho_n
 -\lambda_V\cos(\alpha_V)c_n\bigr)
 =\int_{\T^2}\J\nabla c_n\cdot\nabla\rho_n=0.
\end{equation*}
Testing the equation by this first variation and using the
self-adjointness of $Q_{1/n}$ together with
\begin{equation*}
 \frac{\bigl(r/\sqrt{r+n^{-1}}\bigr)^2}{r}
 =\frac{r}{r+n^{-1}}\leq1\qquad(r>0),
\end{equation*}
gives
\begin{equation}\label{eq:skeleton-approximate-free-energy}
\begin{aligned}
 &\left[
 \int_{\T^2}H(\rho_n)
 -\frac{\lambda_V\cos(\alpha_V)}2
  \int_{\T^2}\rho_nc_n
 \right](t)
 +\frac12\int_0^t\!\int_{\T^2}\rho_n
 \left|\nabla\bigl(\log\rho_n
 -\lambda_V\cos(\alpha_V)c_n\bigr)\right|^2\dd x\dd s\\
 &\qquad\leq
 \left[
 \int_{\T^2}H(\rho_0^n)
 -\frac{\lambda_V\cos(\alpha_V)}2
  \int_{\T^2}\rho_0^n(Q_{1/n}\Green*\rho_0^n)
 \right]
 +\frac12\int_0^t\!\int_{\T^2}|g_n|^2.
\end{aligned}
\end{equation}

The strict-subcritical coercivity is uniform in $n$ because
\begin{equation*}
 \frac12\int_{\T^2}\rho_nc_n
 =\frac12\int_{\T^2}(Q_{1/(2n)}\rho_n)
   \Green*(Q_{1/(2n)}\rho_n),
\end{equation*}
so Lemma~\ref{lem:green-potential-coercivity} applies after entropy
contraction by the heat semigroup.  Thus
\eqref{eq:skeleton-approximate-free-energy} controls the entropy and
physical dissipation uniformly in $n$.  Proposition
\ref{prop:fisher-from-dissipation} also applies: heat contraction gives
$\int\rho_nQ_{1/n}\rho_n\leq\|\rho_n\|_2^2$, and
Lemma~\ref{lem:entropy-sublevel-interpolation} absorbs this term under the
strict subcritical gap.  This gives all three bounds in
\eqref{eq:skeleton-entropy-fisher-bound}, with a generally nonlinear
dependence on the entropy level.

Compactness and removal of the square-root approximation follow the
argument in \cite[the proof of Proposition~5.7,
equations~(5.29)--(5.37)]{FG23}.  The only additional term in the
time-compactness estimate, $\rho_n(V_n*\rho_n)$, is bounded in
$L^2(0,T;L^1(\T^2))$ by
Lemma~\ref{lem:mass-fisher-consequences} and heat contraction.  Hence,
along a subsequence,
\begin{equation*}
 \rho_n\longrightarrow\rho\quad\text{in }L^1(\QT),
 \qquad
 \sqrt{\rho_n}\longrightarrow\sqrt\rho
 \quad\text{in }L^2(\QT),
 \qquad
 \frac{\rho_n}{\sqrt{\rho_n+n^{-1}}}
 \longrightarrow\sqrt\rho
 \quad\text{in }L^2(\QT).
\end{equation*}
The diffusion and control terms converge as in that proof.  At fixed
velocity cutoffs, the Coulomb terms converge by
\cite[Theorem~3.1, Step~1, equations~(3.20)--(3.23)]{JSW26}; in the weak
formulation without velocity cutoffs,
Lemma~\ref{lem:interaction-compactness} gives
\begin{equation*}
 \rho_n(V_n*\rho_n)
 \longrightarrow\rho(V*\rho)
 \quad\text{in }\mathcal D'(\QT;\R^2).
\end{equation*}
This has the same role as the kernel limit in
\cite[proof of Proposition~3.16, especially equations~(3.74)--(3.79)]{WZ24},
but the distributional-flux argument is different: it uses odd-kernel
symmetrization at the weak-$L^2$ Coulomb endpoint, rather than the $L^p$
approximation and H\"older--Young estimates used there.
Together these limits prove \eqref{eq:skeleton-weak}.  Finally,
Lemma~\ref{lem:physical-flux-lower-semicontinuity} and
Lemma~\ref{lem:green-potential-coercivity} allow us to pass to the limit in
\eqref{eq:skeleton-approximate-free-energy}, giving
\eqref{eq:skeleton-free-energy-bound}.  The time-compactness argument in
\cite[Proposition~5.7]{FG23} preserves the initial trace and mass and
gives a representative in $C([0,T];\mathcal D'(\T^2))$; lower
semicontinuity preserves the entropy and Fisher information.  This also
proves the uniform assertion \eqref{eq:skeleton-entropy-fisher-bound}.
\end{proof}

The entropy and Fisher-information bounds established above allow us to
apply the kinetic comparison argument.  The next section proves the
resulting $L^1$ estimate for both equations.

\section{Kinetic Coulomb comparison, uniqueness, and stability}
\label{sec:comparison}

We prove an $L^1$ comparison estimate for the stochastic equation and its
skeleton.  The Coulomb contribution is controlled by the logarithmic
estimate of Section~\ref{sec:analytic}; the Osgood lemma then yields
uniqueness and continuous dependence.

\subsection{The Coulomb contribution to the doubled kinetic comparison}

\begin{proposition}[Coulomb perturbation of the kinetic comparison]
\label{prop:coulomb-comparison}
Let $g\in L^2(Q_T;\R^2)$, and let $\rho^1,\rho^2$ be two renormalized
kinetic skeleton solutions in the sense of
Definition~\ref{def:kinetic-skeleton}, with the same control and equal
mass $M$.  Define
\begin{equation}\label{eq:L1-discrepancy}
 \delta(t)=\|\rho^1(t)-\rho^2(t)\|_{L^1(\T^2)},
 \qquad
 \delta(0)=\|\rho_0^1-\rho_0^2\|_{L^1(\T^2)}.
\end{equation}
Then, for almost every $t\in(0,T)$,
\begin{equation}\label{eq:shared-L1-comparison}
\begin{aligned}
 \delta(t)\leq \delta(0)+C_{V,M}\int_0^t
 \left(1+\sum_{i=1}^2\mathcal I(\rho^i(s))\right)\delta(s)
 \log\!\left(\frac{e(1+2M)}{\delta(s)}\right)\dd s.
\end{aligned}
\end{equation}
The integrand is zero when $\delta(s)=0$.
Moreover, for two stochastic kinetic solutions driven by the same noise
and control, the Coulomb part of the doubled identity has the same limit on
every common canonical energy stop.  Thus
after this Coulomb calculation, the remaining stochastic terms are the
common-noise and kinetic-measure terms of
\cite[Theorem~4.7]{FG24}.
\end{proposition}

\begin{proof}
If $M=0$, nonnegativity and mass conservation imply that both solutions
vanish.  Suppose henceforth that $M>0$.

We use the doubled-variable calculation of
\cite[equations~(3.10)--(3.14)]{FG23}, recording its normalization because
the Coulomb perturbation depends on it.  Choose nonnegative, even,
unit-mass functions
$\varrho\in C_c^\infty(\R^2)$ and
$\psi\in C_c^\infty((-1,1))$, with $\varrho$ supported in the unit ball.
Periodize the spatial mollifier and define
\begin{equation}\label{eq:comparison-mollifiers}
 \begin{aligned}
 \varrho_\nu(z)&=\nu^{-2}\varrho(z/\nu),
 &\psi_\gamma(r)&=\gamma^{-1}\psi(r/\gamma),\\
 \chi^i_{\nu,\gamma}(t,y,\eta)
 &=\int_{\T^2\times\R}
 \chi_{\rho^i}(t,x,\xi)\varrho_\nu(x-y)
 \psi_\gamma(\xi-\eta)\dd x\dd\xi,
 &i&=1,2.
 \end{aligned}
\end{equation}
For $L>1$, choose $\zeta_L\in C_c^\infty(0,\infty)$, taking values
between zero and one, such that
\begin{equation}\label{eq:comparison-level-cutoff}
 \zeta_L=1\ \text{on }[2/L,L],\qquad
 \zeta_L=0\ \text{outside }[1/L,L+1],\qquad
 |\zeta_L'|\lesssim
 \begin{cases}
 L,&\eta\in[1/L,2/L],\\
 1,&\eta\in[L,L+1].
 \end{cases}
\end{equation}
With \eqref{eq:comparison-mollifiers} and
\eqref{eq:comparison-level-cutoff}, the doubled functional is
\begin{equation}\label{eq:doubled-kinetic-functional}
 \int_{\T^2\times\R}
 \left(
 \chi^1_{\nu,\gamma}+\chi^2_{\nu,\gamma}
 -2\chi^1_{\nu,\gamma}\chi^2_{\nu,\gamma}
 \right)(t,y,\eta)\zeta_L(\eta)\dd y\dd\eta.
\end{equation}
It converges to $\delta(t)$ from \eqref{eq:L1-discrepancy} when the spatial scale
tends to zero, then the
velocity scale tends to zero, and finally the level cutoff tends to
infinity.

For $0<\gamma<(2L)^{-1}$, the complete additional contribution of the
Coulomb flux to the increment of \eqref{eq:doubled-kinetic-functional}
between times zero and $t$ is
\begin{equation}\label{eq:coulomb-prelimit-contribution}
\begin{aligned}
 &\int_0^t\!\int_{(\T^2)^2\times\R}
 \diver\!\left(\rho^1(s,x)(V*\rho^1)(s,x)\right)
 \varrho_\nu(x-y)\psi_\gamma(\rho^1(s,x)-\eta)\\
 &\hspace{38mm}\times
 \left(2\chi^2_{\nu,\gamma}(s,y,\eta)-1\right)
 \zeta_L(\eta)\dd x\dd y\dd\eta\dd s\\
 &+\int_0^t\!\int_{(\T^2)^2\times\R}
 \diver\!\left(\rho^2(s,x)(V*\rho^2)(s,x)\right)
 \varrho_\nu(x-y)\psi_\gamma(\rho^2(s,x)-\eta)\\
 &\hspace{38mm}\times
 \left(2\chi^1_{\nu,\gamma}(s,y,\eta)-1\right)
 \zeta_L(\eta)\dd x\dd y\dd\eta\dd s.
\end{aligned}
\end{equation}
Apart from \eqref{eq:coulomb-prelimit-contribution}, this is precisely the
doubled identity of \cite[Theorem~3.3]{FG23}.  Its hypotheses follow from
\eqref{eq:skeleton-entropy-class} and $g\in L^2(Q_T;\R^2)$.  In particular,
the local terms have a nonpositive limit and the common control leaves no
positive contribution.  For the common-control term, the translation error
is bounded by
\begin{equation*}
 \sum_{i=1}^2
 \left(\int_0^T\mathcal I(\rho^i(s))\,\dd s\right)^{1/2}
 \left(
 \int_{\T^2}\varrho_\nu(h)
 \|g(\mathord\cdot,\mathord\cdot+h)-g\|_{L^2(Q_T)}^2\dd h
 \right)^{1/2},
\end{equation*}
which tends to zero as $\nu\downarrow0$, since translations are continuous
in $L^2(Q_T)$.

We pass to the limit in \eqref{eq:coulomb-prelimit-contribution} directly.
Lemma~\ref{lem:mass-fisher-consequences} gives
$\nabla\rho^i\in L^{4/3}(Q_T)$, $V*\rho^i\in L^4(Q_T)$, and
$\rho^i,\diver(V*\rho^i)\in L^2(Q_T)$.  Spatial approximation in these
spaces justifies the product rule
\begin{equation*}
 \diver\bigl(\rho^i(V*\rho^i)\bigr)
 =\nabla\rho^i\cdot(V*\rho^i)
  +\rho^i\diver(V*\rho^i)\in L^1(Q_T).
\end{equation*}
For fixed $L$ and $\gamma$, the spatially convolved kinetic factors
converge in measure on the compact velocity support and are bounded by
one.  Thus, as
$\nu\downarrow0$, \eqref{eq:coulomb-prelimit-contribution} converges to
\begin{equation*}
\begin{aligned}
 \sum_{\substack{i,j\in\{1,2\}\\i\ne j}}
 \int_0^t\!\int_{\T^2}
 &\diver\bigl(\rho^i(V*\rho^i)\bigr)
 \int_{\R}\psi_\gamma(\rho^i-\eta)\\
 &\quad\times\left(2\int_0^{\rho^j}
 \psi_\gamma(\xi-\eta)\dd\xi-1\right)
 \zeta_L(\eta)\dd\eta\dd x\dd s.
\end{aligned}
\end{equation*}
Here both densities in the integrand are evaluated at $(s,x)$.  The
velocity integral has absolute value at most one.  The convergence
therefore also holds after integration against the absolute values of
the integrable divergences.

Extend $\zeta_L$ by zero to $\R$.  For $a,b\geq0$, we claim that
\begin{equation}\label{eq:coulomb-velocity-pairing-limit}
 \int_{\R}\psi_\gamma(a-\eta)
 \left(2\int_0^b\psi_\gamma(\xi-\eta)\dd\xi-1\right)
 \zeta_L(\eta)\dd\eta
 \longrightarrow\zeta_L(a)\operatorname{sgn}(b-a),
\end{equation}
where $\operatorname{sgn}(0)=0$.  For $a>0$ and $a\ne b$, this follows
from the supports of the mollifiers.  When $a=b>0$ and $2\gamma<a$,
evenness and unit mass give
\begin{equation*}
 \int_{\R}\psi_\gamma(a-\eta)
 \int_0^a\psi_\gamma(\xi-\eta)\dd\xi\dd\eta
 =\int_0^a(\psi_\gamma*\psi_\gamma)(a-\xi)\dd\xi
 =\frac12.
\end{equation*}
Replacing $\zeta_L(\eta)$ by $\zeta_L(a)$ therefore makes the integral
in \eqref{eq:coulomb-velocity-pairing-limit} zero; the replacement error
is at most $\gamma\|\zeta_L'\|_\infty$.  If $a=0$, the integral is zero
for $\gamma<(2L)^{-1}$ because the supports of
$\psi_\gamma(-\mathord\cdot)$ and $\zeta_L$ are disjoint.  This proves
\eqref{eq:coulomb-velocity-pairing-limit}, including the diagonal and
vacuum cases.

Applying this limit with $a=\rho^i$ and $b=\rho^j$, dominated convergence
removes the velocity mollifier.  It also removes the level cutoff from
the Coulomb terms as $L\to\infty$: $\zeta_L(\rho^i)\to1$ on
$\{\rho^i>0\}$, while
$\diver(\rho^i(V*\rho^i))=0$ almost everywhere on $\{\rho^i=0\}$ by
$\nabla\rho^i=2\sqrt{\rho^i}\nabla\sqrt{\rho^i}$ and the product rule
above.

For the remaining diffusion and common-control terms, the proof of
\cite[Theorem~3.3, after equation~(3.14)]{FG23} removes the lower cutoff by
dominated convergence and bounds the upper-level terms by
\begin{equation}\label{eq:comparison-high-level-bound}
 (L+1)\sum_{i=1}^2\int_{(0,t)\times\T^2}
 \mathbf1_{\{L\leq\rho^i<L+1\}}
 \left(|g|^2+|\nabla\sqrt{\rho^i}|^2\right).
\end{equation}
Its disjoint-strip argument gives a single sequence $L_n\uparrow\infty$
along which \eqref{eq:comparison-high-level-bound} vanishes for both
solutions.  The exact defects in Definition~\ref{def:kinetic-skeleton}
require no separate tail condition.  Since the Coulomb terms converge
for the full limit $L\to\infty$, they also converge along this sequence.

For stochastic kinetic solutions, stop at the minimum of the canonical
times in \eqref{eq:kinetic-energy-stops}.  The condition
\eqref{eq:stopped-high-velocity-tail} replaces the exact-defect observation,
and \cite[Theorem~4.7]{FG24} treats the unchanged common-noise and
kinetic-measure terms.  The same Coulomb limit argument applies pathwise.
The stopped Fisher bound in
Lemma~\ref{lem:canonical-energy-localization}, together with
Lemma~\ref{lem:mass-fisher-consequences}, makes the dominating divergences
integrable on $\Omega\times(0,T)\times\T^2$ up to each common stop, so
these limits also hold in $L^1(\Omega)$.  After removing
the velocity cutoffs, Lemma~\ref{lem:canonical-energy-localization} lets the
stops increase to $T$ almost surely.

Consequently, the limit order is
\begin{equation}\label{eq:comparison-limit-order}
 \nu\downarrow0,\qquad
 \gamma\downarrow0,\qquad
 L_n\uparrow\infty.
\end{equation}
With the limit order in \eqref{eq:comparison-limit-order}, the Coulomb contribution
in \eqref{eq:coulomb-prelimit-contribution} converges to
\begin{equation*}
 -\int_0^t\!\int_{\T^2}
 \operatorname{sgn}(\rho^1-\rho^2)
 \diver\!\left(
 \rho^1(V*\rho^1)-\rho^2(V*\rho^2)
 \right).
\end{equation*}
Decompose the interaction flux as
\begin{equation}\label{eq:coulomb-flux-decomposition}
 \rho^1(V*\rho^1)-\rho^2(V*\rho^2)
 = (\rho^1-\rho^2)(V*\rho^1)
 +\rho^2\bigl(V*(\rho^1-\rho^2)\bigr).
\end{equation}
The first term in \eqref{eq:coulomb-flux-decomposition} is conservative in
the limiting Kato identity.  To justify the sign chain rule, let
$s_\ell\in C^2(\R)$ be convex approximations of the absolute value such that
$|s_\ell'|\leq1$ and
$|r s_\ell'(r)-s_\ell(r)|\lesssim\ell^{-1}$.  Periodicity gives
\begin{align*}
 &-\int_{\T^2}s_\ell'(\rho^1-\rho^2)
 \diver\!\left((\rho^1-\rho^2)(V*\rho^1)\right)\\
 &\qquad=-\int_{\T^2}
 \left((\rho^1-\rho^2)s_\ell'(\rho^1-\rho^2)
 -s_\ell(\rho^1-\rho^2)\right)\diver(V*\rho^1).
\end{align*}
The right-hand side tends to zero by
\eqref{eq:coulomb-identities} and $\rho^1\in L^1(\T^2)$.  Letting
$\ell\to\infty$ therefore yields
\begin{equation*}
 -\int_{\T^2}\operatorname{sgn}(\rho^1-\rho^2)
 \diver\!\left((\rho^1-\rho^2)(V*\rho^1)\right)
 =-\int_{\T^2}\diver\!\left(
 |\rho^1-\rho^2|(V*\rho^1)\right)=0.
\end{equation*}
Since the masses agree, \eqref{eq:coulomb-identities} shows that the
remaining contribution is
\begin{equation}\label{eq:limiting-coulomb-kato-term}
\begin{aligned}
 \int_0^t\!\int_{\T^2}
 &\operatorname{sgn}(\rho^2-\rho^1)
 \nabla\rho^2\cdot\bigl(V*(\rho^1-\rho^2)\bigr)\\
 &+\lambda_V\cos(\alpha_V)\rho^2
 |\rho^1-\rho^2|.
\end{aligned}
\end{equation}
For $\cos(\alpha_V)\leq0$, the second term is nonpositive.  In every
case, \eqref{eq:limiting-coulomb-kato-term} is bounded above by
\begin{equation*}
 \int_0^t\!\int_{\T^2}
 |\nabla\rho^2|\,\bigl|V*(\rho^1-\rho^2)\bigr|
 +\lambda_V|\cos(\alpha_V)|
 \int_0^t\!\int_{\T^2}
 \rho^2|\rho^1-\rho^2|.
\end{equation*}
The combined contribution of the diffusion, defect measure, and common
control is nonpositive in the limit, so
Lemma~\ref{lem:log-interaction} proves \eqref{eq:shared-L1-comparison}.
\end{proof}

\subsection{Stochastic uniqueness and the Borel solution map}

\begin{proposition}[Pathwise stochastic Coulomb comparison]
\label{prop:stochastic-coulomb-comparison}
Let $\rho^1$ and $\rho^2$ be stochastic kinetic solutions driven
by the same $W_K$ and the same predictable control $g$, and suppose that
they have equal mass $M$.  Lemma~\ref{lem:canonical-energy-localization}
gives, almost surely,
\begin{equation*}
 \int_0^T\sum_{i=1}^2\mathcal I(\rho^i(s))\,\dd s<\infty.
\end{equation*}
Outside one null set, define
$\delta(t)=\|\rho^1(t)-\rho^2(t)\|_{L^1(\T^2)}$ and
$\delta(0)=\|\rho_0^1-\rho_0^2\|_{L^1(\T^2)}$.  Then, for every $t\in[0,T]$,
\begin{equation}\label{eq:spde-equal-mass-dependence}
 \delta(t)\leq \delta(0)+C_{V,M}\int_0^t
 \left(1+\sum_{i=1}^2\mathcal I(\rho^i(s))\right)
 \delta(s)\log\!\left(\frac{e(1+2M)}{\delta(s)}\right)\,\dd s.
\end{equation}
The integrand is defined to be zero when $\delta(s)=0$.
\end{proposition}

\begin{proof}
For each $m$, stop at the minimum of the canonical energy times in
\eqref{eq:kinetic-energy-stops}.  The stopped bounds justify optional
sampling.  The unchanged diffusion, kinetic-measure, common-noise, and
common-control limits, including the common strip sequence, are precisely
\cite[equations~(4.15)--(4.28) and the proof of Theorem~4.7]{FG24} and
\cite[equations~(3.10)--(3.14) and the proof of Theorem~3.3]{FG23}; their
tail hypotheses are \eqref{eq:stopped-low-velocity-tail} and
\eqref{eq:stopped-high-velocity-tail}.

On the stopped interval, Proposition~\ref{prop:coulomb-comparison} applies
pathwise by \eqref{eq:coulomb-interaction-integrability}.  Its full cutoff
limits require no additional strip sequence.  The order
\eqref{eq:comparison-limit-order}, followed by
Lemma~\ref{lem:log-interaction}, gives
\eqref{eq:spde-equal-mass-dependence} first at rational times.  Pathwise
$L^1(\T^2)$ continuity and the increase of the stops to $T$ complete the
argument.
\end{proof}

\begin{proof}[Completion of the proof of Theorems~\ref{thm:spde-gwp}
and~\ref{thm:main-spde}\textup{(i)--(ii)}]
Section~\ref{sec:spde} gives unforced weak existence, the pathwise
estimates, and the canonical kinetic-measure tails.
Lemma~\ref{lem:shared-osgood}, applied
pathwise to Proposition~\ref{prop:stochastic-coulomb-comparison}, gives
pathwise uniqueness and the stated continuous dependence for equal masses,
both for the unforced equation and, conditionally on existence, for any
common predictable control.

Applying the Gy\"ongy--Krylov argument and measurable factorization as in
\cite[Theorem~5.29 and Corollary~5.31]{FG24} and
\cite[Theorem~6.1, equation~(6.6)]{FG23}, with the velocity-localized Coulomb
convergence supplied by \eqref{eq:renormalized-coulomb-compactness} in
Lemma~\ref{lem:interaction-compactness}, gives, for fixed $\rho_0$,
$\varepsilon>0$, and $K<\infty$, an $L^1(Q_T)$-valued Borel function of
the finite-dimensional driving path representing the unforced solution.

For bounded controls, \eqref{eq:finite-mode-novikov} and
\eqref{eq:finite-mode-cameron-martin-shift} identify the shifted map as in
\cite[Proposition~6.2]{FG23}.  Under an equivalent probability measure,
\begin{equation*}
 W_K+\varepsilon^{-1/2}\int_0^\cdot P_Kg(s)\,\dd s
\end{equation*}
is a Brownian noise with the original covariance.  Evaluating the
unforced solution map at \eqref{eq:finite-mode-cameron-martin-shift}
therefore gives an unforced kinetic solution under that measure.
Equivalence preserves continuity, mass, pathwise physical energy, and
the parabolic lower bound.  Rewriting the compact-velocity semimartingale
identity under the original measure yields exactly the control terms in
\eqref{eq:controlled-stochastic-kinetic-identity}.

The expected velocity-tail bounds require a separate verification under
the original measure.  The fixed-shell identity used in the proof of
Proposition~\ref{prop:approximation-high-velocity-tail} transfers as a
semimartingale identity.  Its stopped estimates, now with the control
term, give \eqref{eq:stopped-high-velocity-tail} under the original
measure; they use only the already established pathwise energy and
parabolic lower bound.  Lemma~\ref{lem:canonical-energy-localization}
then gives the low-velocity tail.  Thus the shifted solution belongs to
Definition~\ref{def:stochastic-kinetic-solution}, and pathwise uniqueness
identifies it as the controlled solution.

To retain the constants in the controlled estimates, evaluate the
regularized solution maps at the same shifted path.  For each fixed
$(\varepsilon,K,g)$, unforced compactness and pathwise uniqueness give
convergence in probability under the equivalent measure, hence also
under the original measure.  The controlled approximation bounds of
Section~\ref{sec:spde}, applied under the original measure, pass to this
limit by Lemmas~\ref{lem:entropy-tail-transfer} and
\ref{lem:physical-flux-lower-semicontinuity} and Fatou's lemma.  This
proves the controlled energy estimates with their stated constants,
including the bounds used in the small-noise limit.  No uniform
comparison of the two probability measures is needed.

Finally, for $N_1<N_2$, the controls stopped at the times in
\eqref{eq:control-action-stop}, and hence their solutions, agree up to
$\gamma_{N_1}$.  Pasting and letting $\gamma_N\uparrow T$ gives the
solution for every predictable control with almost surely finite action
and preserves \eqref{eq:spde-pathwise-estimates}.  This completes the
controlled well-posedness and Borel solution-map assertions.
\end{proof}

\subsection{Skeleton comparison and uniqueness}

The comparison calculation is common to the deterministic and stochastic
equations.  Proposition~\ref{prop:coulomb-comparison} contains every term
created by the Coulomb flux; all local diffusion and common-control terms
are those of \cite[Theorem~3.3]{FG23}.

\begin{proposition}[\texorpdfstring{$L^1$}{L1} comparison]
\label{prop:skeleton-comparison}
Let $\rho^1$ and $\rho^2$ be exact-defect kinetic solutions of
\eqref{eq:skeleton} with the same control, equal mass $M$, and finite
integrated Fisher information.  With $\delta$ defined by
\eqref{eq:L1-discrepancy}, estimate \eqref{eq:shared-L1-comparison} holds for
almost every $t\in(0,T)$.
For $\delta(0)>0$, Lemma~\ref{lem:shared-osgood} also gives the explicit
continuous-dependence estimate \eqref{eq:osgood-continuous-dependence}.
\end{proposition}

\begin{proof}
The assertion is immediate for $M=0$.  Otherwise, the common control is
square integrable, the two Fisher integrals are finite, and the defects are
the exact measures required in Definition~\ref{def:kinetic-skeleton}.
Hence all hypotheses of
Proposition~\ref{prop:coulomb-comparison} hold, and
\eqref{eq:shared-L1-comparison} applies.  In particular, no boundedness of
either density is required.
\end{proof}

\begin{proof}[Proof of Proposition~\ref{prop:skeleton-gwp}]
Proposition~\ref{prop:skeleton-free-energy} gives a global entropy
solution, and Remark~\ref{rem:weak-kinetic-equivalence} places every
entropy solution in the exact-defect kinetic class.  For two solutions
with the same initial datum and the same control,
Proposition~\ref{prop:skeleton-comparison} has $\delta(0)=0$.
Lemma~\ref{lem:shared-osgood} yields $\delta=0$ almost everywhere.  This
argument uses only common mass and finite integrated Fisher information;
in particular, it does not use \eqref{eq:subcritical}.
\end{proof}

\subsection{Weak-to-strong stability and compactness}

\begin{theorem}[Weak-to-strong stability of the skeleton map]
\label{thm:skeleton-stability}
Fix a nonnegative $\rho_0\in L^1(\T^2)$ with finite entropy satisfying
\eqref{eq:subcritical}.  If
\begin{equation*}
 g_n\weakto g
 \quad\text{weakly in }L^2(\QT;\R^2),
\end{equation*}
then
\begin{equation}\label{eq:skeleton-stability}
 \rho[g_n]\longrightarrow\rho[g]
 \quad\text{strongly in }L^1(\QT).
\end{equation}
\end{theorem}

\begin{proof}
This is the compactness--identification argument of
\cite[Proposition~5.8 and its proof]{FG23}; we give the two limit arguments
that change in the presence of the interaction.  Since weakly convergent
controls are bounded,
\eqref{eq:skeleton-entropy-fisher-bound} gives uniform entropy,
dissipation, and Fisher-information estimates.  The equation and
\eqref{eq:coulomb-flux-integrability} give the time-derivative estimate
\begin{equation*}
\begin{aligned}
 \|\partial_t\rho[g_n]\|_{L^2(0,T;W^{-1,1})}
 &\lesssim\sqrt M\bigl(
 2\|\nabla\sqrt{\rho[g_n]}\|_{L^2(Q_T)}\\
 &\qquad\qquad
 +\|\sqrt{\rho[g_n]}(V*\rho[g_n])\|_{L^2(Q_T)}
 +\|g_n\|_{L^2(Q_T)}\bigr).
\end{aligned}
\end{equation*}
The right-hand side is bounded uniformly in $n$.  Here we used
H\"older's inequality in space and mass conservation for each of the
three fluxes.
Since
$\nabla\rho[g_n]=2\sqrt{\rho[g_n]}\nabla\sqrt{\rho[g_n]}$,
mass conservation and the uniform Fisher-information bound imply
boundedness in $L^2(0,T;W^{1,1}(\T^2))$.  The compact embedding
$W^{1,1}(\T^2)\Subset L^1(\T^2)$, the continuous embedding
$L^1(\T^2)\hookrightarrow W^{-1,1}(\T^2)$, and
\cite[Corollary~4]{Sim87} therefore give, along a subsequence,
\begin{equation*}
 \rho[g_n]\longrightarrow\bar\rho\quad\text{in }L^1(\QT),
 \qquad
 \sqrt{\rho[g_n]}\longrightarrow\sqrt{\bar\rho}
 \quad\text{in }L^2(\QT).
\end{equation*}
Here the second convergence follows from
$\|\sqrt u-\sqrt v\|_2^2\leq\|u-v\|_1$.
The uniform Fisher-information bound also gives, after extraction,
$\nabla\sqrt{\rho[g_n]}\rightharpoonup\nabla\sqrt{\bar\rho}$ in
$L^2(\QT)$, which will be used through lower semicontinuity below.

For every smooth $\varphi$, the control term converges because
\begin{equation}\label{eq:skeleton-control-limit}
\begin{aligned}
 &\int_{\QT}g_n\cdot\sqrt{\rho[g_n]}\nabla\varphi
 -\int_{\QT}g\cdot\sqrt{\bar\rho}\nabla\varphi\\
 &\quad=
 \int_{\QT}g_n\cdot
 \bigl(\sqrt{\rho[g_n]}-\sqrt{\bar\rho}\bigr)\nabla\varphi
 +\int_{\QT}(g_n-g)\cdot\sqrt{\bar\rho}\nabla\varphi
 \longrightarrow0.
\end{aligned}
\end{equation}
The first term tends to zero by strong $L^2$ convergence and boundedness
of $g_n$, and the second by weak convergence.  The other new passage is
\begin{equation}\label{eq:skeleton-coulomb-limit}
 \int_{\QT}\rho[g_n]\bigl(V*(\rho[g_n])\bigr)\cdot\nabla\varphi
 \longrightarrow
 \int_{\QT}\bar\rho(V*\bar\rho)\cdot\nabla\varphi,
\end{equation}
which is precisely the distributional closure in
Lemma~\ref{lem:interaction-compactness}.

For every smooth spatial test $\varphi$ and $0\leq s\leq t\leq T$,
the $L^2$ time-derivative bound gives
\begin{equation*}
 \sup_n\left|\langle\rho[g_n](t)-\rho[g_n](s),\varphi\rangle\right|
 \lesssim_\varphi |t-s|^{1/2}.
\end{equation*}
Arzel\`a--Ascoli on a countable dense family of smooth tests, together
with the uniform mass bound, gives convergence along a further
subsequence in $C([0,T];\mathcal D'(\T^2))$.  Strong space--time $L^1$
convergence identifies the limit with $\bar\rho$, so its value at zero
is $\rho_0$.  Testing with the spatial constant preserves mass, and
lower semicontinuity preserves the essential-supremum entropy bound and
integrated Fisher information.  Equations
\eqref{eq:skeleton-control-limit} and \eqref{eq:skeleton-coulomb-limit}
therefore identify $\bar\rho$ as an entropy solution with control $g$.
By Remark~\ref{rem:weak-kinetic-equivalence} and
Proposition~\ref{prop:skeleton-gwp}, it equals $\rho[g]$.  Since every
subsequence has the same limit, \eqref{eq:skeleton-stability} follows.
\end{proof}

\begin{proof}[Proof of Proposition~\ref{prop:main-skeleton}]
Proposition~\ref{prop:skeleton-gwp} proves global well-posedness and the
equivalence of the entropy and renormalized kinetic formulations, while
Theorem~\ref{thm:skeleton-stability} gives the asserted weak-to-strong
continuity.  Finally, every closed bounded ball of $L^2(\QT;\R^2)$ is
weakly compact and metrizable.  Its continuous image under
$g\mapsto\rho[g]$ is therefore compact in $L^1(\QT)$, which proves
\eqref{eq:main-skeleton-compactness}.
\end{proof}

Thus the finite-noise equation has a Borel solution map, while the skeleton
map is single valued and continuous from weak $L^2$ controls to strong
$L^1$ trajectories.  These are the two solution maps used in the
weak-convergence argument of Section~\ref{sec:full-ldp}.

\section{Controlled convergence and the large deviation principle}
\label{sec:full-ldp}

We verify the weak-convergence criterion as in the proof of
\cite[Theorem~6.8]{FG23}.  Compactness of the skeleton images is already
given by Proposition~\ref{prop:main-skeleton}; it remains to establish
convergence of the controlled stochastic equations.  Throughout this
section the initial datum is fixed, \eqref{eq:subcritical} holds, and
$K(\varepsilon)\to\infty$ with
$\varepsilon N_{K(\varepsilon)}\to0$.

For $r>1$, we use the following driving, control, and state spaces:
\begin{equation}\label{eq:ldp-spaces}
 \mathcal W=C([0,T];H^{-r}(\T^2;\R^2)),
 \qquad \mathcal H=L^2(Q_T;\R^2),
 \qquad \cX=L^1(Q_T).
\end{equation}
Realize the cylindrical Wiener process $W$ in $\mathcal W$, with
$P_KW=W_K$, and define
\begin{equation}\label{eq:ldp-objects}
\begin{gathered}
 S_N=\{g\in\mathcal H:\|g\|_{\mathcal H}^2\leq N\},\qquad
 \mathcal A_N=\{g\text{ predictable}:g\in S_N\quad\Pp\text{-a.s.}\},\\
 \cG_{\rho_0}^{\varepsilon,K}:\mathcal W\to\cX,
 \qquad \cG_{\rho_0}^{0}:\mathcal H\to\cX,
 \qquad \cG_{\rho_0}^{0}(g)=\rho[g],\\
 \rho^{\varepsilon,K,g}
 =\cG_{\rho_0}^{\varepsilon,K}
 \left(\sqrt\varepsilon W+\int_0^{\mathord\cdot}g(s)\,\dd s\right).
\end{gathered}
\end{equation}
Equip $S_N$ with its compact metrizable weak topology.
Theorem~\ref{thm:main-spde}\textup{(ii)} provides the Borel map in
\eqref{eq:ldp-objects}, which depends only on the projected driving path.
Its Cameron--Martin shift solves
\eqref{eq:controlled-spde}, as in \cite[Proposition~6.2]{FG23}.
The zero-noise map is continuous on $S_N$ by
Theorem~\ref{thm:skeleton-stability}.  We use the rate function
$I_{\rho_0}$ defined in \eqref{eq:rate-intro}.

\subsection{Controlled estimates and compactness}

For the estimates below, write
\begin{equation}\label{eq:controlled-density-family}
 \rho^\varepsilon=\rho^{\varepsilon,K(\varepsilon),g^\varepsilon},
 \qquad g^\varepsilon\in\mathcal A_N.
\end{equation}
The entropy and dissipation bounds of Section~\ref{sec:spde} imply
\begin{equation}\label{eq:uniform-small-noise-energy}
 \limsup_{\varepsilon\downarrow0}\sup_{g^\varepsilon\in\mathcal A_N}
 \E\left[\sup_{t\leq T}\int_{\T^2}H(\rho^\varepsilon(t))
 +\int_0^T\bigl(\mathcal D_V(\rho^\varepsilon(t))
 +\mathcal I(\rho^\varepsilon(t))\bigr)\,\dd t\right]<\infty.
\end{equation}
Indeed, in the attractive case fix $L\geq L_0$ in
\eqref{eq:dyadic-entropy-localization}.  Since
$\varepsilon F_{1,K(\varepsilon)}L\to0$, the deterministic coefficients
in \eqref{eq:entropy-dissipation-tail} remain bounded, and its
exponential tail gives every fixed exponential moment of the entropy
maximum for sufficiently small $\varepsilon$, uniformly over
$\mathcal A_N$.  Equation~\eqref{eq:global-entropy-dissipation-estimate}
controls the physical dissipation.  These bounds pass from the compatible
approximations by Lemmas~\ref{lem:physical-flux-lower-semicontinuity}
and~\ref{lem:entropy-tail-transfer}.  Integration of
\eqref{eq:full-fisher-physical-dissipation} then gives a first-moment bound
for the time-integrated Fisher information.  In the nonattractive case, the entropy estimate of
\cite[Proposition~3.3 and Lemma~4.2]{JSW26}, with the control included
by Young's inequality, gives the entropy and Fisher bounds;
\eqref{eq:mass-fisher-L2}, \eqref{eq:weighted-potential-gradient}, and
the expansion of $\mathcal D_V$ give
$\mathcal D_V(\rho)\lesssim_{V,M}1+\mathcal I(\rho)$.
Thus
\begin{equation}\label{eq:controlled-probability-tail}
 \lim_{R\to\infty}\limsup_{\varepsilon\downarrow0}
 \sup_{g^\varepsilon\in\mathcal A_N}
 \Pp\left(\sup_{t\leq T}\int_{\T^2}H(\rho^\varepsilon(t))
 +\int_0^T\mathcal I(\rho^\varepsilon(t))\,\dd t>R\right)=0.
\end{equation}
Together with the physical-dissipation bound, this proves
Theorem~\ref{thm:main-spde}\textup{(iii)}.  Moreover,
Proposition~\ref{prop:approximation-high-velocity-tail} applies at every
sufficiently small $\varepsilon$: its global high-velocity condition
\eqref{eq:global-high-velocity-tail} is available for the kinetic measure
$q^\varepsilon$ of $\rho^\varepsilon$.

We use the compactness argument of
\cite[Proposition~6.4 and the proof of Theorem~6.6, following (6.15)]{FG23}
on entropy--Fisher levels.  The density-cutoff version needed here is
Lemma~\ref{lem:cutoff-time-compactness}, based on
\cite[Propositions~5.14 and~5.26]{FG24}.  We record the two additional
Coulomb bounds.  First, Lemma~\ref{lem:mass-fisher-consequences} gives,
on the level $R$ and for $0<h<T$,
\begin{equation}\label{eq:coulomb-time-translation}
\begin{aligned}
 &\int_0^{T-h}\left\|\int_t^{t+h}
 \diver\bigl(\rho^\varepsilon(V*\rho^\varepsilon)\bigr)(s)\,\dd s
 \right\|_{W^{-1,1}}\dd t\\
 &\qquad\leq h\int_{Q_T}|\rho^\varepsilon(V*\rho^\varepsilon)|
 \lesssim_{V,M,T}h\left(1+\int_0^T\mathcal I(\rho^\varepsilon)\,\dd t\right)
 \lesssim_{V,M,T,R}h.
\end{aligned}
\end{equation}
Second, the kinetic measures satisfy the weighted bound
\begin{equation}\label{eq:uniform-weighted-kinetic-mass}
 \limsup_{\varepsilon\downarrow0}\sup_{g^\varepsilon\in\mathcal A_N}
 \E\int_{\T^2\times(0,\infty)\times[0,T]}
 \frac1\xi\,q^\varepsilon(\dd x\,\dd\xi\,\dd t)<\infty.
\end{equation}
To see this, use convex renormalizations with
\[
 b_{a,R}''(r)=\frac{\vartheta(r/R)}{r+a},\qquad
 b_{a,R}(0)=b_{a,R}'(0)=0,\qquad 0<a\leq1,\quad R\geq1,
\]
where $\vartheta$ is smooth, nonincreasing, equal to one on $[0,1]$
and zero on $[2,\infty)$.  For fixed $\varepsilon,a,R$, the outer
velocity cutoffs are removed by \eqref{eq:global-high-velocity-tail}
and Lemma~\ref{lem:canonical-energy-localization}.
Mass conservation permits replacing $b_{a,R}(r)$ at both endpoints by
$b_{a,R}(r)+r\log a$, which is bounded below uniformly in $a,R$ and
above by $H(r)+C(1+r)$.  The only additional drift is
\[
 \lambda_V\cos(\alpha_V)\int_{Q_T}
 \left(\int_0^{\rho^\varepsilon}s b_{a,R}''(s)\,\dd s\right)
 (\rho^\varepsilon-M),\qquad
 0\leq\int_0^r s b_{a,R}''(s)\,\dd s\leq r.
\]
Its absolute value is controlled by
$\int_{Q_T}((\rho^\varepsilon)^2+M\rho^\varepsilon)$.
The control term is bounded by
$\int_{Q_T}|P_Kg^\varepsilon|^2+\int_0^T\mathcal I(\rho^\varepsilon)\,\dd t$,
up to an absolute constant; the trace is at most
$\varepsilon N_KT/2$.  Bessel's inequality bounds the martingale
quadratic variation by $4\varepsilon\int_0^T\mathcal I(\rho^\varepsilon)\,\dd t$.
Thus \eqref{eq:uniform-small-noise-energy} and
\eqref{eq:mass-fisher-L2} bound the expected dissipation
$\int b_{a,R}''(\xi)\,\dd q^\varepsilon$ uniformly.  Monotone convergence,
first as $R\uparrow\infty$ and then as $a\downarrow0$, proves
\eqref{eq:uniform-weighted-kinetic-mass}.  In particular, the masses of
$q^\varepsilon$ on compact velocity intervals are tight.

Together with \eqref{eq:coulomb-time-translation} and the compactness
estimates cited above, this proves tightness of the joint laws of
\begin{equation}\label{eq:controlled-tightness-tuple}
 (\rho^\varepsilon,\nabla\sqrt{\rho^\varepsilon},
 q^\varepsilon,g^\varepsilon,W)
\end{equation}
in the product of the strong $L^1(Q_T)$ topology for densities, the weak
$L^2$ topologies for gradients and controls, the local vague topology for
kinetic measures, and the topology of $\mathcal W$ in
\eqref{eq:ldp-spaces} for the driving paths.
For the mixed control--gradient term, adjoin as in the same FG23 proof
the auxiliary measures
\[
 \delta_{\rho^\varepsilon(x,t)}(\dd\xi)
 \bigl(|P_{K(\varepsilon)}g^\varepsilon|^2
       +|\nabla\sqrt{\rho^\varepsilon}|^2\bigr)\dd x\dd t,
\]
on $\T^2\times[0,\infty)\times[0,T]$, whose total masses are bounded
on each Fisher level.  The weak factors
then lie in compact metrizable sets; increasing the levels and using
\eqref{eq:controlled-probability-tail} removes the localization.

The extension from $L^2$ to finite-entropy initial data uses our
continuous-dependence estimate.  Couple solutions with initial data
$e^{m^{-1}\Delta}\rho_0$ and $\rho_0$, preserving their common mass and
using the same noise and control.  Their initial entropies are uniformly
bounded.  On a common Fisher level,
\eqref{eq:spde-equal-mass-dependence} and
\eqref{eq:osgood-continuous-dependence} give convergence in $L^1(Q_T)$;
\eqref{eq:controlled-probability-tail} removes the level uniformly in
$\varepsilon$.  This replaces the contraction argument at the end of
\cite[the proof of Theorem~6.6]{FG23}.

\subsection{Controlled convergence and proof of the LDP}

\begin{theorem}[Controlled convergence]\label{thm:controlled-convergence}
Let $g^\varepsilon\in\mathcal A_N$ converge in distribution to $g$ in the
weak topology of $S_N$.  Then
\begin{equation}\label{eq:controlled-convergence}
 (\rho^{\varepsilon,K(\varepsilon),g^\varepsilon},g^\varepsilon)
 \Longrightarrow(\rho[g],g)
 \qquad\text{in }\cX\times S_N.
\end{equation}
\end{theorem}

\begin{proof}
For the family \eqref{eq:controlled-density-family}, apply the joint
compactness of \eqref{eq:controlled-tightness-tuple} and the representation argument of
\cite[the proof of Theorem~6.6, following (6.15)]{FG23}, first on compact
entropy--Fisher levels and then by diagonal extraction.  Along a
subsequence on the representing space, almost surely,
\[
 \rho^\varepsilon\to\rho\quad\text{in }L^1(Q_T),
 \qquad g^\varepsilon\weakto g\quad\text{in }L^2(Q_T;\R^2).
\]
The kinetic identities, martingale structure, and initial datum transfer
as in that proof.  The additional Coulomb drift is an adapted integrable
finite-variation term, with Borel test pairings given by
\eqref{eq:odd-kernel-symmetrization}.

Strong convergence of the square roots, as in
\eqref{eq:skeleton-control-limit}, and strong convergence of $P_K$ to
the identity give
\begin{equation}\label{eq:projected-control-limit}
\begin{aligned}
 P_{K(\varepsilon)}(\sqrt{\rho^\varepsilon}\nabla\varphi)
 &\longrightarrow\sqrt\rho\nabla\varphi
 &&\text{in }L^2(Q_T;\R^2),\\
 \int_{Q_T}g^\varepsilon\cdot
 P_{K(\varepsilon)}(\sqrt{\rho^\varepsilon}\nabla\varphi)
 &\longrightarrow\int_{Q_T}g\cdot\sqrt\rho\nabla\varphi.
\end{aligned}
\end{equation}
Lemma~\ref{lem:interaction-compactness} gives convergence of both the
distributional Coulomb flux and its renormalized terms at every fixed
velocity cutoff.

For a renormalization $b$ with
$\operatorname{supp}b'\Subset(0,\infty)$, the It\^o correction terms are
\[
 -\frac{\varepsilon F_{1,K}}8\int_{Q_T}
 b'(\rho^\varepsilon)\nabla\log\rho^\varepsilon\cdot\nabla\varphi,
 \qquad
 \frac{\varepsilon N_K}2\int_{Q_T}
 b''(\rho^\varepsilon)\rho^\varepsilon\varphi.
\]
They vanish on every Fisher level because
$|b'(\rho)\nabla\log\rho|\lesssim_b|\nabla\sqrt\rho|$,
$|b''(\rho)\rho|\lesssim_b1$, and
$\varepsilon F_{1,K},\varepsilon N_K\to0$.
The velocity-localized martingales vanish by the estimate in
\cite[the proof of Theorem~6.6, (6.17)--(6.19)]{FG23}, using
\eqref{eq:uniform-small-noise-energy}.  For a classical spatial test
$\varphi$, the unlocalized martingale has quadratic variation at most
$\varepsilon MT\|\nabla\varphi\|_\infty^2$ by Bessel's inequality.

It remains to remove the velocity cutoff from the limiting Coulomb
terms.  Choose $\phi_\delta\in C_c^\infty((0,\infty);[0,1])$ equal to
one on $[2\delta,\delta^{-1}]$, supported in
$[\delta,2\delta^{-1}]$, with
$|\phi_\delta'|\lesssim
\delta^{-1}\mathbf1_{[\delta,2\delta]}
+\delta\mathbf1_{[\delta^{-1},2\delta^{-1}]}$.
For a smooth test $\varphi$, integration by parts gives
\begin{equation}\label{eq:coulomb-velocity-localized-term}
\begin{aligned}
 &\int_{Q_T}\phi_\delta(\rho)\rho(V*\rho)\cdot\nabla\varphi
 +\int_{Q_T}\phi_\delta'(\rho)\rho\nabla\rho\cdot(V*\rho)\varphi\\
 &=\int_{Q_T}\phi_\delta(\rho)\rho(V*\rho)\cdot\nabla\varphi
 -\int_{Q_T}\left(\int_0^\rho s\phi_\delta'(s)\,\dd s\right)
 (V*\rho)\cdot\nabla\varphi\\
 &\quad+\lambda_V\cos(\alpha_V)
 \int_{Q_T}\left(\int_0^\rho s\phi_\delta'(s)\,\dd s\right)
 (\rho-M)\varphi.
\end{aligned}
\end{equation}
The primitive satisfies
\[
 \left|\int_0^r s\phi_\delta'(s)\,\dd s\right|
 =\left|r\phi_\delta(r)-\int_0^r\phi_\delta(s)\,\dd s\right|
 \lesssim\delta+r\mathbf1_{\{r\geq\delta^{-1}\}}.
\]
Since $\rho^2,\rho|V*\rho|\in L^1(Q_T)$ by
Lemma~\ref{lem:mass-fisher-consequences}, the difference between
\eqref{eq:coulomb-velocity-localized-term} and
$\int_{Q_T}\rho(V*\rho)\cdot\nabla\varphi$ is bounded by
\[
 C_{V,\varphi}\left[
 \int_{Q_T}|1-\phi_\delta(\rho)|\rho|V*\rho|
 +\delta\bigl(\|V*\rho\|_{L^1(Q_T)}+2MT\bigr)
 +\int_{\{\rho\geq\delta^{-1}\}}
 \bigl(\rho|V*\rho|+\rho^2+M\rho\bigr)\right]\longrightarrow0.
\]
Thus the Coulomb cutoff limit holds as $\delta\downarrow0$.
For the limiting kinetic measure, local vague lower semicontinuity and
\eqref{eq:uniform-weighted-kinetic-mass} give $\int\xi^{-1}\,\dd q<\infty$
almost surely.  Since $|\phi_\delta'|\lesssim\xi^{-1}$ on its support,
\[
 \left|\int\varphi\phi_\delta'\,\dd q\right|
 \lesssim_\varphi
 \int_{\{\delta\leq\xi\leq2\delta\}\,\cup\,
        \{\delta^{-1}\leq\xi\leq2\delta^{-1}\}}
 \frac1\xi\,\dd q\longrightarrow0.
\]
The auxiliary control--gradient measure has finite total mass, so its
terms with $\xi\phi_\delta'$ vanish in the same way.  These bounds justify
the remaining cutoff passages in
\cite[the proof of Theorem~6.6, (6.20)--(6.25)]{FG23}.
The order is first $\varepsilon\downarrow0$ at fixed velocity cutoff,
then $\delta\downarrow0$.

Passing to the limit in the diffusion and control terms as in the cited
proof, and using \eqref{eq:projected-control-limit}, gives
\eqref{eq:skeleton-weak}, including $\int\rho_0\varphi(0)$.
Nonnegativity and mass pass to the limit; the entropy, physical
dissipation, and Fisher information are retained by
Lemmas~\ref{lem:physical-flux-lower-semicontinuity}
and~\ref{lem:entropy-tail-transfer} and weak lower semicontinuity on
the localization levels.  Equation~\eqref{eq:controlled-probability-tail}
then makes these quantities finite almost surely.  Hence $\rho$ is an
entropy skeleton solution with initial datum $\rho_0$ and control $g$.
Proposition~\ref{prop:main-skeleton}, whose uniqueness follows from the
Osgood lemma, identifies $\rho=\rho[g]$.  Every subsequential
limit is therefore the same, proving \eqref{eq:controlled-convergence}.
\end{proof}

\begin{proof}[Proof of Theorem~\ref{thm:ldp-intro}]
Proposition~\ref{prop:main-skeleton} gives compactness of
$\{\rho[g]:g\in S_N\}$ for every $N$, and
Theorem~\ref{thm:controlled-convergence} gives convergence of the shifted
solutions.  These are the two hypotheses of the weak-convergence criterion
\cite[Theorems~2 and~6]{BDM08}, applied as in the proof of
\cite[Theorem~6.8]{FG23}.  It yields the Laplace principle with rate
\eqref{eq:rate-intro}.

For completeness, whenever $I_{\rho_0}(\rho)<\infty$, a minimizing
sequence of controls has a weakly convergent subsequence.  Skeleton
stability and weak lower semicontinuity show that its limit represents
$\rho$ and attains the infimum.  Consequently, for every $a\geq0$,
\[
 \{\rho:I_{\rho_0}(\rho)\leq a\}
 =\{\rho[g]:g\in S_{2a}\},
\]
which is compact by Proposition~\ref{prop:main-skeleton}.  Thus the rate
is good, and the Laplace principle on the Polish space $\cX$ is equivalent
to the asserted large deviation principle.
\end{proof}

\subsection{Dual representation of the action}

For a nonnegative $\rho\in\cX$ with mass $M$ almost everywhere in time,
essentially bounded entropy, and finite integrated Fisher information,
define
\begin{equation}\label{eq:dual-residual}
 \mathscr L_{\rho,\rho_0}(\varphi)
 =-\int_{\T^2}\rho_0\varphi(0)
 -\int_{Q_T}\rho(\partial_t+\Delta)\varphi
 -\int_{Q_T}\rho(V*\rho)\cdot\nabla\varphi,
\end{equation}
for $\varphi\in C_c^\infty([0,T)\times\T^2)$.
The Coulomb term is integrable by
Lemma~\ref{lem:mass-fisher-consequences}; the boundary term encodes the
initial datum without assuming a temporal trace.

\begin{proposition}[Dual representation]\label{prop:dual-rate}
For every $\rho$ with the properties above,
\begin{equation}\label{eq:dual-rate}
 I_{\rho_0}(\rho)
 =\sup_{\varphi\in C_c^\infty([0,T)\times\T^2)}
 \left\{\mathscr L_{\rho,\rho_0}(\varphi)
 -\frac12\int_{Q_T}\rho|\nabla\varphi|^2\right\}.
\end{equation}
For all other $\rho\in\cX$, $I_{\rho_0}(\rho)=+\infty$.
When the value is finite, the unique minimal representing control satisfies
\begin{equation}\label{eq:minimal-control}
 g_\rho\in\overline{\{\sqrt\rho\nabla\varphi:
 \varphi\in C_c^\infty([0,T)\times\T^2)\}}^{L^2(Q_T;\R^2)},
 \qquad I_{\rho_0}(\rho)=\tfrac12\|g_\rho\|_2^2.
\end{equation}
\end{proposition}

\begin{proof}
Use the weighted-Riesz argument of
\cite[Lemma~8.3 and Proposition~8.4]{FG23}, with the residual
\eqref{eq:dual-residual}.  The weak skeleton equation and Young's
inequality give the bound by $I_{\rho_0}$ in \eqref{eq:dual-rate}.
Conversely, if the supremum equals $a<\infty$, scaling the test function
gives $|\mathscr L_{\rho,\rho_0}(\varphi)|
\leq\sqrt{2a}\|\sqrt\rho\nabla\varphi\|_2$.
Riesz representation on the closed space in \eqref{eq:minimal-control}
therefore supplies $g_\rho$ with
$\mathscr L_{\rho,\rho_0}(\varphi)
=\int_{Q_T}g_\rho\cdot\sqrt\rho\nabla\varphi$ and
$a=\frac12\|g_\rho\|_2^2$.  This is precisely
\eqref{eq:skeleton-weak}, including its initial boundary term, so
Proposition~\ref{prop:main-skeleton} identifies $\rho=\rho[g_\rho]$.
Every other representing control differs from $g_\rho$ by a vector
orthogonal to that closed space, proving minimality and equality.
Finally, failure of any stated property excludes a representing control
by Proposition~\ref{prop:main-skeleton}.
\end{proof}

\appendix
\section{Regularization and passage to the limit}
\label{app:approximation}

This appendix justifies passage from the regularized equations to
stochastic kinetic solutions.  We first
choose compatible Green-kernel and square-root regularizations and justify
entropy continuity.  We then establish the regularized nonlocal It\^o trace
underlying the localized estimates of Section~\ref{sec:spde}.  Next we record
the closure and lower-semicontinuity properties under strong $L^1$
convergence and then prove compactness for the approximations.  Finally,
we fix the order of the
singular limits and verify the velocity-tail conditions in
Definition~\ref{def:stochastic-kinetic-solution}.  These results are also
used in Section~\ref{sec:full-ldp}.

Throughout this appendix we assume $M>0$.  When $M=0$, nonnegativity and
mass conservation force the density to vanish identically, so no
logarithmic approximation is needed.

\subsection{Compatible regularizations and entropy continuity}

Let $Q_\eta=e^{\eta\Delta}$ be the torus heat semigroup and set
\begin{equation}\label{eq:green-regularization}
 \begin{aligned}
 G_\eta&=Q_\eta G,
 &G_\eta*f&=Q_\eta(G*f)=G*(Q_\eta f),\\
 V_\eta&=Q_\eta V
 =\lambda_V\cos(\alpha_V)\nabla G_\eta
   +\lambda_V\sin(\alpha_V)\J\nabla G_\eta,
 &-\Delta(G_\eta*f)&=Q_\eta\left(f-\int_{\T^2}f\right).
 \end{aligned}
\end{equation}
The kernel $G_\eta$ is even and mean zero, and convolution by $G_\eta$
is self-adjoint.  Moreover,
\begin{equation*}
 \inf_{\T^2}G_\eta\geq\inf_{\T^2}G,
 \qquad
 0\leq\langle f,G_\eta*f\rangle
 \leq\langle f,G*f\rangle,
 \qquad
 \|\nabla G_\eta\|_{L^1}\leq\|\nabla G\|_{L^1}.
\end{equation*}
The quadratic-form inequalities follow from the nonnegative heat multiplier.
Positivity and Jensen's inequality make all entropy--potential bounds
uniform in $\eta$.  With $\Lambda_*$ from
\eqref{eq:localized-cutoff-objects}, in particular,
\begin{equation}\label{eq:regularized-positive-shift}
 G_\eta*h_S(\rho)+\Lambda_*\geq1
 \qquad\text{for every nonnegative density of mass }M.
\end{equation}
The shift is independent of all cutoff and regularization parameters.  The
local expansion of $G$ and standard heat-kernel estimates also give,
uniformly in $\eta$,
\begin{equation*}
 |\nabla G_\eta(z)|\lesssim
 1+\frac1{d_{\T^2}(z,0)+\sqrt\eta},
 \qquad
 |D^2G_\eta(z)|\lesssim
 1+\frac1{(d_{\T^2}(z,0)+\sqrt\eta)^2}.
\end{equation*}
Consequently the mismatch estimates are uniform; the rotational estimate
follows by writing $\J\nabla G_\eta*f=\J\nabla G*(Q_\eta f)$ in the
log--Lipschitz argument.

The lower-end approximation of the square root is
\begin{equation}\label{eq:square-root-approximation}
 \sigma_\zeta(r)=\sqrt{r+\zeta}-\sqrt\zeta,
 \qquad r\geq0,\quad \zeta>0.
\end{equation}
If the regularized existence theorem requires a bounded coefficient, fix
$0\leq\chi_m\leq1$ with $\chi_m=1$ on $[0,m]$ and $\chi_m=0$ on
$[2m,\infty)$, and use
\begin{equation}\label{eq:bounded-square-root-approximation}
 \sigma_{\zeta,m}(r)
 =\int_0^r\frac{\chi_m(s)}{2\sqrt{s+\zeta}}\,\dd s.
\end{equation}

\begin{lemma}[Admissible noise regularization]
\label{lem:admissible-noise-regularization}
Both coefficients in \eqref{eq:square-root-approximation}--
\eqref{eq:bounded-square-root-approximation} vanish at zero and satisfy
\begin{equation}\label{eq:admissible-noise-coefficient}
 \sigma(r)^2\leq r,
 \qquad
 4r\sigma'(r)^2\leq1
 \qquad(r>0).
\end{equation}
Furthermore, $\sigma_{\zeta,m}\to\sigma_\zeta$ locally in $C^1$ as
$m\to\infty$, and $\sigma_\zeta\to\sqrt{\mathord\cdot}$ locally in
$C^1((0,\infty))$ as $\zeta\downarrow0$.

For any smooth coefficient satisfying \eqref{eq:admissible-noise-coefficient},
define
\begin{equation}\label{eq:regularized-noise-operators}
 \mathcal B_{k,j}^{\sigma}(\rho)
 =-\partial_j(e_k\sigma(\rho)).
\end{equation}
Then the sine--cosine identities in \eqref{eq:finite-mode-noise} give
\begin{equation}\label{eq:regularized-noise-identities}
 \begin{aligned}
 \frac12\sum_{k\in\mathcal E_K}\sum_{j=1}^2
 \partial_j\!\left(
 e_k\sigma'(\rho)\partial_j(e_k\sigma(\rho))
 \right)
 &=\frac{F_{1,K}}2
 \diver\bigl(\sigma'(\rho)^2\nabla\rho\bigr),\\
 \sum_{k\in\mathcal E_K}\sum_{j=1}^2
 \left|\partial_j(e_k\sigma(\rho))\right|^2
 &=F_{1,K}\sigma'(\rho)^2|\nabla\rho|^2
   +N_K\sigma(\rho)^2.
 \end{aligned}
\end{equation}
\end{lemma}

\begin{proof}
The bounds follow from
$\sqrt{r+\zeta}-\sqrt\zeta\leq\sqrt r$ and
$4r\sigma_\zeta'(r)^2=r/(r+\zeta)\leq1$; truncation multiplies the
derivative by $\chi_m\leq1$.  The asserted convergences are immediate, and
expansion of $\partial_j(e_k\sigma(\rho))$, together with
$\sum_ke_k\nabla e_k=0$, proves
\eqref{eq:regularized-noise-identities}.
\end{proof}

For a coefficient from Lemma~\ref{lem:admissible-noise-regularization},
the compatible It\^o equation is
\begin{equation}\label{eq:regularized-ito-equation}
\begin{aligned}
 \dd\rho={}&\Bigl[
 \Delta\rho-\diver\bigl(\rho(V_\eta*\rho)\bigr)
 -\diver\bigl(\sigma(\rho)P_Kg\bigr)
 +\frac{\varepsilon F_{1,K}}2
 \diver\bigl(\sigma'(\rho)^2\nabla\rho\bigr)
 \Bigr]\dd t\\
 &+\sqrt\varepsilon
 \sum_{k\in\mathcal E_K}\sum_{j=1}^2
 \mathcal B_{k,j}^{\sigma}(\rho)\,\dd\beta_t^{k,j}.
\end{aligned}
\end{equation}
The same $G_\eta$ is used in the physical drift, the physical energy, and
the auxiliary interaction energy.  This compatibility is essential:
replacing only one occurrence of $G$ by a smoother kernel would introduce
an uncontrolled mismatch in the energy identity.

We next justify the entropy hitting times used in the localization scheme.  For
$\ell>0$, let
\begin{equation}\label{eq:entropy-test-regularization}
 \beta_\ell(r)
 =(r+\ell)\log(r+\ell)-r-\ell\log\ell+1,
 \qquad \beta_\ell''(r)=\frac1{r+\ell}.
\end{equation}
For every nonnegative mass-$M$ density,
\begin{equation}\label{eq:uniform-entropy-test-error}
 0\leq\int_{\T^2}\beta_\ell(\rho)-\int_{\T^2}H(\rho)
 \leq\ell\left[1+\log\left(1+\frac M\ell\right)\right]
 \longrightarrow0.
\end{equation}
Indeed, the pointwise difference is
$r\log(1+\ell/r)+\ell\log(1+r/\ell)$; the first term is at most
$\ell$, and Jensen's inequality treats the second after integration.

\begin{lemma}[Entropy chain rule and continuity]
\label{lem:regularized-entropy-chain-rule}
Let $\rho$ solve \eqref{eq:regularized-ito-equation} at fixed
$(\eta,\zeta,m)$, preserve mass, and satisfy
$\E\int_0^T\|g(t)\|_2^2\,\dd t<\infty$.  After the usual martingale
localization, It\^o's formula with
\eqref{eq:entropy-test-regularization} gives
\begin{equation*}
\begin{gathered}
 \dd\int_{\T^2}\beta_\ell(\rho)
 +\frac12\int_{\T^2}\frac{|\nabla\rho|^2}{\rho+\ell}\,\dd t
 \leq\left[
 M\|V_\eta*\rho\|_{L^\infty}^2
 +\|g\|_2^2+\frac{\varepsilon N_K}{2}
 \right]\dd t+\dd\mathcal M_t^\ell,\\
 \dd\langle\mathcal M^\ell\rangle_t
 \leq\varepsilon\int_{\T^2}
 \frac{|\nabla\rho|^2}{\rho+\ell}\,\dd t.
\end{gathered}
\end{equation*}
Consequently, at fixed $\eta$ the expected entropy supremum and the
integrated gradient quantity above are bounded uniformly in $\ell$.  The
constant may depend on $\eta$ and is used only to justify the subsequent
approximation-uniform calculation.  For controls of merely finite action,
the same statement holds after stopping at a deterministic action level.
Moreover,
\begin{equation}\label{eq:continuous-entropy-representative}
 t\longmapsto\int_{\T^2}H(\rho(t))
\end{equation}
has a continuous adapted representative.
\end{lemma}

\begin{proof}
The exact-square-root correction, its Hessian cancellation, and the
localized Burkholder--Davis--Gundy estimate are those in
\cite[Proposition~5.18, especially equation~(5.29)]{FG24}.  The two added
finite-variation terms satisfy
\begin{align*}
 \left|\int_{\T^2}\frac{\rho}{\rho+\ell}
 (V_\eta*\rho)\cdot\nabla\rho\right|
 &\leq\frac14\int_{\T^2}\frac{|\nabla\rho|^2}{\rho+\ell}
 +M\|V_\eta*\rho\|_\infty^2,\\
 \left|\int_{\T^2}\frac{\sigma(\rho)}{\rho+\ell}
 P_Kg\cdot\nabla\rho\right|
 &\leq\frac14\int_{\T^2}\frac{|\nabla\rho|^2}{\rho+\ell}
 +\|g\|_2^2.
\end{align*}
Together with \eqref{eq:regularized-noise-identities}, these estimates give
the asserted bounds without introducing an inverse power of the coefficient
regularization.

For each $\ell$, It\^o's formula gives a continuous version of
$\int\beta_\ell(\rho(t))$.  The time-uniform error
\eqref{eq:uniform-entropy-test-error} therefore proves continuity and
adaptedness of \eqref{eq:continuous-entropy-representative}.
\end{proof}

At fixed regularization, existence, nonnegativity, mass preservation, and
the kinetic formulation follow from \cite[Proposition~3.2]{JSW26}.  Its
hypotheses follow from the parity of $G_\eta$ and $\nabla G_\eta$, the
identity
$\diver(V_\eta*f)=-\lambda_V\cos(\alpha_V)Q_\eta(f-M)$, and
\eqref{eq:green-regularization}--\eqref{eq:regularized-positive-shift}.
Finite-dimensional Girsanov, as in \cite[Proposition~6.2]{FG23}, gives the
controlled equation through \eqref{eq:finite-mode-novikov}--
\eqref{eq:finite-mode-cameron-martin-shift}; the uniform continuation
estimate is proved in Section~\ref{sec:spde}.

The smooth positive approximations
$(1-\delta)Q_\eta\rho_0+\delta M$, with $\delta,\eta\downarrow0$, preserve
mass and converge both in $L^1$ and in entropy by Jensen's inequality and
lower semicontinuity.

\subsection{Regularized localized energy and the nonlocal It\^o trace}

For later reference, define the regularized physical and auxiliary energies
and dissipations by
\begin{equation}\label{eq:regularized-physical-auxiliary-energies}
\begin{aligned}
 \mathcal F_V^\eta(\rho)
 &=\int_{\T^2}H(\rho)
 -\frac{\lambda_V\cos(\alpha_V)}2
 \langle\rho,G_\eta*\rho\rangle,\\
 \mathcal D_V^\eta(\rho)
 &=\left\|2\nabla\sqrt\rho
 -\lambda_V\cos(\alpha_V)\sqrt\rho\,
      \nabla G_\eta*\rho\right\|_{L^2}^2,\\
 \mathcal F_S^\eta(\rho)
 &=\int_{\T^2}H(\rho)
 -\lambda_V\cos(\alpha_V)\left[
 \frac12\langle h_S(\rho),G_\eta*h_S(\rho)\rangle
 +\Lambda_*\left(\int_{\T^2}h_S(\rho)-M\right)
 \right],\\
 \mathcal D_S^\eta(\rho)
 &=\int_{\T^2}\rho\left|\nabla\!\left(
 \log\rho-\lambda_V\cos(\alpha_V)\theta_S(\rho)
 \bigl(G_\eta*h_S(\rho)+\Lambda_*\bigr)
 \right)\right|^2.
\end{aligned}
\end{equation}
The two first variations in \eqref{eq:regularized-physical-auxiliary-energies}
are different: the physical drift contains
$\log\rho-\lambda_V\cos(\alpha_V)G_\eta*\rho$, whereas the derivative of
the auxiliary energy contains
\begin{equation*}
 \log\rho-\lambda_V\cos(\alpha_V)\theta_S(\rho)
 \bigl(G_\eta*h_S(\rho)+\Lambda_*\bigr).
\end{equation*}
The stopped mismatch estimate compares these two quantities; it does not
replace one by the other.

For the trace calculation, define
\begin{equation}\label{eq:approximate-logarithmic-primitive}
 \ell_{S,\sigma}(0)=0,
 \qquad
 \ell_{S,\sigma}'(r)=4\theta_S(r)\sigma'(r)^2.
\end{equation}
Lemma~\ref{lem:admissible-noise-regularization} implies
\begin{equation}\label{eq:logarithmic-primitive-bound}
 0\leq\ell_{S,\sigma}(r)\lesssim S+\log_+r.
\end{equation}

At fixed $S$, the It\^o formula for $\mathcal F_S^\eta$ is justified by
replacing its entropy part by $\int\beta_\ell(\rho)$ and sending
$\ell\downarrow0$ after the two local gradient cancellations.  Indeed,
\begin{equation*}
 \frac{\sqrt\rho\,\nabla\rho}{\rho+\ell}
 =\frac{\rho}{\rho+\ell}\,2\nabla\sqrt\rho
 \longrightarrow2\nabla\sqrt\rho
 \quad\text{in }L^2(Q_T),
\end{equation*}
the cutoff derivatives are supported away from zero, and
\eqref{eq:logarithmic-primitive-bound} controls the residual trace.  No
singular member of either canceled pair is passed separately.

\begin{proposition}[Regularized nonlocal It\^o trace]
\label{prop:regularized-ito-trace}
Assume $\lambda_V\cos(\alpha_V)>0$ and let $\rho$ be smooth and strictly
positive.  For the noise directions in
\eqref{eq:regularized-noise-operators},
\begin{equation}\label{eq:regularized-nonlocal-trace}
\begin{aligned}
 &\left\langle D\mathcal F_S^\eta(\rho),
 \frac{\varepsilon F_{1,K}}2
 \diver\bigl(\sigma'(\rho)^2\nabla\rho\bigr)\right\rangle
 +\frac\varepsilon2\sum_{k\in\mathcal E_K}\sum_{j=1}^2
 D^2\mathcal F_S^\eta(\rho)
 [\mathcal B_{k,j}^{\sigma}(\rho),
  \mathcal B_{k,j}^{\sigma}(\rho)]\\
 &\quad=\frac{\varepsilon N_K}{2}
 \int_{\T^2}\frac{\sigma(\rho)^2}{\rho}
 +\frac{\varepsilon\lambda_V\cos(\alpha_V)F_{1,K}}8
 \int_{\T^2}Q_\eta\left(
 h_S(\rho)-\int_{\T^2}h_S(\rho)
 \right)\ell_{S,\sigma}(\rho)\\
 &\qquad-\frac{\varepsilon\lambda_V\cos(\alpha_V)N_K}{2}
 \int_{\T^2}\theta_S'(\rho)
 \bigl(G_\eta*h_S(\rho)+\Lambda_*\bigr)\sigma(\rho)^2\\
 &\qquad-\frac{\varepsilon\lambda_V\cos(\alpha_V)}2
 \sum_{k\in\mathcal E_K}\sum_{j=1}^2
 \left\langle
 \theta_S(\rho)\mathcal B_{k,j}^{\sigma}(\rho),
 G_\eta*\bigl(\theta_S(\rho)
 \mathcal B_{k,j}^{\sigma}(\rho)\bigr)
 \right\rangle.
\end{aligned}
\end{equation}
The last two lines are nonpositive.  Uniformly in $\eta$, $\zeta$, and
$m$, the left-hand side satisfies
\begin{equation*}
 \text{left-hand side of \eqref{eq:regularized-nonlocal-trace}}
 \lesssim_{V,M}
 \varepsilon N_K+\varepsilon F_{1,K}
 \left(1+\int_{\T^2}H(\rho)+MS\right).
\end{equation*}
\end{proposition}

\begin{proof}
The entropy part of the correction paired with the first variation is
\[
 -\frac{\varepsilon F_{1,K}}2
 \int_{\T^2}\sigma'(\rho)^2\frac{|\nabla\rho|^2}{\rho}.
\]
The gradient part of the entropy Hessian is its exact opposite by
\eqref{eq:regularized-noise-identities}; their sum leaves the first term
on the right of \eqref{eq:regularized-nonlocal-trace}.  The local
interaction Hessian similarly cancels the term containing
$\theta_S'(\rho)|\nabla\rho|^2$ before any singular limit is taken.
Its $N_K$ remainder is the third line of
\eqref{eq:regularized-nonlocal-trace}, and the nonlocal interaction Hessian
is the fourth line.

By \eqref{eq:approximate-logarithmic-primitive}, the remaining interaction
integral is
\[
 \frac{\varepsilon\lambda_V\cos(\alpha_V)F_{1,K}}8
 \int_{\T^2}\nabla G_\eta*h_S(\rho)
 \cdot\nabla\ell_{S,\sigma}(\rho),
\]
which equals the second term on the right of
\eqref{eq:regularized-nonlocal-trace} by
\eqref{eq:green-regularization}.  Its upper bound, rather than its sign,
is what is needed.  Since $\ell_{S,\sigma}\geq0$,
\begin{align*}
 &\int_{\T^2}Q_\eta\left(
 h_S(\rho)-\int h_S(\rho)\right)\ell_{S,\sigma}(\rho)\\
 &\qquad\lesssim MS+
 \int_{\T^2}(Q_\eta\rho)\log(1+\rho)\\
 &\qquad\leq MS+\int_{\T^2}H(Q_\eta\rho)+M
 \leq MS+\int_{\T^2}H(\rho)+M.
\end{align*}
The middle inequality is the entropy Young inequality
$x\log(1+y)\leq H(x)+y$, and the last is heat-semigroup entropy
contraction.  Finally, \eqref{eq:admissible-noise-coefficient},
\eqref{eq:regularized-positive-shift}, and positivity of the regularized
Green quadratic form prove the asserted signs and upper bound.
\end{proof}

The martingale associated with $\mathcal F_S^\eta$ has bracket
\begin{equation*}
 \dd\langle\mathcal M^{S,\eta}\rangle_t
 =\varepsilon\sum_{k,j}
 \left(\int_{\T^2}e_k\sigma(\rho)
 \partial_jD\mathcal F_S^\eta(\rho)\right)^2\dd t
 \leq\varepsilon\mathcal D_S^\eta(\rho)\,\dd t.
\end{equation*}
Here Bessel's inequality for the orthonormal modes is used; the pointwise
quantity $F_{1,K}$ is not substituted for the covariance-operator norm.
Likewise,
\begin{equation*}
 \int_{\T^2}\sigma(\rho)P_Kg\cdot
 \nabla D\mathcal F_S^\eta(\rho)
 \leq\frac14\mathcal D_S^\eta(\rho)+\|g\|_{L^2}^2.
\end{equation*}
Together with the uniform heat-kernel versions of the stopped physical
mismatch and skew estimates, Proposition~\ref{prop:regularized-ito-trace}
shows that every constant in the entropy-level localization is independent of
$(\eta,\zeta,m)$.  This is an a priori conclusion for the already
constructed regularized solutions, not a construction theorem.

\subsection{Closure under strong \texorpdfstring{$L^1$}{L1} convergence}

The results in this subsection are conditional on strong $L^1$ convergence.
The density-cutoff argument in the next subsection supplies precisely this
compactness for the approximating solutions.  Once it is available, the
lemmas below pass the Coulomb drift, the physical flux, and the entropy
bounds to the limit.

We first record the endpoint closure of the nonlinear interaction.  Its
distributional-flux proof is different from the
Ladyzhenskaya--Prodi--Serrin kernel passage in
\cite[proof of Proposition~3.16, equations~(3.74)--(3.79)]{WZ24}: the
Coulomb kernel lies outside those assumptions, and the argument below uses
oddness to cancel its diagonal singularity.

\begin{lemma}[Closure of the Coulomb drift under strong $L^1$ convergence]
\label{lem:interaction-compactness}
Suppose that $\rho_n,\rho\geq0$ have uniformly bounded mass in
$L^\infty(0,T)$ and
\begin{equation*}
 \rho_n\longrightarrow\rho
 \qquad\text{strongly in }L^1(Q_T).
\end{equation*}
For densities in this class, the Coulomb flux is understood as the
distribution defined, for $a\in C_c^\infty(Q_T;\R^2)$, by
\begin{equation}\label{eq:odd-kernel-symmetrization}
\begin{aligned}
 \langle\rho(V*\rho),a\rangle
 :=\frac12\int_0^T\!\iint_{\T^2\times\T^2}
 &\rho(t,x)\rho(t,y)V(x-y)\\
 &\cdot\bigl(a(t,x)-a(t,y)\bigr)
 \dd x\dd y\dd t.
\end{aligned}
\end{equation}
This agrees with the ordinary integrable flux when $\rho\in L^2(Q_T)$,
in particular for the finite-Fisher-information densities used in the
solution theories.  With this convention,
\begin{equation}\label{eq:interaction-distributional-compactness}
 \rho_n(V*\rho_n)\longrightarrow\rho(V*\rho)
 \qquad\text{in }\mathcal D'(Q_T;\R^2).
\end{equation}
The same conclusion holds if $V$ on the left is replaced by
$Q_{\eta_n}V$, where $\eta_n\downarrow0$ and $Q_t=e^{t\Delta}$.
If, in addition, every density has the same mass $M$ almost everywhere
in time, $b\in C^2([0,\infty))$, and
$\operatorname{supp}b'\Subset(0,\infty)$, then the corresponding
velocity-localized renormalized terms are closed under the same
convergence.  More precisely, with
$B_b(r)=\int_0^r sb''(s)\,\dd s$, the derivative term is understood through
\begin{equation}\label{eq:renormalized-coulomb-compactness}
\begin{aligned}
 &\int_{Q_T}b''(\rho_n)\rho_n\nabla\rho_n\cdot
 \bigl((Q_{\eta_n}V)*\rho_n\bigr)\varphi\\
 &\quad=-\int_{Q_T}B_b(\rho_n)
 \bigl((Q_{\eta_n}V)*\rho_n\bigr)\cdot\nabla\varphi\\
 &\qquad+\lambda_V\cos(\alpha_V)\int_{Q_T}B_b(\rho_n)
 Q_{\eta_n}(\rho_n-M)\varphi.
\end{aligned}
\end{equation}
Under the stated $L^1$ assumptions, the right-hand side defines the
left-hand side as a distribution; no gradient regularity is asserted.
For densities with finite integrated Fisher information, it agrees with
the displayed gradient expression.
For every $\varphi\in C_c^\infty(Q_T)$, the sum of
\eqref{eq:renormalized-coulomb-compactness} and
\begin{equation*}
 \int_{Q_T}b'(\rho_n)\rho_n
 \bigl((Q_{\eta_n}V)*\rho_n\bigr)\cdot\nabla\varphi
\end{equation*}
converges to the same expression with $\rho_n$ replaced by $\rho$ and
$\eta_n=0$.
\end{lemma}

\begin{proof}
The kernel $V$ is odd and
$|V(z)|\lesssim_V 1+d_{\T^2}(z,0)^{-1}$.  The Lipschitz difference
of the test field in \eqref{eq:odd-kernel-symmetrization} cancels the
diagonal singularity, so this formula defines a distribution under the
mass bound alone.  If $\rho\in L^2(Q_T)$, then
\begin{equation*}
 \int_{Q_T}\rho\,(|V|*\rho)
 \leq\|V\|_{L^1(\T^2)}\|\rho\|_{L^2(Q_T)}^2<\infty.
\end{equation*}
Fubini's theorem and oddness then identify
\eqref{eq:odd-kernel-symmetrization} with the ordinary flux pairing.
Lemma~\ref{lem:mass-fisher-consequences} supplies this $L^2$ regularity
whenever mass and integrated Fisher information are finite.

Uniform mass and strong space--time $L^1$ convergence give
\begin{equation*}
 \rho_n(t,x)\rho_n(t,y)\longrightarrow
 \rho(t,x)\rho(t,y)
 \quad\text{in }L^1((0,T)\times\T^2\times\T^2).
\end{equation*}
Thus \eqref{eq:odd-kernel-symmetrization} converges to the corresponding
formula for $\rho$, proving
\eqref{eq:interaction-distributional-compactness}.  Heat convolution
preserves oddness, and
\begin{equation*}
 |Q_{\eta_n}V(z)|
 \lesssim_V 1+\frac1{d_{\T^2}(z,0)+\sqrt{\eta_n}}.
\end{equation*}
After symmetrization, the factors
$((Q_{\eta_n}V)(x-y))\cdot(a(t,x)-a(t,y))$ are uniformly bounded and
converge pointwise away from the diagonal.  First replace
$\rho_n(t,x)\rho_n(t,y)$ by $\rho(t,x)\rho(t,y)$ using their strong
$L^1$ convergence.  Dominated convergence with this fixed integrable
density then proves the assertion for regularized kernels.

For finite-Fisher-information densities, the map
$r\mapsto B_b(r^2)$ is Lipschitz, so $\nabla B_b(\rho_n)\in L^2(Q_T)$.
The velocity and its divergence also belong to $L^2(Q_T)$ by
Lemma~\ref{lem:mass-fisher-consequences} and
\eqref{eq:coulomb-identities}.  Sobolev integration by parts therefore
gives \eqref{eq:renormalized-coulomb-compactness}.  Under only the
stated mass assumptions, its right-hand side is well defined and is
the prescribed distributional interpretation.  The convergence argument
is the same localized interaction passage as
\cite[Theorem~3.1, Step~1, equations~(3.20)--(3.23)]{JSW26}.  Indeed,
$B_b$ and $rb'(r)-B_b(r)=b(r)-b(0)$ are bounded and continuous, while
\begin{equation*}
 (Q_{\eta_n}V)*\rho_n\longrightarrow V*\rho,
 \qquad
 Q_{\eta_n}(\rho_n-M)\longrightarrow\rho-M
 \quad\text{in }L^1(Q_T).
\end{equation*}
The first convergence follows from $V\in L^1(\T^2)$ and the common mass
bound, and the second from strong continuity of the heat semigroup on
$L^1$.  Combining the first term in
\eqref{eq:renormalized-coulomb-compactness} with its companion $b'$ term
leaves the bounded factor $b(\rho_n)-b(0)$.  Convergence in measure and
the two strong $L^1$ convergences above pass the resulting products and
prove the claim.
\end{proof}

The next lemma transfers the physical part of the energy estimate.  It
does not identify the kinetic measure or the singular correction at
vacuum.

\begin{lemma}[Lower semicontinuity of the physical flux]
\label{lem:physical-flux-lower-semicontinuity}
Assume
\begin{equation*}
 \eta_n\downarrow0,
 \qquad \rho_n\to\rho\quad\text{strongly in }L^1(Q_T),
 \qquad \int_{\T^2}\rho_n(t)=M\quad\text{for a.e. }t,
\end{equation*}
and
\begin{equation}\label{eq:physical-flux-lsc-bound}
 \sup_n\left[
 \operatorname*{ess\,sup}_{0<t<T}\int_{\T^2}H(\rho_n(t))
 +\int_0^T\mathcal D_V^{\eta_n}(\rho_n(t))\,\dd t
 \right]<\infty.
\end{equation}
Under the strict mass-subcritical condition, $\sqrt\rho$ belongs to
$L^2(0,T;H^1(\T^2))$, and
\begin{equation}\label{eq:physical-flux-lower-semicontinuity}
\begin{aligned}
 \operatorname*{ess\,sup}_{0<t<T}\int_{\T^2}H(\rho(t))
 &\leq\liminf_{n\to\infty}
 \operatorname*{ess\,sup}_{0<t<T}\int_{\T^2}H(\rho_n(t)),\\
 \int_0^T\mathcal D_V(\rho(t))\,\dd t
 &\leq\liminf_{n\to\infty}
 \int_0^T\mathcal D_V^{\eta_n}(\rho_n(t))\,\dd t.
\end{aligned}
\end{equation}
\end{lemma}

\begin{proof}
The Fisher recovery argument of
Proposition~\ref{prop:fisher-from-dissipation} is uniform under heat
regularization.  Indeed,
\begin{equation*}
 \int_{\T^2}\rho_n(Q_{\eta_n}\rho_n)-M^2
 \leq\|\rho_n\|_{L^2}^2-M^2,
\end{equation*}
and the heat multiplier preserves positivity.  Thus
\eqref{eq:physical-flux-lsc-bound} gives
\begin{equation*}
 \sup_n\int_0^T\mathcal I(\rho_n(t))\,\dd t<\infty.
\end{equation*}
The mass bound and the two-dimensional Gagliardo--Nirenberg inequality
then bound $\rho_n$ in $L^2(Q_T)$.  Along a subsequence,
\begin{equation}\label{eq:physical-flux-compactness}
 \sqrt{\rho_n}\to\sqrt\rho\quad\text{strongly in }L^2(Q_T),
 \qquad
 \nabla\sqrt{\rho_n}\weakto\nabla\sqrt\rho
 \quad\text{weakly in }L^2(Q_T).
\end{equation}

Since $\nabla G\in L^1(\T^2)$ and
$\nabla G_{\eta_n}=Q_{\eta_n}\nabla G$, transfer of convolution to an
$L^2$ test function gives
\begin{equation}\label{eq:regularized-potential-weak-limit}
 \nabla G_{\eta_n}*\rho_n\weakto\nabla G*\rho
 \quad\text{weakly in }L^2(Q_T).
\end{equation}
The products
$\sqrt{\rho_n}\,\nabla G_{\eta_n}*\rho_n$ therefore converge in
distributions to $\sqrt\rho\,\nabla G*\rho$: when paired with a smooth
bounded field, the square-root factor multiplying the test converges
strongly in $L^2$, while the potential gradient converges weakly.

The vector fields
\[
 2\nabla\sqrt{\rho_n}
 -\lambda_V\cos(\alpha_V)\sqrt{\rho_n}\,
   \nabla G_{\eta_n}*\rho_n
\]
are bounded in $L^2(Q_T)$ by
\eqref{eq:physical-flux-lsc-bound}.  Their weak $L^2$ limit is therefore
the physical weighted flux with $G$ by
\eqref{eq:physical-flux-compactness}--
\eqref{eq:regularized-potential-weak-limit}.  Weak lower semicontinuity
gives the second inequality in
\eqref{eq:physical-flux-lower-semicontinuity}.  After a further
subsequence, $\rho_n(t)\to\rho(t)$ in $L^1(\T^2)$ for almost every $t$;
entropy lower semicontinuity at those times gives the first inequality.
\end{proof}

The probabilistic transfer of entropy tails is also deterministic after a
joint representation with almost-sure strong $L^1$ convergence.

\begin{lemma}[Transfer of entropy exit bounds]
\label{lem:entropy-tail-transfer}
Suppose that nonnegative mass-$M$ processes $\rho_n$ are represented on
one probability space and $\rho_n\to\rho$ almost surely in $L^1(Q_T)$.
If a nonincreasing function $r_T$ satisfies
\begin{equation*}
 \sup_n\Pp\left(
 \operatorname*{ess\,sup}_{0<t<T}\int_{\T^2}H(\rho_n(t))>L
 \right)\leq r_T(L),
 \qquad r_T(L)\longrightarrow0,
\end{equation*}
then
\begin{equation*}
 \Pp\left(
 \operatorname*{ess\,sup}_{0<t<T}\int_{\T^2}H(\rho(t))>L
 \right)
 \leq\lim_{\delta\downarrow0}r_T(L-\delta).
\end{equation*}
In particular, the entropy essential supremum of $\rho$ is finite almost
surely.
\end{lemma}

\begin{proof}
Along an almost surely convergent subsequence, spatial $L^1$ convergence
holds at almost every time.  Entropy lower semicontinuity puts the event at
level $L$ inside the lower limit of the approximating events at level
$L-\delta$; Fatou's lemma and then $\delta\downarrow0$ give the claim.
\end{proof}

No convergence of the active cutoff is needed; lower semicontinuity of the
dissipation follows separately from
Lemma~\ref{lem:physical-flux-lower-semicontinuity}.

\subsection{Strong \texorpdfstring{$L^1$}{L1} compactness from density cutoffs}

The singular logarithmic correction does not give a useful unweighted
time derivative of $\rho$ near vacuum.  The following cutoff argument
proves the strong compactness required above without estimating
$\log_-\rho$.

Fix $\delta>0$.  Choose a smooth nondecreasing
$\Theta_\delta$ that vanishes on $[0,\delta]$, equals one on
$[2\delta,\infty)$, and satisfies
$0\leq\Theta_\delta'\lesssim\delta^{-1}$.  Set
\begin{equation}\label{eq:density-time-cutoff}
 p_\delta(r)=\int_0^r\Theta_\delta(s)\,\dd s,
 \qquad
 0\leq r-p_\delta(r)\leq2\delta.
\end{equation}

\begin{lemma}[Cutoff time compactness]
\label{lem:cutoff-time-compactness}
Consider solutions of \eqref{eq:regularized-ito-equation} with uniformly
bounded control energy and with entropy, physical dissipation, and total
Fisher information bounded in probability.  Assume also that
$\varepsilon F_{1,K}$ and $\varepsilon N_K$ are bounded along the family.
For each fixed $\delta>0$, the laws of the cutoffs in
\eqref{eq:density-time-cutoff} are tight in
$L^1(Q_T)$, uniformly in $(\eta,\zeta,m)$.  Consequently, the laws of
$\rho$ are tight in $L^1(Q_T)$.
\end{lemma}

\begin{proof}
Apply It\^o's formula to $p_\delta$ as in
\cite[Proposition~5.14]{FG24}; the
$p_\delta''(\rho)\sigma'(\rho)^2|\nabla\rho|^2$ terms cancel exactly.  The
only additional finite-variation terms are
\begin{equation*}
\begin{aligned}
 &-\diver\bigl(p_\delta'(\rho)\rho(V_\eta*\rho)\bigr)
 +p_\delta''(\rho)\rho(V_\eta*\rho)\cdot\nabla\rho,\\
 &-\diver\bigl(p_\delta'(\rho)\sigma(\rho)P_Kg\bigr)
 +p_\delta''(\rho)\sigma(\rho)P_Kg\cdot\nabla\rho .
\end{aligned}
\end{equation*}
On the support of $p_\delta''$ they obey
\begin{align*}
 |p_\delta''(\rho)\rho\nabla\rho\cdot
   (V_\eta*\rho)|
 &\lesssim\sqrt\delta\,|\nabla\sqrt\rho|
   |V_\eta*\rho|,\\
 |p_\delta''(\rho)\sigma(\rho)P_Kg\cdot\nabla\rho|
 &\lesssim|\nabla\sqrt\rho|\,|P_Kg|.
\end{align*}
The divergence terms are controlled by
\begin{equation*}
 \|\rho(V_\eta*\rho)\|_{L^1(Q_T)}
 \leq\|\rho\|_{L^2(Q_T)}\|V_\eta*\rho\|_{L^2(Q_T)},
 \qquad
 \|\sigma(\rho)P_Kg\|_{L^1(Q_T)}
 \leq\sqrt{TM}\,\|g\|_{L^2(Q_T)}.
\end{equation*}
Indeed, $\|V_\eta*\rho\|_2\leq\|V\|_1\|\rho\|_2$, and the Fisher bound
controls $\rho$ in $L^2(Q_T)$.  The coefficient bounds
\eqref{eq:admissible-noise-coefficient} and the identities
\eqref{eq:regularized-noise-identities} verify, after stopping at a Fisher
threshold, the remaining spatial, finite-variation, and martingale
estimates in \cite[Proposition~5.14, equations~(5.21)--(5.25)]{FG24}.
Consequently, \cite[Proposition~5.26]{FG24}, with the interaction-flux
extension in \cite[Lemma~3.4 and Proposition~3.5]{JSW26}, makes
$p_\delta(\rho)$ tight in $L^1(Q_T)$.  The displays above contain the
only additional Coulomb and control estimates.

Finally,
\begin{equation*}
 \|\rho-p_\delta(\rho)\|_{L^1(Q_T)}\leq2T\delta
\end{equation*}
deterministically.  Sending $\delta\downarrow0$ by the standard diagonal
compact-set argument proves tightness of the original densities.
\end{proof}

\subsection{Singular-limit order and kinetic identification}

The approximation parameters have different roles and are removed in the
following order.  At fixed smooth kernel and coefficient, first remove the
entropy-test parameter $\ell\downarrow0$.  The vacuum scale $S$ is then
held fixed between consecutive entropy hitting times, with $L$ a
sufficiently large fixed constant; the terminal entropy threshold is removed at this
regularized level.  Only afterward take
\begin{equation}\label{eq:approximation-limit-order}
 m\to\infty,
 \qquad \zeta\downarrow0,
 \qquad \eta\downarrow0.
\end{equation}
The parameter $L$ is a sufficiently large fixed constant.  The same choice
gives the uniform small-noise bounds because
$\varepsilon F_{1,K(\varepsilon)}L\to0$; the limit
$\varepsilon\downarrow0$ is taken only after
\eqref{eq:approximation-limit-order}.  There is no global It\^o formula
with $S$ treated as a state-dependent parameter.

For fixed $\varepsilon>0$ and $K<\infty$, we first carry out this
construction with $g=0$.  By Lemma~\ref{lem:cutoff-time-compactness} and
\cite[Propositions~5.26--5.27 and the proof of Theorem~5.29]{FG24}, a joint
representation of the densities, kinetic measures, Brownian paths, and
compact-velocity test martingales, together with Fisher localization, gives
\begin{equation*}
\begin{aligned}
 \rho_n&\to\rho &&\text{strongly in }L^1(Q_T),\\
 \sqrt{\rho_n}&\to\sqrt\rho &&\text{strongly in }L^2(Q_T),\\
 \nabla\sqrt{\rho_n}&\weakto\nabla\sqrt\rho
 &&\text{weakly in }L^2(Q_T).
\end{aligned}
\end{equation*}
Here the first convergence implies the second, while
Lemmas~\ref{lem:physical-flux-lower-semicontinuity} and
\ref{lem:entropy-tail-transfer} pass the physical estimates.  The Coulomb
flux and its localized renormalizations pass by
the preceding Lemma~\ref{lem:interaction-compactness} and
\eqref{eq:renormalized-coulomb-compactness}.  We pass directly in the
stochastic kinetic identity on compact velocity sets.  Local
$C^1$ convergence of the coefficients allows us to pass to the limit in
the localized It\^o correction, which is retained at fixed noise strength.  The
kinetic-measure and initial-trace identifications, and the identification
of the martingales through their quadratic variations and covariations
with the Brownian motions, follow
\cite[the proof of Theorem~5.29]{FG24} and
\cite[Proposition~3.5 and the proof of Theorem~3.1]{JSW26}.
No classical weak formulation or deterministic weak--kinetic equivalence
is used in this fixed-noise construction; in particular, no unlocalized
distribution $(\varepsilon F_{1,K}/8)\Delta\log\rho$ is asserted at vacuum.

After unforced pathwise uniqueness and measurable factorization,
Section~\ref{sec:comparison} constructs solutions for bounded predictable
controls by Girsanov.  Weakly converging controls are treated separately
in Section~\ref{sec:full-ldp}: there the noise and its correction vanish
before the classical weak formulation and deterministic weak--kinetic
equivalence identify the skeleton limit.

\subsubsection{Velocity-tail identification}

We finish by verifying the velocity-tail conditions for the selected
kinetic measure.  The first part of
Lemma~\ref{lem:canonical-energy-localization} gives the stopped
Fisher-information and flux bounds in
\eqref{eq:stopped-flux-integrability} without using the kinetic measure or
its velocity tails.  These bounds imply the high-velocity condition in
Proposition~\ref{prop:approximation-high-velocity-tail}.  The low-velocity
condition then follows from the second part of
Lemma~\ref{lem:canonical-energy-localization}.

\begin{lemma}[Canonical energy localization]
\label{lem:canonical-energy-localization}
Let $g$ be a predictable control with almost surely finite $L^2(Q_T)$
action.  Let $\rho$ be a nonnegative, almost surely continuous
$L^1(\T^2)$-valued predictable process which conserves the mass $M$ and
satisfies \eqref{eq:stochastic-finite-energy}, and define $\sigma_m$ by
\eqref{eq:kinetic-energy-stops}.  Then $\sigma_m\uparrow T$ almost surely and
\begin{equation}\label{eq:pathwise-fisher-from-energy}
 \int_0^T\mathcal I(\rho(t))\,\dd t<\infty
 \qquad\Pp\text{-a.s.}
\end{equation}
\begin{equation}\label{eq:localizing-fisher-bound}
 \int_0^{\sigma_m}\mathcal I(\rho(t))\,\dd t
 \lesssim_{m,T,V,M}1
 \qquad\Pp\text{-a.s.}
\end{equation}
Consequently, for every $\ell\geq1$,
\begin{equation}\label{eq:stopped-compact-velocity-sobolev}
 \mathbf1_{[0,\sigma_m]}\bigl[(\rho\wedge\ell)\vee\ell^{-1}\bigr]
 \in L^2\bigl(\Omega\times(0,T);H^1(\T^2)\bigr),
\end{equation}
and
\begin{equation}\label{eq:stopped-flux-integrability}
 \E\int_0^{\sigma_m}\!\int_{\T^2}
 \bigl(|\rho(V*\rho)|+|\sqrt\rho\,P_Kg|\bigr)\,\dd x\,\dd t<\infty.
\end{equation}
Suppose, in addition, that $q$ satisfies the kinetic-measure clauses of
Definition~\ref{def:stochastic-kinetic-solution}, namely the parabolic lower bound
\eqref{eq:stochastic-parabolic-lower-bound}, the compact-velocity identity
\eqref{eq:controlled-stochastic-kinetic-identity}, and the stopped
high-velocity condition \eqref{eq:stopped-high-velocity-tail}.  Then the
low-velocity tail follows on the same stops:
\begin{equation}\label{eq:stopped-low-velocity-tail}
 \lim_{\beta\downarrow0}\E\left[
  \beta^{-1}q\bigl(\T^2\times[\beta/2,\beta]
  \times[0,\sigma_m]\bigr)\right]=0.
\end{equation}
In particular, the corresponding almost-sure $\liminf$ is zero.  If the
high-velocity condition and all the preceding integrability statements hold
on $[0,T]$ in expectation, then \eqref{eq:stopped-low-velocity-tail} also
holds with $\sigma_m$ replaced by $T$.
The conclusions through \eqref{eq:stopped-flux-integrability} do not use a
kinetic measure.  In particular, they apply to a candidate approximation
limit before its high-velocity admissibility has been verified.
\end{lemma}

\begin{proof}
The physical-energy condition \eqref{eq:stochastic-finite-energy} and the
almost-sure finiteness of the control action give
$\sigma_m\uparrow T$ almost surely.  Integrating
\eqref{eq:full-fisher-physical-dissipation} gives
\eqref{eq:pathwise-fisher-from-energy}.  Before $\sigma_m$, the entropy and
the accumulated physical dissipation are bounded by $m$;
\eqref{eq:full-fisher-physical-dissipation} therefore gives
\eqref{eq:localizing-fisher-bound}.  A possible value at
the single stopping time does not affect the space--time integral.  Since
\begin{equation*}
 \left|\nabla\bigl[(\rho\wedge\ell)\vee\ell^{-1}\bigr]\right|^2
 \leq\ell\,\frac{|\nabla\rho|^2}{\rho},
\end{equation*}
the stopped Fisher bound proves
\eqref{eq:stopped-compact-velocity-sobolev}.  The same bound,
Lemma~\ref{lem:mass-fisher-consequences}, the $L^2$ contraction of $P_K$,
and the control part of \eqref{eq:kinetic-energy-stops} prove
\eqref{eq:stopped-flux-integrability} and the local integrability of every
compact-velocity term in
\eqref{eq:controlled-stochastic-kinetic-identity}.

For the low-velocity estimate, adapt the proof of
\cite[Proposition~4.6]{FG24}, using $\varphi_\beta$ and $\zeta_R$ from
\cite[Definition~4.5 and equations~(4.12)--(4.14)]{FG24}.  The additional
Coulomb and control terms are
\begin{equation}\label{eq:low-velocity-added-drifts}
 \begin{aligned}
 &\frac2\beta\int_0^{\sigma_m}\!\int_{\T^2}
  \mathbf1_{\{\beta/2<\rho<\beta\}}
  \rho(V*\rho)\cdot\nabla\rho\,\dd x\,\dd t,\\
 &\frac2\beta\int_0^{\sigma_m}\!\int_{\T^2}
  \mathbf1_{\{\beta/2<\rho<\beta\}}
  \sqrt\rho\,P_Kg\cdot\nabla\rho\,\dd x\,\dd t.
 \end{aligned}
\end{equation}
Since
$0\leq\int_0^r s\mathbf1_{\{\beta/2<s<\beta\}}\,\dd s
\leq3\beta^2/8$, the first line of
\eqref{eq:low-velocity-added-drifts} equals
\begin{equation*}
 \frac{2\lambda_V\cos(\alpha_V)}\beta
 \int_0^{\sigma_m}\!\int_{\T^2}
 \left(\int_0^\rho
 s\mathbf1_{\{\beta/2<s<\beta\}}\,\dd s\right)
 (\rho-M)\,\dd x\,\dd t,
\end{equation*}
because its rotational part again vanishes.  Its absolute value is
$\lesssim_{V,M,T}\beta$.  The parabolic lower bound
\eqref{eq:stochastic-parabolic-lower-bound} and Young's inequality
give for the second line
\begin{equation*}
 \begin{aligned}
 &\left|\frac2\beta\int_0^{\sigma_m}\!\int_{\T^2}
  \mathbf1_{\{\beta/2<\rho<\beta\}}
  \sqrt\rho\,P_Kg\cdot\nabla\rho\,\dd x\,\dd t\right|\\
 &\quad\leq\frac1{2\beta}
 q\bigl(\T^2\times[\beta/2,\beta]\times[0,\sigma_m]\bigr)
 +\frac2\beta\int_0^{\sigma_m}\!\int_{\T^2}
 \mathbf1_{\{\beta/2<\rho<\beta\}}\rho|P_Kg|^2\,\dd x\,\dd t.
 \end{aligned}
\end{equation*}
The last integral tends to zero in expectation by dominated convergence
and the stopped control-action bound.  After absorption, the endpoint
and exact-square-root terms vanish as in
\cite[equations~(4.13)--(4.14)]{FG24}.  It remains only to remove the upper
cutoff.  The kinetic-measure contribution vanishes by
\eqref{eq:stopped-high-velocity-tail}.  The radial Coulomb contribution is
bounded by a high-density tail of $\rho^2+M\rho$, which converges to zero in
expectation by \eqref{eq:localizing-fisher-bound} and
\eqref{eq:mass-fisher-L2}; its rotational part vanishes.  Young's inequality,
the finite-mode estimate $\|P_Kg(t)\|_\infty^2\lesssim_K\|g(t)\|_2^2$,
the stopped control action, and
\begin{equation*}
 (R+2)|\{\rho(t)>R-1\}|
 \lesssim\frac{1}{\log R}
 \left(1+\int_{\T^2}H(\rho(t))\right)
\end{equation*}
dispose of the control term.  The trace and endpoint terms vanish by the
same entropy tail.  This proves \eqref{eq:stopped-low-velocity-tail}.
With global versions of the hypotheses, the identical argument applies on
$[0,T]$.
\end{proof}

\begin{proposition}[High-velocity admissibility of approximation limits]
\label{prop:approximation-high-velocity-tail}
Let $(\rho,q)$ be a limit pair produced by the compatible approximation
scheme of Appendix~\ref{app:approximation} for a predictable control with
almost surely finite action, and suppose that it satisfies
the mass, physical-energy, parabolic lower-bound, and compact-velocity
identity clauses of Definition~\ref{def:stochastic-kinetic-solution}.
Then it also satisfies \eqref{eq:stopped-high-velocity-tail}, and hence is
a stochastic kinetic solution.  If, in addition,
\eqref{eq:bounded-control-action} holds,
\begin{equation*}
 \E\left[\sup_{t\leq T}\int_{\T^2}H(\rho(t))
 +\int_0^T\mathcal D_V(\rho(t))\,\dd t\right]<\infty,
\end{equation*}
and either
$\lambda_V\cos(\alpha_V)\leq0$ or \eqref{eq:expected-full-fisher} holds,
then
\begin{equation}\label{eq:global-high-velocity-tail}
 \lim_{R\to\infty}
 \E q\bigl(\T^2\times[R,R+1]\times[0,T]\bigr)=0.
\end{equation}
This is the high-velocity condition in
\cite[Definition~3.4, equation~(3.7)]{FG24}.
\end{proposition}

\begin{proof}
For $R>2$, let $S_R$ be a smooth convex approximation of the shell
renormalization used in
\cite[the proof of Proposition~5.9, equation~(5.15)]{FG24}, chosen so that
$S_R(0)=S_R'(0)=0$, $0\leq S_R''\leq1$, $S_R''=1$ on $[R,R+1]$, and
$\operatorname{supp}S_R''\Subset(R-1,R+2)$, and write locally
\begin{equation*}
 A_R(r)=\int_0^r sS_R''(s)\,\dd s,
 \qquad
 0\leq A_R(r)\lesssim(R+1)\mathbf1_{\{r>R-1\}}.
\end{equation*}
At fixed $R$, first pass the $S_R$-identity from the compatible smooth
approximations to $(\rho,q)$, before introducing the limiting stop.  The
endpoint passes because $S_R$ is Lipschitz and the densities converge
strongly in $L^1$; the kinetic-measure term passes by local vague
convergence because $S_R''$ is smooth and compactly supported in velocity.
The bounded primitive $A_R$ and localized gradient convergence pass the
Coulomb and control terms.  This gives the fixed-$R$ shell identity for the
selected limiting measure $q$.  Before taking expectations, we stop further
when either the bracket of the fixed-$R$ martingale or the $q$-mass on
$\operatorname{supp}S_R''$ reaches a finite level.  The Fisher and control
bounds on $[0,\sigma_m]$ make the stopped martingale square integrable;
monotone convergence and Fatou's lemma then remove this auxiliary stop.
Thus optional sampling at $\sigma_m$ is legitimate without assuming in
advance that the target shell has finite expectation.

The diffusion and noise terms, including the exact-square-root gradient
cancellation, are treated as in the cited proof.  The additional Coulomb
terms are
\begin{equation}\label{eq:high-velocity-coulomb-polarization}
 \begin{aligned}
 &\lambda_V\cos(\alpha_V)
 \int_0^{\sigma_m}\!\int_{\T^2}
 \nabla A_R(\rho)\cdot\nabla(\Green*\rho)\,\dd x\,\dd t
 =\lambda_V\cos(\alpha_V)
 \int_0^{\sigma_m}\!\int_{\T^2}A_R(\rho)(\rho-M)\,\dd x\,\dd t,\\
 &\lambda_V\sin(\alpha_V)
 \int_0^{\sigma_m}\!\int_{\T^2}
 \nabla A_R(\rho)\cdot\J\nabla(\Green*\rho)\,\dd x\,\dd t=0.
 \end{aligned}
\end{equation}
The first part of Lemma~\ref{lem:canonical-energy-localization}, which does
not use the high-velocity condition, and \eqref{eq:mass-fisher-L2} give
\begin{equation*}
 \E\int_0^{\sigma_m}\|\rho(t)\|_2^2\,\dd t
 \lesssim_{m,T,V,M}1.
\end{equation*}
Since
\begin{equation*}
 A_R(r)|r-M|\lesssim(r^2+Mr)\mathbf1_{\{r>R-1\}},
\end{equation*}
the radial term in \eqref{eq:high-velocity-coulomb-polarization} tends to
zero in expectation.  The control contribution satisfies
\begin{equation}\label{eq:high-velocity-control}
 \begin{aligned}
 &\left|\int_0^{\sigma_m}\!\int_{\T^2}
 S_R''(\rho)\sqrt\rho\,P_Kg\cdot\nabla\rho\,\dd x\,\dd t\right|\\
 &\quad\leq\frac14
 \int_{\T^2\times(0,\infty)\times[0,\sigma_m]}
 S_R''(\xi)\,\dd q
 +\frac{C_{K}m(1+m)}{\log R},
 \end{aligned}
\end{equation}
where we used the parabolic lower bound, Young's inequality, the
finite-mode $L^\infty$ estimate, and the entropy tail displayed in the
proof of Lemma~\ref{lem:canonical-energy-localization}.  The remaining
noise trace tends to zero by the same entropy tail.  Dropping the
nonnegative terminal value and absorbing the first term on the right of
\eqref{eq:high-velocity-control} therefore yields
\begin{equation*}
 \begin{aligned}
 \frac12\E\int_{\T^2\times(0,\infty)\times[0,\sigma_m]}
 S_R''(\xi)\,\dd q
 \lesssim{}&\int_{\T^2}(\rho_0-(R-1))_+\,\dd x
 +\frac{C_{m,\varepsilon,K}}{\log R}\\
 &+|\lambda_V\cos(\alpha_V)|
 \E\int_0^{\sigma_m}\!\int_{\T^2}
 A_R(\rho)|\rho-M|\,\dd x\,\dd t.
 \end{aligned}
\end{equation*}
Every term tends to zero, and $S_R''=1$ on the closed target slab.  This
proves \eqref{eq:stopped-high-velocity-tail} without assuming that $q$ has
no atoms in velocity.

For the global assertion, apply the same identity on $[0,T]$.  If
$\lambda_V\cos(\alpha_V)\leq0$, the radial contribution is nonpositive,
because $A_R$ is nondecreasing and
\begin{equation*}
 \int_{\T^2}A_R(\rho)(\rho-M)
 =\frac12\iint_{\T^2\times\T^2}
 \bigl(A_R(\rho(x))-A_R(\rho(y))\bigr)
 \bigl(\rho(x)-\rho(y)\bigr)\,\dd x\,\dd y\geq0.
\end{equation*}
In this case $\mathcal D_V(\rho)\geq\mathcal I(\rho)$, so the assumed
physical-energy estimate also supplies the global local-Sobolev and flux
integrability used in the shell identity.
If \eqref{eq:expected-full-fisher} holds, its absolute value instead tends
to zero by \eqref{eq:mass-fisher-L2} and dominated convergence.  The
global entropy and control bounds treat all remaining terms, proving
\eqref{eq:global-high-velocity-tail}.
\end{proof}

\section{Proof of Osgood's lemma}
\label{app:osgood}

\begin{proof}
The function
$r\mapsto r\log(e(1+2M)/r)$ is increasing on $[0,2M]$ and
\begin{equation*}
 \int_{0^+}\frac{\dd r}{r\log(e(1+2M)/r)}=\infty.
\end{equation*}
Apply the chain rule to the absolutely continuous right-hand side of
\eqref{eq:abstract-osgood-inequality}, stopped when it reaches $2M$.
For $\delta(0)>0$, integration of
\begin{equation*}
 \frac{1}{r\log(e(1+2M)/r)}
\end{equation*}
gives \eqref{eq:osgood-continuous-dependence}; after the stopping time its
right-hand side only decreases, while $\delta\leq2M$.  For $\delta(0)=0$,
first replace $\delta(0)$ on the right-hand side of
\eqref{eq:abstract-osgood-inequality} by an arbitrary positive number and
then let it decrease to zero.  The divergent integral forces $\delta=0$.
\end{proof}

\section*{Use of AI tools}
The author used GPT to assist with the writing and revision of this
paper.  GPT identified the need for a stopping-time argument in the
stochastic free-energy estimates.  The resulting proof was developed
through discussions with GPT.  The author takes full
responsibility for the mathematical arguments, references, and final text.

\begingroup
\hfuzz=2pt
\hbadness=5000
\renewcommand{\bibliofont}{\fontsize{9}{11.4}\selectfont}
\bibliographystyle{abbrvurl}
\bibliography{dk_coulomb}
\endgroup

\end{document}